\documentclass{amsart}
\usepackage{fullpage}
\usepackage{amsmath, amssymb}
\usepackage{amsthm}
\usepackage{amscd}
\usepackage{color}
\usepackage{tikz}
\usepackage[all,cmtip]{xy}
\usepackage{xcolor}
\usepackage{pdfcomment}
\usepackage{bbm}
\usepackage[all,cmtip]{xy}
\usepackage{appendix}
\usepackage{mathtools}
\usepackage{ytableau}
\usepackage{youngtab}
\usepackage[normalem]{ulem}

\newtheorem{theorem}{Theorem}[section]
\newtheorem{proposition}[theorem]{Proposition}
\newtheorem{corollary}[theorem]{Corollary}
\newtheorem{lemma}[theorem]{Lemma}
\newtheorem{definition}[theorem]{Definition}

\newtheorem{remark}[theorem]{Remark}
\newtheorem{example}{Example}
\newtheorem{notation}{Notation}

\newcommand{\Hom}{\mathrm{Hom}}

\newcommand{\mCl}{\mathsf{mCl}}

\newcommand{\Proj}{\mathrm{Proj}}
\newcommand{\Gr}{\mathrm{Gr}}
\newcommand{\ring}{\mathbb{C}[\mathrm{Gr}_{m,n}]}
\newcommand{\iring}{\mathbb{C}[\mathrm{Gr}]}

\newcommand{\cA}{\mathcal{A}}

\newcommand{\cC}{\mathcal{C}}

\newcommand{\cF}{\mathcal{F}}

\newcommand{\bC}{\mathbb{C}}

\newcommand{\bE}{\mathbb{E}}

\newcommand{\bN}{\mathbb{N}}

\newcommand{\bQ}{\mathbb{Q}}

\newcommand{\bX}{\mathbb{X}}

\newcommand{\bZ}{\mathbb{Z}}

\newcommand{\X}{\mathrm{X}}
\newcommand{\ex}{\mathrm{ex}}

\makeatletter
\newcommand{\colim@}[2]{%
  \vtop{\m@th\ialign{##\cr
    \hfil$#1\operator@font colim$\hfil\cr
    \noalign{\nointerlineskip\kern1.5\ex@}#2\cr
    \noalign{\nointerlineskip\kern-\ex@}\cr}}%
}
\newcommand{\alim@}[2]{%
  \vtop{\m@th\ialign{##\cr
    \hfil$#1\operator@font lim$\hfil\cr
    \noalign{\nointerlineskip\kern1.5\ex@}#2\cr
    \noalign{\nointerlineskip\kern-\ex@}\cr}}%
}
\newcommand{\colim}{%
  \mathop{\mathpalette\colim@{\rightarrowfill@\scriptscriptstyle}}\nmlimits@
}
\renewcommand{\varprojlim}{%
  \mathop{\mathpalette\varlim@{\leftarrowfill@\scriptscriptstyle}}\nmlimits@
}
\renewcommand{\varinjlim}{%
  \mathop{\mathpalette\varlim@{\rightarrowfill@\scriptscriptstyle}}\nmlimits@
}
\newcommand{\alim}{%
  \mathop{\mathpalette\alim@{\leftarrowfill@\scriptscriptstyle}}\nmlimits@
}
\makeatother

\numberwithin{theorem}{section}
\numberwithin{equation}{section}
\numberwithin{figure}{section}
\numberwithin{example}{section}

\begin{document}

\title[Ind-cluster algebras and infinite Grassmannians]{Ind-cluster algebras and infinite Grassmannians}
%\thanks{Add acknowledgments here.}
\author{Sira Gratz}
\address{
Sira Gratz, Department of Mathematics,
Aarhus University,
Ny Munkegade 118,
8000 Aarhus C, Denmark
}
\email{Sira@math.au.dk}
\urladdr{https://sites.google.com/view/siragratz}
\author{Christian Korff}
\address{
Christian Korff, School of Mathematics and Statistics,
University of Glasgow,
Glasgow G12 8QQ, United Kingdom
}
\email{Christian.Korff@glasgow.ac.uk}
\urladdr{https://sites.google.com/view/christiankorff}

%13F60 Cluster algebras
%37K10 Completely integrable infinite-dimensional Hamiltonian and Lagrangian systems, integration methods, integrability tests, integrable hierarchies (KdV, KP, Toda, etc.)
%14M15 Grassmannians, Schubert varieties, flag manifolds
%14P10  	Semialgebraic sets and related spaces
%05E05  	Symmetric functions and generalizations
%05E14  	Combinatorial aspects of algebraic geometry
%16E45  	Differential graded algebras and applications (associative algebraic aspects) 
%18E30???
%Arxiv Subjects:	Combinatorics (math.CO); High Energy Physics - Theory (hep-th); Mathematical Physics (math-ph); Algebraic Geometry (math.AG); Exactly Solvable and Integrable Systems (nlin.SI)

\begin{abstract}
A prototypical examples of a cluster algebra is the coordinate ring of a finite Grassmannian: using the Pl\"ucker embedding the cluster algebra structure allows one to move between `maximal sets' of algebraically independent Pl\"ucker coordinates via mutations. Fioresi and Hacon studied a specific colimit of the coordinate rings of finite Grassmannians and its link with the infinite Grassmannian introduced by Sato and independently by Segal and Wilson in connection with the Kadomtsev-Petiashvili (KP) hierarchy, an infinite set of nonlinear partial differential equations which possess soliton solutions. In this article we prove that this ring is a cluster algebra of infinite rank with the structure induced by the colimit construction. More generally, we prove that cluster algebras of infinite rank are precisely the ind-objects of a natural category of cluster algebras.
\end{abstract}

\maketitle

\tableofcontents

\section{Introduction}

Infinite-dimensional Grassmann manifolds were introduced by Sato in \cite{sato1981soliton,sato1983soliton} and by Segal, Wilson in \cite{segal1985loop} as well as by Pressley, Segal in \cite{pressley1985loop}. The constructions in \cite{sato1981soliton,sato1983soliton} and \cite{segal1985loop} are motivated by the observation that their points are in one-to-one correspondence with solutions of an integrable hierarchy of non-linear partial differential equations (PDEs), the Kadomtsev-Petiashvili (KP) hierarchy, which describes shallow water waves and possesses soliton solutions. The KP hierarchy plays a central role in the Riemann-Schottky problem and the proof of the Novikov conjecture \cite{shiota1986characterization}. It connects the area of integrable systems with algebraic geometry \cite{krichever2011soliton} but it is also of interest in representation theory of loop groups \cite{pressley1985loop} and in connection with the boson-fermion correspondence as well as vertex operator algebras and symmetric functions, see e.g. \cite{miwa2000solitons,kac2013bombay} as well as references therein. 

Key to understanding the link between the KP-hierarchy and the underlying algebraic and geometric structure of infinite dimensional Grassmannians is the notion of a $\tau$-function introduced by the Kyoto School in the 1990s, which allows one to formulate the PDEs in an algebraic form, known as Hirota's bilinear identity, that is equivalent to the union of the Pl\"ucker relations of all finite Grassmannians; see e.g. the textbook \cite{miwa2000solitons} and references therein. The solutions of the KP-hierarchy can therefore be identified with points on an infinite  Grassmannian (defined independently in \cite{sato1981soliton,sato1983soliton} and \cite{segal1985loop}),  which is sometimes in the literature simply called the `Sato Grassmannian' although the constructions in these sources differ in some technical aspects; we refer the reader to \cite[\S10]{segal1985loop} and \cite{pressley1985loop} for details.

For the sake of concreteness we recall the definition of the infinite Grassmannian $\Gr(H)$ for a given separable Hilbert space $H$ from \cite[\S2,\S10]{segal1985loop} and \cite[Def.7.1]{pressley1985loop}: let $H=L^2(S^1)$ be the Hilbert space of square integrable functions $f:S^1\to\bC$ on the unit circle $S^1\subset\bC$, which we can write as a Fourier series in the complex variable $z\in S^1$ as
\[
f(z)=\sum_{n\in\bZ} f_n z^{n+1}\;.
\]
We call $f\in H$ of {\em finite order} $N$ if it is of the form $f(z)=\sum_{N\le n}f_nz^{n+1}$ with $f_{N}\neq 0$. That is, we can think of the subspace $H^{\rm fin}$ of finite order elements as the space $V=\bC(\!(z)\!)$ of formal Laurent series in the variable $z$ and the latter is dense in $H$. Consider the decomposition\footnote{We have made slight modifications in the conventions when defining the subspaces $H_\varnothing$ and $H_0$ compared to \cite{segal1985loop} in order to more easily match the discussion of the coordinate rings of Grassmannians in \cite{FH-Sato}.} $H=H_{\varnothing}\oplus H_{0}$, where $H_0\cong z\bC[\![z]\!]$ is the subspace of order zero elements and $H_\varnothing$ the subspace of elements $f$ which do not contain any positive powers of $z$, i.e. $f$ is of the form $f(z)=f_0+f_{-1}z^{-1}+f_{-2}z^{-2}+\cdots$. 
\begin{definition}[Sato \cite{sato1981soliton,sato1983soliton}, Segal-Wilson \cite{segal1985loop}, Pressley-Segal \cite{pressley1985loop}]\label{def:Gr}
    Denote by $\Gr(H)$ the set of all closed subspaces $W\subset H$ such that (i) the orthogonal projection $\mathrm{pr}:W\to H_\varnothing$ has finite-dimensional kernel and cokernel and (ii) the orthogonal projection $\mathrm{pr}:W\to H_0$ is a Hilbert-Schmidt operator.
\end{definition}
As explained in \cite{segal1985loop} the infinite Grassmannian $\Gr(H)$ is not connected and two subspaces belong to the same connected component if and only if they have the same {\em virtual dimension}, where the latter is defined as the difference of the dimensions of the kernel and cokernel of the projection onto $H_\varnothing$. In what follows we shall only consider the connected component of virtual dimension zero, i.e. the set of subspaces $W\in\Gr(H)$ for which $\dim(\ker\mathrm{pr})=\dim(\mathrm{coker}\,\mathrm{pr})$ with $\mathrm{pr}:W\to H_\varnothing$ being the orthogonal projection from (i) in Definition \ref{def:Gr}. Up to certain technical differences, it is this connected component which has been called the `Universal Grassmann Manifold' in \cite{sato1983soliton} and is simply denoted by $\Gr$ in \cite{segal1985loop}. We shall adopt the latter notation and refer to it as the Sato-Segal-Wilson Grassmannian.

\begin{remark}\rm
Although we have included the definition of the infinite Grassmannian $\Gr$ here for completeness, our main focus will be its coordinate ring $\iring$ instead, which is defined independently as a colimit $\iring = \colim \ring$ of the coordinate rings $\ring$ of finite Grassmannians $\Gr_{m,n}=\Gr(m;\bC^{m+n})$ as discussed in \cite{FH-Sato}; compare also with the definition of $\widetilde{\mathrm{GM}}$ in \cite[pp261]{sato1983soliton}.
\end{remark}

\subsection{Pl\"ucker relations and cluster algebras}
There is a one-to-one correspondence between the Pl\"ucker relations defining the coordinate ring $\iring$ and the PDEs making up the KP hierarchy \cite{sato1983soliton}. However, as is already known from the finite case, the set of Pl\"ucker coordinates is algebraically dependent. This is where cluster algebras come in: their clusters consist of maximal sets of algebraically independent Pl\"ucker coordinates from which, by a purely combinatorial procedure called mutation, all other Pl\"ucker coordinates can be obtained. For the coordinate rings $\ring$ of finite Grassmannians $\Gr_{m,n}$ the cluster algebra structure was worked out by Scott \cite{Scott}. In particular, Scott's work describes clusters of Pl\"ucker coordinates and their mutations in explicit combinatorial terms using Postnikov diagrams.

    To capture the algebraic-combinatorial structure of the coordinate ring $\iring$ (as introduced in \cite{sato1981soliton,sato1983soliton} and studied by Fioresi and Hacon in \cite{FH-Sato}), finite-rank cluster algebras will not do. In this article the focus is instead on infinite rank cluster algebras whose study was initiated in \cite{GG1}  inspired by work on infinite rank cluster categories \cite{HJ-cat}, and has since been extended and applied in the context of combinatorial models \cite{BG}, Teichm\"uller theory and cluster algebras from infinitely marked surfaces \cite{CF19}, quantum affine algebras \cite{HL-KR-modules}, \cite{HL-Borel}, \cite{GHL-simplylaced} as well as from a representation theoretic perspective notably with the introduction of discrete cluster catgories of type $A$ in \cite{ITcyclic}, as well as their combinatorial \cite{PY21} and categorical \cite{CG24} completions, and Grassmannian cluster categories of infinite rank in \cite{ACFGS}, inspired by the construction from \cite{JKS16}. These categories have been studied from a cluster combinatorial perspective for example in \cite{BKM16}, \cite{GHJ}, \cite{PY21}, \cite{ACFGSII} and \cite{CKP-completed}.

\subsection{Main result: Ind-Cluster Algebras}
In this paper, we tackle infinite rank versions of cluster algebras by studying ind-objects in a category of finite rank cluster algebras. The approach to study infinite rank cluster algebras as ind-objects was first employed in \cite{Gcolimits}, where it was shown that every cluster algebra of infinite rank arises as a colimit of cluster algebras of finite rank in the category $\mCl$ of {\em inducible melting cluster morphisms}, introduced, with slightly different conventions, in \cite{ADSClus}. Its objects are rooted cluster algebras, that is, cluster algebras with a fixed initial seed, with clusters of arbitrary cardinality, and its maps are morphisms which commute with mutation starting at the initial cluster. The freedom we have is that we are allowed sending cluster variables to integers (``specialising''), as well as sending frozen cluster variables to exchangeable cluster variables (``melting''), but not vice versa. We provide a converse to \cite[Theorem~1.1]{Gcolimits}.

\begin{theorem}[Theorem~\ref{T:filteredclosed}]\label{T:ind-objects}
	The category $\mCl$ of inducible melting cluster morphisms is closed under filtered colimits.
\end{theorem}

We show that the compact objects of $\mCl$ are precisely the rooted cluster algebras of finite rank. In particular, if we stick with the common convention that stipulates that a cluster algebra has initial cluster of finite cardinality, it makes sense to initially consider the full category $\mCl^{\mathrm{f}}$ of $\mCl$ whose objects are finite rank rooted cluster algebras. We then obtain

\begin{corollary}
    The rooted cluster algebras of infinite rank are the ind-objects of $\mCl^{\mathrm{f}}$.
\end{corollary}

We are thus motivated to call cluster algebras of infinite rank {\em ind-cluster algebras}, to distinguish them from classical cluster algebras. That is, an ind-cluster algebra is just an ind-object of $\mCl^{\mathrm{f}}$, where we forget the datum of the initial seed.

\begin{remark}\rm While cluster algebras of infinite rank successfully deal with the situation where we encounter infinite clusters---a situation that naturally arises both from representation theoretic and combinatorial perspectives---it fails to take into account the possibility of infinite mutations, which were successfully studied from combinatorial, and representation theoretic perspectives in \cite{BG}, \cite{CF19} and \cite{ACFGS}. We close this gap in a forthcoming paper \cite{gratz2025}, by also introducing {\em pro-cluster algebras}; the idea of which was first introduced in \cite{Wierzbicki}.
\end{remark}

\subsection{Application: the coordinate ring of the Sato-Segal-Wilson Grassmannian}
We approach the cluster algebra structure of the coordinate ring $\iring$ from this perspective of ind-cluster algebras: 
as already mentioned, the ring $\iring$ studied in \cite{FH-Sato} is defined as a colimit $\iring = \colim \ring$ of the homogeneous coordinate rings of finite dimensional Grassmannians which carry a classical cluster algebra structure, as proved in \cite{Scott}. It is this natural cluster algebra structure which induces the ind-cluster algebra on the coordinate ring $\iring$.

 It has been argued in \cite[Prop 2.10 (ii)]{FH-Sato} that $\iring$ can be viewed as a {\em coordinate ring of $\Gr$}, in the sense that the closed points of $\Proj\; \iring$ are in one-to-one correspondence with the points of the `universal Grassmann manifold' postulated by Sato and which is closely related to the one considered by Wilson and Segal \cite{segal1985loop}. Using Theorem \ref{T:ind-objects}, we show the following:
\begin{theorem}
The ring $\iring$ is a cluster algebra of infinite rank. Its cluster structure is induced by the classical cluster structures on homogeneous coordinate rings of finite dimensional Grassmannians. 
\end{theorem}

Using Sato's results on the connection between points on the infinite Grassmannian $\Gr$ and solutions of the KP-hierarchy, it then follows that the general solutions, $\tau$-functions, of the KP-hierarchy carry an ind-cluster algebra structure. Using the ind-cluster algebra we will prove the following:

\begin{theorem}
    Let $W\in\Gr$ be a point on the Sato-Segal-Wilson Grassmannian. Then the associated Pl\"ucker coordinates $\Delta_\lambda(W)$, where $\lambda$ ranges over all partitions, are Laurent polynomials with non-negative coefficients of those coordinates for which $\lambda$ has a Young diagram of rectangular shape. 
\end{theorem}
We will recall the precise definition of the coordinates $\Delta_\lambda(W)$ later in Section \ref{S:tau functions}.

\begin{remark}\rm 
Various types of $\tau$-functions have been studied extensively in the literature. Perhaps most famously, multi-soliton solutions of the KP-hierarchy have been constructed (see e.g. \cite{miwa2000solitons}, \cite{kac2013bombay} as well as references therein) and each such solution corresponds to a point $W$ on the Sato-Segal-Wilson Grassmannian $\Gr$. Other celebrated examples include $\tau$-functions obtained from algebraic curves; see e.g. \cite{krichever1977methods,krichever1977integration,dubrovin1981theta,harnad2011schur} and \cite{kodama2024kp}. Another point of view is to interpret $\tau$-functions of integrable hierarchies, and the ones of the KP-hierarchy in particular, as generating functions for objects in enumerative geometry; see e.g. the review \cite{kazarian2015combinatorial} as well as references therein. 

More recently, Kodama and Williams \cite{kodama2011kp,kodama2014kp} connected a special class of soliton solutions of the KP-equation to the combinatorics of plabic graphs which describe the cluster algebra structure for the coordinate ring of {\em finite} Grassmannians. These data of the finite Grassmannians describe the asymptotics of these soliton solutions and are not connected to the Pl\"ucker coordinates of the Sato-Segal-Wilson Grassmannian which are the focus of this article; see the comments in Section \ref{S:multi-soliton} at the end of this article.
\end{remark}

There is a growing literature and research activity on `positive geometries' emerging from physics and related mathematics and combinatorics; see e.g. \cite{lam2024invitation,ranestad2025positive} for some recent overview articles as well as references therein. On the mathematics side the probably best studied examples are related to Lusztig's notion of {\em total positivity} \cite{lusztig1994total,lusztig1998introduction} in the case of (finite-dimensional) Grassmannians; for a small sample of works see e.g. \cite{speyer2005tropical,Postnikov,kodama2014kp,williams2021positive}. Inspired by the connection between total positivity of finite Grassmannians and KP solitons \cite{kodama2014kp} we generalise in this article the definition of total positivity to the infinite Sato-Segal-Wilson Grassmannian: 
\begin{definition}
    Define the {\em totally positive Sato-Segal-Wilson Grassmannian} $\Gr^+$ to be the set of all points $W\in\Gr$ such that their Pl\"ucker coordinates $\Delta_\lambda(W)>0$ for all partitions $\lambda$.
\end{definition}

As an immediate corollary of the ind-cluster algebra structure on $\iring$ and the above theorem we can then conclude the following:
\begin{corollary}
    We have $W\in\Gr^+$ if and only if $\Delta_\lambda(W)>0$ for all {\em rectangular} partitions $\lambda$. 
\end{corollary}

It would be interesting as a future application to further investigate positivity aspects of the known stratification of the Sato-Segal-Wilson Grassmannian \cite{segal1985loop,pressley1985loop} and solutions of the KP-hierarchy by making use of the ind-cluster algebra structure similar as finite rank cluster algebras are being used to discuss positivity for finite Grassmannians; see e.g. \cite{williams2021positive}.

\bigskip
\noindent{\bf Acknowledgments.} It is a pleasure to thank Dinakar Muthiah, Greg Stevenson and Christian Voigt for several insightful discussions and their generous sharing of knowledge. C.K. gratefully acknowledges the hospitality of Aarhus University during a research stay in February 2023 were part of the research was carried out. S.G. was supported by VILLUM FONDEN (Grant Number VIL42076) and
EPSRC (Grant Number EP/V038672/1). C.K. acknowledges the financial support of EPSRC (Grant Number EP/V053728/1).
%%%%%%%%%%%%%%%%%%%%%%%%%%%%%%%%%%%%%%%%%%%%

\section{Ind-cluster algebras}
\subsection{Rooted cluster algebras}

A holistic definition of a structure preserving map between cluster algebras is hard to come by: The recursive combinatorial definition of a cluster structure on an algebra, which contributes to its computational attractiveness, also makes it difficult to pin down what a map preserving this structure should look like.

A version of a category of pointed cluster algebras has been proposed by Assem, Dupont and Schiffler \cite{ADSClus}. They introduce the category $\sf{Clus}$ of rooted cluster algebras. Its object are {\em rooted cluster algebras}, that is, cluster algebras with a fixed initial seed. This allows for a convenient definition of what it means for a ring homomorphism to commute with mutation. Before we treat the morphisms, we give a brief introduction to (rooted) cluster algebras. We follow the convention and notation from\cite{ADSClus}; the definitions in this section go, unless otherwise stated, back to Fomin and Zelevinsky \cite{FZI}. While for the reader encountering cluster algebras for the first time, the definitions might appear at first quite technical, they will soon become more natural when observed in concrete examples.

We explicitly include seeds of infinite rank, as introduced in \cite{GG1}. For this purpose, we consider {\em infinite integer matrices}: Specifically, given a set $\X$ of arbitrary cardinality, and a subset $\ex \subseteq \X$, we consider tuples of the form $B = (b_{vw})_{v \in \ex, w \in \X}$ from $\bZ$, and call this an integer matrix with rows labelled by $\ex$ and columns labelled by $\X$. The submatrix $B \mid_\ex = (b_{vw})_{v,w \in \ex}$ is called the {\em $\ex$-part of $B$}. The matrix $B$ is called {\em $\ex$-locally finite}, if  every row has only finitely many non-zero entries, that is, for all $v \in \ex$ there are only finitely many $w \in \X$ such that $b_{vw} \neq 0$. It is called {\em skew-symmetriseable} if the $\ex$-part of $B$ is skew-symmetriseable, that is, there exist positive integers $\{d_v \in \bZ_{>0}\}_{v \in \ex}$ such that for all $v,w \in \ex$ we have $d_vb_{vw} = -b_{wv}d_w$.

\begin{definition}\label{D:seed}
	A {\em seed} is a triple $\Sigma = (\X,\ex,B)$, where $\X$ is a set of indeterminates, $\ex \subseteq \X$ and $B = (b_{vw})_{v \in \ex,w\in \X}$ is an $\ex$-locally finite skew-symmetriseable integer matrix.
		
		The set $\X$ is called the {\em cluster}, the elements of the set $\ex$ the {\em exchangeable variables}, the elements of the set $\X \setminus \ex$ the {\em frozen variables}, and the matrix $B$ the {\em exchange matrix of $\Sigma$}. The cardinality $|\X|$ is the {\em rank of $\Sigma$}. The $\ex$-part of $B$ is called the {\em principal part of $B$}.
\end{definition}

\begin{remark}\rm
	Conventionally, the cluster $\X$ is assumed to be finite. We explicitly allow the infinite rank case. The only concession we make is the $\ex$-local finiteness condition on the exchange matrix $B$. This will ensure that the exchange relations (\ref{eqn:exchange relation}) from Definition \ref{D:mutation} are well-defined.
\end{remark}

\begin{remark}\label{R:quiver}\rm
	If the principal part of the exchange matrix $B$ of a seed $\Sigma = (\X,\ex,B)$ is skew-symmetric, that is, if for all $v,w \in \ex$ we have $b_{vw} = - b_{wv}$, we can consider the {\em associated ice quiver} instead: It is the quiver with vertices given by $\X$ and with $b_{uv}$ arrows from vertex $u$ to vertex $v$ whenever $b_{uv}>0$ and $-b_{uv}$ arrows from $u$ to $v$ whenever $b_{uv}<0$. The vertices corresponding to frozen cluster variables are marked as frozen, symbollically this is done by drawing a box around the vertex. When working in concrete examples, it will sometimes be useful to allow arrows between frozen variables (cf. Section \ref{S:combinatorial_description}). These are of no importance for the seeds, and are not recorded in the exchange matrix.
The ice quiver associated with a seed encodes all the information of the seed. In examples of seeds, we will often present the ice quiver in lieu of the triple of cluster, exchangeable variables, and exchange matrix. For example, the seed 
\[
	\Sigma = (\{x_1, x_2, x_3, x_4\}, \{x_1,x_2\}, B=\begin{bmatrix} 0 & 1 & 1 & 1 \\ -1 & 0 & -2 & 1 
    %\\ -1 & -2 & 0 & -1 \\ -1 & 0 & 1 & 0
    \end{bmatrix})
\]
can be represented by either of the two ice quivers
\begin{equation*}
\begin{tikzpicture}
	\node (1) at (0,0) {$x_1$};
	\node[draw] (2) at (1,1) {$x_3$};
	\node (3) at (1,-1) {$x_2$};
	\node[draw] (4) at (-1,-1) {$x_4$};
	
	\draw[->] (1) to (2);
	\draw[->] (1) to (3);
	\draw[->] (1) to (4);
	
	\draw[->, bend left] (2) to (3);
	\draw[->, bend right] (2) to (3);
	
	\draw[->] (3) to (4);
\end{tikzpicture}
\qquad\text{ or }\qquad 
\begin{tikzpicture}
	\node (1) at (0,0) {$x_1$};
	\node[draw] (2) at (1,1) {$x_3$};
	\node (3) at (1,-1) {$x_2$};
	\node[draw] (4) at (-1,-1) {$x_4$};
	
	\draw[->] (1) to (2);
	\draw[->] (1) to (3);
	\draw[->] (1) to (4);
	
	\draw[->, bend left] (2) to (3);
	\draw[->, bend right] (2) to (3);

        \draw[->, bend right] (2) to (4);
	
	\draw[->] (3) to (4);
\end{tikzpicture}.
\end{equation*}
Note that the requirement for a skew-symmetric exchange matrix to be $\ex$-locally finite translates to each exchangeable vertex in the associated ice quiver having finite valency.
\end{remark}

\begin{definition}\label{D:mutation}
	Let $\Sigma = (\X,\ex,B = (b_{uv}))$ be a seed, and let $x \in \ex$. The {\em mutation $\mu_x(\Sigma)$ of $\Sigma$ at $x$} is the seed 
	\[
		\mu_x(\Sigma) = (\mu_x(\X),\mu_x(\ex),\mu_x(B)),
	\]
where, setting
\begin{eqnarray}\label{eqn:exchange relation}
	x' = \frac{\prod_{b_{xu}>0}v^{b_{xu}} + \prod_{b_{xv}<0}v^{-b_{xv}}}{x} \in \bQ(\X)
\end{eqnarray}
we have
\begin{itemize}
	\item $\mu_x(X) = (X \setminus \{x\}) \cup x'$;
	\item $\mu_x(\ex) = (\ex \setminus \{x\}) \cup x'$;
	\item $\mu_x(B)$ is Fomin-Zelevinsky mutation of the matrix $B$ at the row and column labelled by $x$, i.e.\ we have $\mu_x(B) = (b'_{uv})_{u \in \mu_x(\ex),v \in \mu_{x}(\X)}$, where
\[
	b'_{uv} = \begin{cases}
			- b_{xv} & \text{ if $u = \mu_{x}(x)$}\\
			-b_{ux} & \text{ if $v = \mu_x(x)$} \\
			b_{uv} + \frac{1}{2}(|b_{ux}| b_{xv} + b_{ux}|b_{xv}|) & \text{otherwise.}

		\end{cases}
\]
\end{itemize}

Equation (\ref{eqn:exchange relation}) is referred to as an {\em exchange relation}. Furthermore, given $y \in \X$ we write
\[
	\mu_x(y) = \begin{cases}
							x' & \text{if $y = x$} \\
							y & \text{else},
						\end{cases}
\]
so that we have $\mu_x(\X) = \{\mu_x(y) \mid x \in \X\}$.
\end{definition}

We can naturallly iterate mutation from a given seed. This prepares the definition of a rooted cluster algebra later.

\begin{definition}
	Let $\Sigma = (\X,\ex,B = (b_{vw}))$ be a seed. A {\em $\Sigma$-admissible sequence} is a finite sequence $(x_1, \ldots, x_\ell)$ such that $x_1 \in \ex$ and for all $2 \leq i \leq \ell$ we have 
	\[
		x_i \in \mu_{x_{i-1}} \circ \cdots \circ \mu_{x_1}(\ex).
	\]
We call $\ell$ the {\em length} of the sequence $(x_1, \ldots, x_\ell)$.
Two seeds $\Sigma$ and $\Sigma'$ are called {\em mutation equivalent}, if there exists a $\Sigma$-admissible sequence $(x_1, \ldots, x_\ell)$ such that 
\[
	\Sigma' = \mu_{x_\ell} \circ \cdots \mu_{x_1}(\Sigma).
\]
\end{definition}

Given a $\Sigma$-admissible sequence $\underline{x} = (x_1, \ldots, x_\ell)$ we write $\mu_{\underline{x}}$ for the composite $\mu_{x_\ell} \circ \cdots \circ \mu_{x_1}$. By convention, mutation along the empty sequence $\varnothing$ of length $\ell = 0$ does not change the seed: $\mu_{\varnothing}(\Sigma) = \Sigma$.
Indeed, mutation equivalence defines an equivalence relation on the class of all seeds. Given a seed $\Sigma$, we denote by $[\Sigma]$ its mutation equivalence class.

\begin{definition}[{\cite[Definition~2.3]{FZI} \cite[Definition~1.4]{ADSClus}}]\label{D:cluster algebra}
	Let $\Sigma = (\X,\ex,B)$ be a seed. The {\em cluster algebra associated to $\Sigma$} is the subalgebra
	\[
		\cA(\Sigma) = \bZ\Bigl[ \bigcup_{(\tilde{\X},\tilde{\ex},\tilde{B}) \in [\Sigma]} \tilde{\X}\Bigr] \subseteq \bQ(\X)
	\]
of the field of rational functions in $\X$, generated by the elements of clusters of seeds in $[\Sigma]$.
\begin{itemize}
 \item The elements in $\bigcup_{(\tilde{\X},\tilde{\ex},\tilde{B}) \in [\Sigma]} \tilde{\X}$ are called the {\em cluster variables of $\cA(\Sigma)$}.
\item  The elements in $\bigcup_{(\tilde{\X},\tilde{\ex},\tilde{B}) \in [\Sigma]} \tilde{\ex}$ are called the {\em exchangeable cluster variables of $\cA(\Sigma)$}
\item The elements in $\X \setminus \ex$, are called the {\em frozen variables of $\cA(\Sigma)$}. 
\item The {\em rank of $\cA(\Sigma)$} is $|\X|$, the rank of $\Sigma$. 
\item The seeds in $[\Sigma]$ are called the {\em seeds of $\cA$}. 
\item The pair $(\cA(\Sigma), \Sigma)$ is called the {\em rooted cluster algebra with initial seed $\Sigma$}. 
\end{itemize}
\end{definition}

Note that the cluster algebra associated to $\Sigma$ only depends on the mutation equivalence class of $\Sigma$. We remove this freedom when passing to the rooted cluster algebra, which comes with a fixed initial seed. By abuse of notation, we write $\cA(\Sigma)$ for the rooted cluster algebra $(\cA(\Sigma),\Sigma)$.

\begin{notation}\label{N:integer mutation}
    For ease of notation, for any $\Sigma$-admissible sequence $(x_1, \ldots, x_k)$ and any integer $m \in \bZ$ we will allow the expression $\mu_{x_k} \circ \cdots \circ \mu_{x_1}(m)$ to mean
\[
    \mu_{x_k} \circ \cdots \circ \mu_{x_1}(m) = m,
\]
that is, mutation of an integer along an admissible sequence leaves the integer unchanged.
\end{notation}

\subsection{A category of rooted cluster algebras}

Let $\cA(\Sigma)$ and $\cA(\Sigma')$ be rooted cluster algebras with initial seeds $\Sigma = (\X, \ex, B)$ and $(\X', \ex', B')$ respectively. Given a ring homomorphism $f \colon \cA(\Sigma) \to \cA(\Sigma')$ and a $\Sigma$-admissible sequence $(x_1, \ldots, x_\ell)$, we call the sequence $(f,\Sigma,\Sigma')$-{\em biadmissible} if its image $(f(x_1), \ldots, f(x_\ell))$ is $\Sigma'$-admissible. When the rootings are clear from context, we usually just write that a sequence $(x_1, \ldots, x_\ell)$ is {\em $f$-biadmissible}. We now have the tools to define structure preserving maps between cluster algebras.

\begin{definition}\label{D:melting}[\cite[Definition~2.2]{ADSClus}]
A {\em melting cluster morphism} $f \colon \cA(\Sigma) \to \cA(\Sigma')$ between rooted cluster algebras with seeds $\Sigma = (\X,\ex,B)$ and $\Sigma' = (\X',\ex',B')$ respectively, is a unital ring homomorphism satisfying the following:
	\begin{description}
		\item[CM1] $f(\X) \subseteq \X' \cup \bZ$;
		\item[CM2] for any {\em $f$-biadmissible sequence} $(x_1, \ldots, x_\ell)$ and for all $y \in \X$, we have
		\[
			f(\mu_{x_\ell} \circ \cdots \circ \mu_{x_1}(y)) = \mu_{f(x_\ell)} \circ \cdots \circ \mu_{f(x_1)}(f(y));
		\]
		\item[MCM] $f(\ex) \subseteq \ex' \cup \bZ$.
	\end{description}
\end{definition}

Axioms CM1 and CM2 are not enough to ensure composability of cluster morphisms, hence further assumptions are required. The first M in the name of axiom MCM stands for {\em melting}: We are allowed to {\em melt} frozen cluster variables, that is, send a frozen cluster variable to an exchangeable. However, we are not allowed to freeze: An exchangeable cluster variable in the domain's initial seed $\Sigma$ must be sent to an exchangeable cluster variable in the target's initial seed $\Sigma'$. In \cite{Wierzbicki} the opposite case is treated, where we are allowed {\em freezing} (sending exchangeable cluster variables to frozen variables), but not melting.

\begin{proposition}[{\cite[Proposition~2.5]{ADSClus}}]
	There is a category $\sf{Clus}$ which has as objects rooted cluster algebras and as morphisms melting cluster morphisms. 
\end{proposition}

However, melting cluster morphisms may exhibit degenerate behaviour; they do not necessarily preserve the cluster structure in a meaningful way, as the following example illustrates.

\begin{example}\rm
	Consider the seed $\Sigma$ of rank $3$ and the seed $\Sigma'$ of rank $1$ determined by their ice quivers, the latter with a single vertex:
\begin{center}
\begin{tikzpicture}
	\node (label) at (-2,0) {$\Sigma=$};
	\node (1) at (0,0) {$x$};
	\node[draw] (2) at (1,0) {$y_1$};
	\node[draw] (3) at (-1,0) {$y_2$};
	
	\draw[->] (1) to (2);
	\draw[->] (3) to (1);

	\node (label') at (3,0) {$\Sigma'=$};
	\node (label) at (4,0) {$z$};
\end{tikzpicture}
\end{center}	
We have $\cA(\Sigma) = \bZ[x, y_1, y_2, \frac{y_1+y_2}{x}]$ and $\cA(\Sigma') = \bZ[z,\frac{2}{z}]$. The map
\begin{eqnarray*}
	f \colon \cA(\Sigma) & \to & \cA(\Sigma') \\
		x &\mapsto & 0 \\
		y_1 &\mapsto & 1 \\
		y_2 &\mapsto & -1 \\
		\frac{y_1 + y_2}{x} & \mapsto & r
\end{eqnarray*}
is a rooted cluster morphism for any $r \in \cA(\Sigma)$. The cluster structure is not meaningfully preserved.
\end{example}

Therefore, we restrict our attention to melting cluster morphisms which guarantee that the cluster structure is preserved, and the morphism is controlled by its values on the initial cluster. These morphisms were first studied by Chang and Zhu in \cite{Chang-Zhu-rooted}.

\begin{definition}[{\cite[Definition~2.8]{Chang-Zhu-rooted}}]
	A melting cluster morphism $f \colon \cA(\Sigma) \to \cA(\Sigma')$ between rooted cluster algebras with initial seeds $\Sigma = (\X,\ex,B)$ and $\Sigma' = (\X',\ex',B')$ respectively is called {\em inducible} if it satisfies the following axiom:
	\begin{itemize}	
	\item[iMCM] $f(\ex) \subseteq \ex' \cup \bZ \setminus \{0\}$.
	\end{itemize}
\end{definition}

By the Laurent phenomenon \cite[Theorem~3.1]{FZI}, which asserts that every cluster variable can be expressed as a Laurent polynomial in its initial cluster, with denominator a product of exchangeable cluster variables, an inducible cluster morphism is uniquely determined by its values on the initial cluster.

\begin{definition}
	We denote by $\mCl$ the wide subcategory of $\sf{Clus}$ consisting of inducible melting cluster morphisms, and call this the {\em category of melting cluster morphisms}.
\end{definition}

We recall that a {\em wide} subcategory of a category of $\cC$ is a subcategory of $\cC$ which has the same objects as $\cC$, but not necessarily all of its morphisms. We are now ready to discuss the main new objects, ind-cluster algebras.

\subsection{Cluster algebras of infinite rank are ind-cluster algebras}

Cluster algebras of infinite rank have first been studied in \cite{GG1}. It was shown in \cite{Gcolimits} that any cluster algebra of infinite rank $\cA$ can be realised as a colimit of cluster algebras of finite rank in $\mCl$. 

\begin{theorem}[\cite{Gcolimits}]
\label{T:colimit}
	Let $\cA$ be a cluster algebra of infinite rank, and let $\Sigma$ be any seed of $\cA$. The rooted cluster algebra $\cA(\Sigma)$ is isomorphic to a filtered colimit of rooted cluster algebras of finite rank in $\mCl$.
\end{theorem}

\begin{proof}
	In \cite[Theorem~4.6]{Gcolimits} it is shown that every rooted cluster algebra of countably infinite rank is isomorphic to an $\bN$-shaped colimit of rooted cluster algebras of finite rank in $\mCl$. It follows from Remark \ref{R:uncountable_rank} that every rooted cluster algebra of infinite rank is isomorphic to a filtered colimit of rooted cluster algebras of countable rank. 
    The assertion follows.
	\end{proof}

The converse also holds: filtered colimits of finite rank cluster algebras exist in $\mCl$. In fact:

\begin{theorem}\label{T:filteredclosed}
	The category $\mCl$ has filtered colimits. 
 \end{theorem}

 \begin{proof}
    By \cite[Theorem~1]{AndrekaNemeti} the existence of directed colimits implies the existence of filtered colimits. The claim now follows from the explicit construction from Theorem \ref{T:directed_closed}, which is proved in Section \ref{S:construction}.
 \end{proof}

In Sections \ref{S:construction_seed} and \ref{S:construction}, we explicitly construct directed colimits in $\mCl$, which requires us to consider exchangeably connected components.

\subsection{Neighbours and exchangeably connected components}

Each seed consists of distinct exchangeably connected components.

\begin{definition}\label{D:exchangeably connected}
	Let $\Sigma=(\X,\ex,B)$ be a seed. Two cluster variables $x \in \ex$ and $y \in \X$ are called {\em $B$-neighbours} if $b_{xy} \neq 0$. For any subset $S \subseteq \X$ we write
\[
	N_{B}(S) = \{y \in \X \mid \text{$y$ is a $B$-neighbour of $z$ for some $z \in S$}\}.
\] 
If $x \in \X$, we willl usually just write $N_B(x)$ for $N_B(\{x\})$.
Two elements $v,w \in \X$ are called {\em $B$-exchangeably connected} if there exists a sequence $(x_i)_{i=0}^{n+1}$ with 
	\begin{itemize}
		\item $v = x_0$ and $w = x_{n+1}$;
		\item $x_1, \ldots, x_n \in \ex$;
		\item for all $i = 0, \ldots, n$ the consecutive elements $x_i$ and $x_{i+1}$ are $B$-neighbours.
	\end{itemize}
\end{definition}

\begin{remark} \rm
	Note that two frozen variables in $\X$ cannot be $B$-neighbours. However, they may still be exchangeably connected.
\end{remark}

Let $\Sigma = (\X,ex,B = (b_{vw})_{v,w \in \X})$ be a seed. For every $v \in \ex$ we obtain the unique {\em exchangeably connected component} $\X[v]$ of $\X$ consisting of all cluster variables in $\X$ which are exchangeably connected to $v$. Setting $\ex[\alpha] = \X[\alpha] \cap \ex$ we can decompose $\ex$ into its $B$-exchangeably connected components $\ex = \bigsqcup_{\alpha \in I} \ex[\alpha]$.
We call the seeds 
\[
    \Sigma_\alpha = (\X[\alpha], \ex[\alpha], B[\alpha] = (b_{vw})_{v,w \in \X[\alpha]})
\]
the {\em exchangeably connected components of $\Sigma$}. Note that if $v \in \X \setminus \ex$ is an isolated frozen cluster variable, that is, we have $b_{vw} = 0$ for all $w \in \X$, then it is not contained in any exchangeably connected component.

\begin{example}\rm
Consider the seed $\Sigma$ determined by its ice quiver
\begin{center}
\begin{tikzpicture}
%middle component
	\node (1) at (0,0) {$x_1$};
	\node[draw] (2) at (1,1) {$x_2$};
	\node (3) at (1,-1) {$x_3$};
	\node[draw] (4) at (-1,-1) {$x_4$};
	
	\draw[->] (1) to (2);
	\draw[->] (1) to (3);
	\draw[->] (1) to (4);
	
	\draw[->, bend left] (2) to (3);
	\draw[->, bend right] (2) to (3);
	
	\draw[->] (3) to (4);

%left component
	\node (5) at (-1,0) {$x_5$};
	\node (6) at (-2,1) {$x_6$};
	%\node (7) at (-2,-1) {$x_7$};	
	
	\draw[->] (1) to (2);
	\draw[->] (1) to (3);
	\draw[->] (1) to (4);
	
	\draw[->, bend right] (2) to (5);
	\draw[->] (4) to (5);
	\draw[->] (5) to (6);

%right component
	\node (7) at (2,0) {$x_7$};
	\node (8) at (3,1) {$x_8$};
	
	\draw[->, bend left] (2) to (7);
	\draw[->] (7) to (8);
	\draw[->] (5) to (6);

%extra component
	\node[draw] (9) at (3,-1) {$x_{9}$};
	\node (10) at (4,0) {$x_{10}$};
	\node[draw] (11) at (5,1) {$x_{11}$};
	
	\draw[->] (9) to (10);
	\draw[->] (10) to (11);
	
%isolated frozen
	\node[draw] (12) at (6,0) {$x_{12}$};
	
\end{tikzpicture}
\end{center}	
It has four exchangeably connected components, given by:
\begin{equation*}
\begin{tikzpicture}
%middle component
	\node (1) at (0,0) {$x_1$};
	\node[draw] (2) at (1,1) {$x_2$};
	\node (3) at (1,-1) {$x_3$};
	\node[draw] (4) at (-1,-1) {$x_4$};
	
	\draw[->] (1) to (2);
	\draw[->] (1) to (3);
	\draw[->] (1) to (4);
	
	\draw[->, bend left] (2) to (3);
	\draw[->, bend right] (2) to (3);
	
	\draw[->] (3) to (4);
\end{tikzpicture},\quad
\begin{tikzpicture}
%middle component
	\node[draw] (2) at (1,1) {$x_2$};
	\node[draw] (4) at (-1,-1) {$x_4$};
	
%left component
	\node (5) at (-1,0) {$x_5$};
	\node (6) at (-2,1) {$x_6$};
	%\node (7) at (-2,-1) {$x_7$};	

	\draw[->, bend right] (2) to (5);
	\draw[->] (4) to (5);
	\draw[->] (5) to (6);
\end{tikzpicture},\quad
\begin{tikzpicture}
%middle component
	\node[draw] (2) at (1,1) {$x_2$};
	
%right component
	\node (7) at (2,0) {$x_7$};
	\node (8) at (3,1) {$x_8$};
	
	\draw[->, bend left] (2) to (7);
	\draw[->] (7) to (8);
	
\end{tikzpicture} \text{ and } \quad
\begin{tikzpicture}

%extra component
	\node[draw] (9) at (3,-1) {$x_{9}$};
	\node (10) at (4,0) {$x_{10}$};
	\node[draw] (11) at (5,1) {$x_{11}$};
	
	\draw[->] (9) to (10);
	\draw[->] (10) to (11);
\end{tikzpicture}\;.
\end{equation*}
\end{example}

\begin{remark}\label{R:uncountable_rank}\rm
    Every rooted cluster algebra $\cA(\Sigma)$ can be realised, in $\mCl$, as a filtered colimit of its exchangeably connected components, and its isolated  frozens. In particular, every rooted cluster algebra can be realised in this manner as a colimit of countable rank rooted cluster algebras in $\mCl$, since by \cite[Remark~3.18]{Gcolimits} every exchangeably connected component is countable, and every isolated frozen component has finite rank $1$.
    
    Concretely, let the index set $I_1$ label the exchangeably connected components $\Sigma_\alpha = (\X[\alpha],\ex[\alpha],B \mid_{\X[\alpha]})$ of $\Sigma$, where $\alpha \in I_1$, and let the index set $I_2$ label its isolated frozen components $\Sigma_\beta = (\X_\beta =  \{x_\beta\}, \varnothing, \varnothing)$, where $\beta \in I_2$. Set $I = I_1 \sqcup I_2$ and consider the set $\cF(I) \subseteq 2^I$ of finite subsets of $I$. For $f \in \cF(I)$ we set $\Sigma_f = (\X_f = \bigcup_{\alpha \in f} \X[\alpha], \bigcup_{\alpha \in f} \ex[\alpha],B \mid_{\X_f})$. If $f \subseteq g \in F(I)$, the inclusion $\X_f \to \X_g$ induces a map $\cA(\Sigma_f) \to \cA(\Sigma_g)$ in $\mCl$. Hence we obtain a filtered diagram $\cF(I) \to \mCl$, of which $\cA(\Sigma)$ is the colimit.
\end{remark}

We can read off from the initial seeds whether two rooted  cluster algebras are isomorphic in $\mCl$. The key is the notion of {\em similarity}.

\begin{definition}\label{D:similar}[{\cite[Definition~3.24]{Gcolimits}}]
	Two seeds $\Sigma = (\X,\ex,B = (b_{uv})_{u,v \in \X})$ and $\Sigma'=(\X',\ex',B' = (b')_{u' v'})_{u',v' \in \X'}$ are called {\em similar}, if there exists a bijection $\varphi \colon \X \to \X'$ such that
	\begin{itemize}
			\item the map $\varphi$ restricts to a bijection $\varphi \colon \ex \to \ex'$;
			\item for every exchangeably connected component $\Sigma[\alpha] = (\X[\alpha],\ex[\alpha],B[\alpha])$ of $\Sigma$ there exists a $\sigma_\alpha \in \{+1,-1\}$ such that following holds: If $u \in \ex[\alpha]$ and $v \in \X[\alpha]$, then
			\[
				b_{uv} = \sigma_\alpha b'_{\varphi(u)\varphi(v)}. 
			\]
	\end{itemize}
	If the map $\varphi$ is the identity map, then we call the seeds {\em strongly similar}.
\end{definition}

\begin{proposition}\label{P:similar_iso}[{\cite[Theorem~3.25]{Gcolimits}}]
    Two rooted cluster algebras are isomorphic in $\mCl$ if and only if their initial seeds are similar.
\end{proposition}

\begin{remark}\rm
    Mutation at an exchangeable variable $x \in \ex$ only affects the exchangeably connected component containing $x$: If $\Sigma_\alpha$ is an exchangeably connected component of $\Sigma$ with $x \in \ex \setminus \ex[\alpha]$, then $\mu_x(\Sigma_\alpha) = \Sigma_\alpha$.
\end{remark}

\begin{lemma}\label{L:ws}[{\cite[Lemma~3.28]{Gcolimits}\cite[Proposition~3.8]{Gcolimits}}]
	Let $\Sigma = (\X,ex,B = (b_{vw})_{v,w \in \X})$ and $\Sigma'=(\X',\ex',B'= (b'_{vw})_{v,w \in \X'})$ be seeds. 
	Let $f \colon \X \to \X'$ be a melting cluster morphism. The following hold.
	\begin{enumerate}	
	\item\label{ws1} If $x,y \in \X$ with $x \neq y$ but $f(x) = f(y) \in \X'$, then $x,y \in \X \setminus \ex$ are frozen variables.
	\item\label{ws2} For every exchangeably connected component $\Sigma'_\alpha = (\X'[\alpha],\ex'[\alpha],(B')^\alpha)$ of $\Sigma'$ there exists a $\sigma_\alpha \in \{+,-\}$ such that the following holds: For every $v \in \ex$ with $f(v) \in \ex'[\alpha]$ and every $w' \in \X'$ we have
		\[
			b_{f(v)w'} = \sigma_\alpha \sum_{f(w)=w'} b_{vw}.
		\]
		Moreover, the summands in the sum are either all positive or all negative.
		\end{enumerate}
\end{lemma}

We conclude the section with an observation on possible specialisations of neighbours of exchangeable cluster variables: Neighbours of $x \in \ex$ with $f(x) \in \ex'$ can only be specialised to either $1$ or $-1$, and the product of all of their images is $1$.

\begin{lemma}\label{L:specialisation of neighbours}
	Let $f \colon \cA(\Sigma) \to \cA(\Sigma')$ be a melting cluster morphism, where $\Sigma = (\X,\ex,B = (b_{uv}))$ and $\Sigma' = (\X',\ex',B' = (b'_{uv}))$. Denote by $\X_\bZ = f^{-1}(\bZ) \cap \X$ the initial cluster variables which get sent to integers. If $x \in \ex$ with $f(x) \in \ex'$, then 
	\[
		f(\prod_{y \in {\X_\bZ} \colon b_{xy}>0} y^{b_{xy}}) = 1 = f(\prod_{y \in {\X_\bZ} \colon b_{xy}<0} y^{-b_{xy}}).
	\]
	In particular, setting $\X_{cl} = \X \setminus \X_{\bZ} =\{x \in \X \mid f(x) \in \X'\}$, we have
	\[
		f(\mu_x(x)) = \frac{\prod_{u \in \X_{cl} \colon b_{xu}>0}u^{b_{xu}} + \prod_{v \in \X_{cl} \colon b_{xv}<0}v^{-b_{xv}}}{x}.
	\]
\end{lemma}

\begin{proof}
	Set $f(x) =x'$. Since the sequence $(x)$ is $f$-biadmissible, we have
	\begin{eqnarray*}
		 \frac{\prod_{y \in \X \colon b_{xy}>0}f(y)^{b_{xy}} + \prod_{y \in \X \colon b_{xy}<0} f(y)^{-b_{xy}}}{f(x)} &=& f(\mu_x(x)) =
		\mu_{f(x)}(f(x))
		\\
		&=& \frac{\prod_{y' \in \X' \colon b'_{x'y'}>0}{y'}^{b'_{x'y'}} + \prod_{y' \in \X' \colon b'_{x'y'}<0} {y'}^{-b'_{x'y'}}}{f(x)} 
	\end{eqnarray*}
	Since $\prod_{y' \in \X' \colon b'_{x'y'}>0}{y'}^{b'_{x'y'}}$ and $\prod_{y' \in \X' \colon b'_{x'y'}<0} {y'}^{-b'_{x'y'}}$ are coprime, by algebraic independence of $\X'$ we must have
	\[
		\prod_{y \in \X \colon b_{xy}>0}f(y)^{b_{xy}}  = \prod_{y' \in \X' \colon b'_{x'y'}>0}{y'}^{b'_{x'y'}} \; \text{ and } \prod_{y \in \X \colon b_{xy}<0} f(y)^{-b_{xy}} = \prod_{y' \in \X' \colon b'_{x'y'}<0} {y'}^{-b'_{x'y'}}
	\]
	or vice versa. Hence, all of these are monic monomials; the claim follows.
\end{proof}

\subsection{Construction of ind-seeds}\label{S:construction_seed}

Throughout this section, we fix a directed system $F \colon \cF \to \mCl$. For all $i \in \cF$ set $F(i) = \cA(\Sigma_i)$ with seed $\Sigma_i = (\X_i,\ex_i,B^i = (b_{vw}^i)_{v,w \in \X_i})$, and for a morphism $i \leq j$ in $\cF$ set $F(i \leq j) = f_{ij}$.

\begin{definition}
The {\em initial cluster functor $\bX$} is the functor $\bX \colon \mCl \to \mathrm{Set}$ which assigns to each rooted cluster algebra 
$\cA((\X,\ex,B))$ the set $\X \cup \bZ$, and to each melting cluster morphism $f \colon \cA((\X,\ex,B)) \to \cA((\X',\ex',B'))$ its restriction $f \mid_{\X \cup \bZ} \colon \X \cup \bZ \to \X' \cup \bZ$. 

Similarly, the {\em exchangeable initial cluster functor $\bE$} is the functor which assigns to each rooted cluster algebra $\cA((\X,\ex,B))$ the set $\ex \cup \bZ$, and to each melting cluster morphism $f \colon \cA((\X,\ex,B)) \to \cA((\X',\ex',B'))$ its restriction $f \mid_{\ex \cup \bZ} \colon \ex \cup \bZ \to \ex' \cup \bZ$.
\end{definition}
We will use these functors to define seeds for ind-cluster algebras. Their clusters and exchangeable variables can be computed using colimits in the category of sets.
\begin{remark}\label{R:ind-cluster}\rm
	We want to consider the sets $\X = (\colim \bX \circ F) \setminus \bZ$ and $\ex = (\colim \bE \circ F) \setminus \bZ$. These are computed via colimits in the category of sets. We remind the reader that this means that 
\[
	\X  \cong \left(\bigsqcup_{i \in \cF} (\X_i \cup \bZ) \big/ \sim \right) \setminus \bZ,
\] 
where for $x \in \X_i$ and $y \in \X_j$ we have $x \sim y$ if and only if $f_{ik}(x) = f_{jk}(y)$ for some $k \in \cF$ with $i,j \leq k$. Setting $\tilde{X}_i = \{x \in \X_i \mid f_{ij}(x) \in \X_j \text{ for all } j \geq i\}$, we can also write
\[
	\X \cong \bigsqcup_{i \in \cF} \tilde{X}_i \big/ \sim
\]
  Similarly, we have 
\[
	\ex \cong  \left(\bigsqcup_{i \in \cF} (\ex_i \cap \tilde{X}_i) \big/ \sim \right) \; \subseteq \X.
\]
For every $k \in \cF$, the colimit map $f_k \colon \X_k \to \X$ restricts to the colimit map $f_k \colon \ex_k \to \ex$.
\end{remark}

The sets $\X$ and $\ex$ from Remark \ref{R:ind-cluster} will give us the cluster, and set of exchangeable variables respectively, for the seed of an ind-cluster algebra. Defining the exchange matrix of this seed requires a bit more finesse, and we set this up via a series of lemmata below.

\begin{lemma}\label{L:well-defined-seed}
	Set $\X = (\colim \bX \circ F) \setminus \bZ$ and $\ex = (\colim \bE \circ F) \setminus \bZ$ with colimit maps $f_k \colon \X_k \to \X \cup \bZ$. The following hold:
	\begin{enumerate}
		\item\label{well-defined-seed1} For all $x \in \ex$ there exists an $\ell \in \cF$ such that for all $k \geq \ell$ there exists a unique $x_k \in \X_k$ with $f_k(x_k) = x$. Moreover, we have $x_k \in \ex_k$.
		\item\label{well-defined-seed2}{Let $x \in \ex, y \in \X$. One of the following holds:
		\begin{itemize}
			\item For all $k \in \cF$ and all $x_k,y_k \in \X_k$ with $f_k(x_k) = x$ and $f_k(y_k) = y$ we have
			\[
				b^k_{x_ky_k} = 0.
			\]
			In this case, we set $\tilde{b}_{xy} = 0$, and we say that $\tilde{b}_{xy}$ is {\em attained at every $k \in \cF$}.
			\item There exists an $\ell \in \cF$ such that for all $k \geq \ell$ there exists a unique $x_k \in \ex_k$ and a unique $y_k \in \X_k$ with $f_k(x_k) = x$ and $f_k(y_k) = y$ and $b^k_{x_ky_k} \neq 0$. Moreover, we have
			\[
				b^k_{x_ky_k} = \pm b^\ell_{x_\ell y_\ell}
			\]
			and for all $j \geq k \geq \ell$ we have $f_{kj}(x_k) = x_j$ and $f_{kj}(y_k)=y_j$.
			In this case, we set $\tilde{b}_{xy} = |b_{x_\ell y_\ell}|$ and say $\tilde{b}_{xy}$ is {\em attained at $\ell$ by $(x_\ell, y_\ell)$.}
		\end{itemize}
		}

		\item\label{well-defined-seed3}{Let $\X_\alpha$ be a $\tilde{B}$-exchangeably connected component and let $u,x \in \X_\alpha \cap \ex$ and $v,y \in \X_\alpha$ with $\tilde{b}_{uv} > 0$ and $\tilde{b}_{xy} > 0$, both attained at all $k \geq \ell \in \cF$ by $(u_k,v_k)$ and $(x_k,y_k)$ respectively.  Then there exists an $\ell' \geq \ell$ such that  either for all $k \geq \ell'$
		\[
			(b^k_{x_ky_k}) \cdot (b^k_{u_kv_k}) > 0 
		\]
 or for all $k \geq \ell'$ 
	\[
		(b^k_{x_ky_k}) \cdot (b^k_{u_kv_k}) <0.
	\]
 If the product is always positive, we say that  the pairs $(u,v)$ and $(x,y)$ are {\em positively aligned} and if it is always negative, we say that the pairs $(u,v)$ and $(x,y)$ are {\em negatively aligned}.}
	\end{enumerate}
\end{lemma}

\begin{proof}
	Let $x \in \ex$. Then there exists some $\ell \in \cF$ and $x_\ell \in \ex_\ell$ such that $f_\ell(x_\ell) = x$. For all $k \geq \ell$ we have $x_k = f_{\ell k}(x_\ell) \in \ex_k \cup \bZ$, since $f_{\ell k}$ is a melting cluster morphism. Since $f_k(x_k) = x \in \ex$, we must have $x_k \in \ex_k$. Moreover, if $x'_k \in \X_k$ with $f_k(x'_k) = x$ then there exists a $m \geq k$ such that $f_{km}(x'_k) = f_{km}(x_k)$. By Lemma \ref{L:ws} we have $x'_k = x_k$. This proves the first claim.
	
	Let $x \in \ex, y \in \X$. Assume that there exists some $m \in \cF$ with $x_m, y_m \in \X_m$, $f_m(x_m)=x$, $f_m(y_m)=y$ and $b^m_{x_my_m} \neq 0$. For all $k \geq m$ set $x_k = f_{mk}(x_m) \in \X_k$ and $y_k = f_{mk}(y_m) \in \X_k$. By Lemma \ref{L:ws} we have
	\[
		0 < |b^m_{x_m y_m}| \leq  |b^k_{x_ky_k}| = \sum_{y'_m \in \X_m \colon f_{mk}(y'_m)=y_k} |b^m_{x_my'_m}| \leq  \sum_{z \in \X_m}|b^m_{x_mz}| = \alpha.
	\]
Hence the non-empty set $\{|b^k_{x_ky_k}| \mid k \geq m \}$ is bounded above by $\alpha$, and has a maximum. Assume the maximum is attained by $|b^\ell_{x_\ell y_\ell}| > 0$ for some $\ell \geq m$. For all $k \geq \ell$ we have
\[
	|b^\ell_{x_\ell y_\ell}| \geq |b^k_{x_k y_k}| = \sum_{y'_\ell \in \X_\ell \colon f_{\ell k}(y'_\ell) = y_k} |b^\ell_{x_\ell y'_\ell}| \geq |b^\ell_{x_\ell y_\ell}|
\]
and hence
\[
	b^k_{x_k y_k} = \pm b^\ell_{x_\ell y_\ell} \neq 0.
\]
By the first part, we can choose $\ell \in \cF$ big enough such that for all $k \geq \ell$ the element $x_k \in \ex_k$ is the unique cluster variable in $\X_k$ such that $f_k(x_k) = x$. For $k \geq \ell$ assume that $y_k \neq y'_k \in \X_k$ is such that $f_k(y'_k) = y$. Then there exists an $m \geq k$ such that
\[
	f_{km}(y'_k) = y_m
\]
and hence
\[
	|b^k_{x_k y_k}| = |b^m_{x_m y_m}| = \sum_{z \in \X_k \colon f_{km}(z) = y_m} |b^k_{x_k z}| \geq |b^k_{x_k y_k}| + |b^k_{x_k y'_k}|.
\]
Hence $b^k_{x_ky'_k} = 0$. This proves the second claim.

For the third claim, observe that since $u,v,x,y \in \X$ are $\tilde{B}$-exchangeably connected, by the second claim there exists an $\ell' \geq \ell$ such that for all $k \geq \ell'$, $u_k, v_k, x_k$ and $y_k$ are $B^k$-exchangeably connected. By (\ref{ws2}) of Lemma \ref{L:ws}, the entries $b^{\ell'}_{u_{\ell'} v_{\ell'}}$ and $b^{\ell'}_{x_{\ell'} y_{\ell'}}$ have the same sign if and only if $b^{k}_{u_k v_k}$ and $b^{k}_{x_k y_k}$ have the same sign for all $k \geq \ell'$. 
\end{proof}

We are now ready to define the candidate for a seed of an ind-cluster algebra.

\begin{definition}\label{D:ind-seed}
	We define the {\em ind-seed associated to $F \colon \cF \to \mCl$} to be the triple $\Sigma(F)=(\X,\ex,B)$ given by the following data.
	\begin{itemize}
			\item{ The {\em cluster $\X$} is the set $\X = (\colim \bX \circ F) \setminus \bZ$.}
			\item{The {\em set of exchangeable variables $\ex$} is the set $\ex = (\colim \bE \circ F) \setminus \bZ$.}
			\item{Consider the matrix $\tilde{B}=(\tilde{b}_{vw})_{v,w \in \X}$ defined as in Lemma \ref{L:well-defined-seed}(2). For each $\tilde{B}$-exchangeably connected component $\X[\alpha]$ of $\X$ with $|\X[\alpha]| \geq 2$, we pick a {\em sign-determining entry $\tilde{b}_{xy} >0$ of $\tilde{B}$}. For every $u \in \ex,v \in \X$ we set
			\[
				b_{uv} = \begin{cases}
									\tilde{b}_{uv} & \text{if $(u,v)$ and $(x,y)$ are positively aligned}\\
			-\tilde{b}_{uv} & \text{if $(u,v)$ and $(x,y)$ are negatively aligned}		\\
			0 & \text{if $\tilde{b}_{uv} = 0$}.				
								\end{cases}
			\]
			This defines the {\em exchange matrix $B = (b_{uv})_{u \in \ex,v \in \X}$ associated to $F$}. By Lemma \ref{L:well-defined-seed} it is well-defined.
			}
	\end{itemize}
\end{definition}

\begin{remark}\rm
	Definition \ref{D:ind-seed} relies on a choice of sign-determining entry for each $\tilde{B}$-exchangeably connected component with at least two elements. The signs of the entries in the exchange matrix $B$ are dependent on this choice. A different choice will lead to a seed $\tilde{\Sigma}$ which is strongly similar to $\Sigma(F)$, and hence to a cluster algebra $\cA(\tilde{\Sigma})$ which is canonically isomorphic in $\mCl$ to $\cA(\Sigma(F))$.
\end{remark}

We need the following stabilising property for entries in the exchange matrices to prove that $\mCl$ is closed under directed colimits.

\begin{lemma}\label{L:attain finite submatrix}
	Consider the ind-seed $\Sigma(F) = (\X,\ex,B)$ and let $\tilde{\X} \subseteq \X$ be a finite subset, with $\tilde{\ex} = \tilde{\X} \cap \ex$ labelling the rows and $\tilde{\X}$ labelling the columns of a finite submatrix  $\tilde{B}$ of $B$. Then there exists an $\ell \in \cF$ such that for all $i \geq \ell$ and all $y \in \tilde{\X}$ the following holds:
\begin{enumerate}
\item{If $y \in \ex$ there exists a unique $y_i \in \X_i$ with $f_i(y_i) = y$.}
 \item{If $y \in \X \setminus \ex$, and $N_{\tilde{B}}(y) \cap \ex \neq \varnothing$ there exists a unique $y_i$ such that for all $x \in N_{\tilde{B}}(x) \cap \ex$ we have
\[
	|b_{xy}| = |b^i_{x_i y_i}|.
\]
}
\item{ If $y \in \X \setminus \ex$ and $N_{\tilde{B}}(y) \cap \ex = \varnothing$, there exists some $y_i \in \X_i$, not necessarily unique, such that $f_i(y_i) = y$. In this case, for all $x \in \tilde{\X} \cap \ex$ we have 
\[
	b_{xy} = b^i_{x_i y_i} =0.
\]
}
\end{enumerate} 
We say that the submatrix $\tilde{B}$ is {\em uniformly attained at $\ell$}, and that {\em $y$ is $\tilde{B}$-attained by $y_i$}.
\end{lemma}

\begin{proof}
Claim (1) is (\ref{well-defined-seed1}) of Lemma \ref{L:well-defined-seed} repeated for notational purposes. Claim (3) follows immediately from the construction of $B$.
It remains to show Claim (2). Observe that by (\ref{well-defined-seed2}) of Lemma \ref{L:well-defined-seed} there exists an $\ell \in \cF$ such that every entry of $\tilde{B}$ is attained at $\ell$. Let $i \geq \ell$ and consider the set $\tilde{\X} \setminus \ex = \{u_1, \ldots, u_m\}$ of frozen variables  of $\tilde{\X}$. 
For $u = u_1$ consider the set $N_{\tilde{B}}(u) \cap \ex$ of exchangeable $\tilde{B}$-neighbours of $u$. This set is finite, since it only contains entries from $\tilde{\X} \cap \ex$, say $N_{\tilde{B}}(u) \cap \ex = \{x(1), \ldots, x(n)\}$. For every $1 \leq p \leq n$ there exists a unique $x_i(p) \in \ex_i$ such that $f_i(x_i(p)) = x(p)$. Moreover, for every $1 \leq p \leq n$ the entry $b_{x(p)u}$ is attained at $\ell$, so there exists a unique $u_i(p)$ with $f_i(u_i(p)) = u$ and 
	\[
		|b_{x(p)u}| = |b_{x_i(p)u_i(p) }|.
	\]
Note that a priori we cannot guarantee that $u_i(p) = u_i(q)$ for $p \neq q$. However, since $f_i(u_i(p)) = f_i(u_i(q))$ for all $1 \leq p,q \leq n$ there exists a $\ell_1 \geq \ell$ such that for all $j \geq \ell_1$ and all $1 \leq p,q \leq n$ we have $f_{ij}(u_i(p)) = f_{ij}(u_i(q)) = u_j$. Setting $x_j(p) = f_{ij}(x_i(p))$ and we have, for all $1 \leq p \leq n$,
\[
	|b^j_{x_j(p)u_j }| = \sum_{u'_i \colon f_{ij}(u'_i) = u_j} |b^i_{x_i(p)u'_k }| = |b^i_{x_i(p)u_i(p) }| = |b_{x(p)u}|.
\]
Hence $u=u_1$ is $\tilde{B}$-attained at $\ell_1$. Repeating the procedure one by one for $u_2, u_3, \ldots, u_m$ allows us to find an $\ell \in \cF$ as desired.
\end{proof}

We now  can show that the ind-seed satisfies the properties of a seed (Definition \ref{D:seed}), and therefore gives rise to a cluster algebra.

\begin{lemma}\label{L:ex-locally finite}
	The ind-seed associated to the directed system $F \colon \cF \to \mCl$ is a seed called the {\em seed of $F$}.
\end{lemma}

\begin{proof}
	We  need to show that $B$ is $\ex$-locally finite and skew-symmetriseable. To see that it is $\ex$-locally finite, let $v \in \ex$ and take $k \in \cF$ such that $v = f_k(v_k)$ for a unique $v_k \in \X_k$. Since $B^k$ is $\ex_k$-locally finite, the set of $B^k$-neighbours
\[
	N_k(v_k) = \{w_k \in \X_k \mid b^k_{v_kw_k} \neq 0\}
\]
of $v_k$ is finite.
Let now $w \in \X$. If $b_{vw} \neq 0$ then by (\ref{well-defined-seed2}) of Lemma \ref{L:well-defined-seed}, the entry $\tilde{b}_{vw} > 0$ is attained at some $\ell \geq k$ by $(v_\ell,w_\ell)$. By (\ref{ws2}) of Lemma \ref{L:ws} we have
\[
	0 < |b_{vw}| = \tilde{b}_{vw} = |b^\ell_{v_\ell w_\ell}| = \sum_{f_{k \ell}(w_k) = w_\ell} |b^k_{v_k w_k}|.
\]
Therefore, there exists a $w_k \in N_k(v_k)$ such that $f_{k \ell}(w_k) = w_\ell$, i.e.\ $w = f_k(w_k)$ for some $w_k \in N_k(v_k)$. Since $N_k(v_k)$ is finite, so is the set of $B$-neighbours $N(v) = \{w \in \X \mid b_{vw} \neq 0\}$ of $v$.

To see that $B$ is skew-symmetriseable, we observe that for every finite subset $\tilde{\ex} \subseteq \ex$ the submatrix $\tilde{B} = B \mid_{\tilde{\ex}}$ of $B$ with rows and columns labelled by $\tilde{\ex}$ is uniformly attained at some $\ell$. 
Since $B^\ell$ is skew-symmetriseable, so is $\tilde{B}$. Hence, so is $B$.
\end{proof}

\begin{example}\rm 
For each $n \geq 1$ consider the skew-symmetric seed $\Sigma_n = (\X_n,\ex_n,B_n)$ given by the ice quiver
	\begin{center}
		\begin{tikzpicture}[scale=2]
				\node (1) at (0,0) {$x_1$};
				\node (2) at (1,0) {$x_2$};
				\node (3) at (2,0) {$x_3$};
				\node (4) at (3,0) {$\cdots$};
				\node (n-2) at (4,0) {$x_{n-2}$};
				\node (n-1) at (5,0.5) {$x_{n-1}$};
				
				\node[draw] (z1) at (0,-1) {$z_1$};
				\node[draw] (z2) at (1,-1) {$z_2$};
				\node[draw] (z3) at (2,-1) {$z_3$};
				\node (z4) at (3,-1) {$\cdots$};
				\node[draw] (zn-2) at (4,-1) {$z_{n-2}$};	
				\node[draw] (zn-1) at (5,-0.5) {$z_{n-1}$};
				\node[draw] (z'n-1) at (5,-1.5) {$z'_{n-1}$};
				
				%\node[draw] (z'n-2) at (4,-2) {$z'_{n-2}$};	
				\node[draw] (s) at (5,1) {$s$};

				\node[draw] (n1) at (6,1) {$v_1$};
				\node[draw] (n2) at (6,0.5) {$v_2$};
				\node (n3) at (6,0) {$\vdots$};
				\node[draw] (n4) at (6,-0.5) {$v_{n-1}$};
				%\node[draw] (nn) at (6,-2) {$s$};
				
				\node (y1) at (0,-2) {$y_1$};
				\node (y2) at (1,-2) {$y_2$};
				\node (y3) at (2,-2) {$y_3$};
				\node (y4) at (3,-2) {$\cdots$};
				\node (yn-2) at (4,-2) {$y_{n-2}$};
				\node (yn-1) at (5,-2.5) {$y_{n-1}$};
				\node[draw] (yn) at (6,-2.5) {$y_{n}$};
				
				\draw[->] (1) to (2);
				\draw[->] (2) to node[midway,above] {$2$} (3);
				\draw[->] (3) to node[midway,above] {$3$} (4);
				\draw[->] (4) to node[midway,above] {$n-3$} (n-2);
				\draw[->] (n-2) to node[midway,above left] {$n-2$} (n-1);
				\draw[->] (n-1) to (n1);
				\draw[->] (n-1) to (n2);
				\draw[->] (n-1) to (n3);
				\draw[->] (n-1) to (n4);
				%\draw[->] (n-1) to (nn);

				\draw[->] (1) to (z1);
				\draw[->] (2) to (z2);
				\draw[->] (3) to (z3);
				\draw[->] (n-2) to (zn-2);
				\draw[->] (n-1) to (zn-1);
				%\draw[->, bend right] (n-2) to (z'n-2);
				\draw[->] (n-1) to (s);
				
				\draw[->] (y1) to (z1);
				\draw[->] (y2) to (z2);
				\draw[->] (y3) to (z3);
				\draw[->] (yn-2) to (zn-2);
				\draw[->] (yn-1) to (z'n-1);
				\draw[->] (yn-1) to (yn);
				
				\draw[->] (y1) to (y2);
				\draw[->] (y2) to  (y3);
				\draw[->] (y3) to  (y4);
				\draw[->] (y4) to (yn-2);
				\draw[->] (yn-2) to (yn-1);
				
		\end{tikzpicture}
	\end{center}	
	Instead of drawing multiple arrows between two nodes, we indicate the number of arrows by a label, where no label indicates that there is just one arrow. 
	By \cite[Theorem~3.30]{Gcolimits} we obtain an inducible melting cluster morphism
	$f_{n} \colon \cA(\Sigma_n) \to \cA(\Sigma_{n+1})$ defined on the initial cluster as follows: For all $1 \leq i \leq n-1$ and $1 \leq j \leq n$ we map
	\[
		x_i \mapsto x_i, \; \;  v_i \mapsto x_n , \; \; y_j \mapsto y_j, \; \; z_i \mapsto z_i,\; \; z'_{n-1} \mapsto z_{n-1}, \; \; \text{ and } \; \; s \mapsto 1.
	\]
	This defines a directed system $F \colon \mathbb{N} \to \mCl$. The ind-seed $\Sigma(F) = (\X,\ex,B)$ is described by the ice quiver
	\begin{center}
		\begin{tikzpicture}[scale=2]
				\node (1) at (1,0) {$x_1$};
				\node (2) at (2,0) {$x_2$};
				\node (3) at (3,0) {$x_3$};
				\node (4) at (4,0) {$\cdots$};
				\node (n-1) at (5,0) {$x_{n-1}$};
				\node (n) at (6,0) {$x_n$};
				\node (n+1) at (7,0) {$x_{n+1}$};
				\node (n+2) at (8,0) {$\cdots$};
				
				\node[draw] (z1) at (1,-1) {$z_1$};
				\node[draw] (z2) at (2,-1) {$z_2$};
				\node[draw] (z3) at (3,-1) {$z_3$};
				\node (z4) at (4,-1) {$\cdots$};
				\node[draw] (zn-1) at (5,-1) {$z_{n-1}$};				
				\node[draw] (zn) at (6,-1) {$z_n$};
				\node[draw] (zn+1) at (7,-1) {$z_{n+1}$};
				\node (zn+2) at (8,-1) {$\cdots$};
				
				\draw[->] (1) to (2);
				\draw[->] (2) to node[midway,above] {$2$} (3);
				\draw[->] (3) to node[midway,above] {$3$} (4);
				\draw[->] (4) to node[midway,above] {$n-2$} (n-1);
				\draw[->] (n-1) to node[midway,above] {$n-1$}  (n);
				\draw[->] (n) to node[midway,above] {$n$}  (n+1);
				\draw[->] (n+1) to node[midway,above] {$n+1$}  (n+2);

				\draw[->] (1) to (z1);
				\draw[->] (2) to (z2);
				\draw[->] (3) to (z3);
				\draw[->] (n-1) to (zn-1);
				\draw[->] (n) to (zn);
				\draw[->] (n+1) to (zn+1);
				
				\node (y1) at (1,-2) {$y_1$};
				\node (y2) at (2,-2) {$y_2$};
				\node (y3) at (3,-2) {$y_3$};
				\node (y4) at (4,-2) {$\cdots$};
				\node (yn-1) at (5,-2) {$y_{n-1}$};				
				\node (yn) at (6,-2) {$y_n$};
				\node (yn+1) at (7,-2) {$y_{n+1}$};
				\node (yn+2) at (8,-2) {$\cdots$};
				
				\draw[->] (y1) to (y2);
				\draw[->] (y2) to  (y3);
				\draw[->] (y3) to  (y4);
				\draw[->] (y4) to  (yn-1);
				\draw[->] (yn-1) to  (yn);
				\draw[->] (yn) to  (yn+1);
				\draw[->] (yn+1) to (yn+2);

				\draw[->] (y1) to (z1);
				\draw[->] (y2) to (z2);
				\draw[->] (y3) to (z3);
				\draw[->] (yn-1) to (zn-1);
				\draw[->] (yn) to (zn);
				\draw[->] (yn+1) to (zn+1);
		\end{tikzpicture}
	\end{center}	
with countable cluster $\X = \{x_i, y_i, z_i \mid i \in \bN\}$, countable set of exchangeable variables $\ex = \{x_i, y_i \mid i \in \bN\}$ and for each $i \in \bN$ we have $i$ arrows from $x_i$ to $x_{i+1}$, one arrow from $y_i$ to $y_{i+1}$ as well as one arrow from $x_i$ to $z_i$ and from $y_i$ to $z_i$. Consider the finite submatrix $B_{n-1}$ of $B$ with rows labelled by $\{x_i, y_i \mid 1 \leq i \leq n-1\}$ and columns labelled by $\{x_i, z_i, y_i \mid 1 \leq i \leq n-1\}$. While every entry of the matrix $B_{n-1}$ is attained at $n \in \bN$, the entry $b_{x_{n-1}z_{n-1}}$ is attained by $(x_{n-1}, z_{n-1})$ in $\X_n$, while the entry $b_{y_{n-1}z_{n-1}}$ is attained by $(y_{n-1}, z'_{n-1})$. The matrix $B_{n-1}$ is thus not uniformly attained at $n$. It is only uniformly attained at $n+1$. 
\end{example}	

\subsection{Directed colimits}\label{S:construction} Throughout this section, we keep working with a fixed directed system $F \colon \cF \to \mCl$. For all $i \in \cF$ set $F(i) = \cA(\Sigma_i)$ with seed $\Sigma_i = (\X_i,\ex_i,B^i = (b_{vw}^i)_{v,w \in \X_i})$, and for a morphism $i \leq j$ in $\cF$ set $F(i \leq j) = f_{ij}$. The ind-seed $\Sigma(F)$ of $F$ gives rise to a co-cone $\cA(\Sigma(F))$ of the diagram $F$.

\begin{proposition}\label{P:cone}
	The colimit maps $f_i \colon \X_i \to \X \cup \bZ$ uniquely extend to ring homomorphisms $f_i \colon \cA(\Sigma_i) \to \cA(\Sigma(F))$. Moreover, the following holds:
\begin{enumerate}	
	\item Every $f_i$-biadmissible sequence $\underline{x}_i$ is $f_{ij}$-biadmissible for all $j \geq i$. We set $f_i(\underline{x}_i) = \underline{x}$ and $f_{ij}(\underline{x}_i) = \underline{x}_j$.
	\item Every such sequence gives rise to a directed system $\mu_{\underline{x}}(F) \colon \cF_{\geq i} \to \mCl$ given by $\mu_{\underline{x}}(F)(j) = \cA(\mu_{\underline{x}_j}\Sigma_j)$ with maps $\mu_{\underline{x}}(F)(i \leq j) = F(i \leq j) = f_{ij}$. Here, $\cF_{\geq i}$ is the full subcategory of objects in $\cF$ that are at least $i$.
	\item We have a strong similarity of seeds $\mu_{\underline{x}}(\Sigma(F)) \cong \Sigma(\mu_{\underline{x}}(F))$. 
	%\item The ring homomorphisms $f_i \colon \cA(\Sigma_i) \to \cA(\Sigma(F))$ are inducible melting cluster morphisms.
    \item For all $y_i \in \X_i$ with $f_i(y_i) = y$ we have $f_i(\mu_{\underline{x}_i}(y_i)) = \mu_{\underline{x}}(y)$.
\end{enumerate}
In particular, the rooted cluster algebra $\cA(\Sigma(F))$ is a co-cone of $F$ in $\mCl$.
\end{proposition}

\begin{proof}
	By the Laurent phenomenon, every cluster variable in $\cA(\Sigma_i)$ can be written as a Laurent polynomial $\frac{P(\X_i)}{M(\ex_i)}$, where $P$ is a polynomial expression in $\X_i$, and $M$ is a monic monomial expression in $\ex_i$. Since the maps $f_{ij}$ are inducible, we have that $f_i(\ex_i) \subseteq \ex \cup \bZ \setminus \{0\}$, and we can algebraically extend the map $f_i$ to a unique ring homomorphism
\[
	f_i \colon \cA(\Sigma_i) \to \cA(\Sigma).
\]
By construction, these ring homomorphisms satisfy CM1 and iMCM, and for all $i \leq j$ we have $f_i = f_j \circ f_{ij}$. In particular, if we can show Claim (4) then we know that $f_i$ also satisfies axiom CM2, and hence is an induced melting cluster morphism. Thus we will have shown that $\cA(\Sigma(F))$ is a co-cone of $F$ in $\mCl$.

The Claims (1)-(4) are trivially true for sequences of length $0$. Now assume they hold for sequences of length $\ell-1$. Let $\underline{x}_i = (x_1(i), \ldots, x_\ell(i))$ and consider its truncation $\underline{x}'_i = (x_1(i), \ldots, x_{\ell-1}(i))$ of length $\ell-1$. For $j \geq i$ we set $\underline{x}'_j = f_{ij}(\underline{x}'_i)$ and $\underline{x}' = f_{i}(\underline{x}'_i)$. By induction assumption for Claim (2), by mutating $\Sigma_j$ along $\underline{x}'_j$ for all $j \geq i$, we get a directed system $\mu_{\underline{x}'}(F)$ with colimit $\cA(\Sigma(\mu_{\underline{x}'}(F)))$, which, by induction assumption for Claim (3) and Proposition \ref{P:similar_iso}, is canonically isomorphic to $\cA(\mu_{\underline{x}'}(\Sigma(F)))$. To show that Claims (1)-(4) hold for the sequence $\underline{x}_i$ it is thus enough to show that they hold for mutation sequences of length $1$, by passing from the directed system $F$ to the directed system $\mu_{\underline{x}'}(F)$.

%By induction it is enough to prove Claims (1)-(4) for sequences of length $1$: This follows for Claims (1), (2) and (3) by \cite[Proposition~3.5]{Gcolimits} and holds for Claim (4) by Claim (3). 
%For Claim (4), assume Claims (1)-(4) are true for sequences of length $1$ and consider our $\Sigma_i$-admissible sequence $\underline{x}_i = (x^i_1, \ldots, x^i_\ell)$ with $f_i(\underline{x}_i) = (x_1, \ldots, x_\ell)$. By Claim (3) we have a canonical isomorphism $\cA(\mu_{x_{\ell-1}} \circ \cdots \circ \mu_{x_1}(\Sigma(F))) \cong \cA(\Sigma(\mu_{x_{\ell-1}} \circ \cdots \circ \mu_{x_1}(F)))$ given by the identity map on $\mu_{x_{\ell-1}} \circ \cdots \circ \mu_{x_1}(\X)$. For all $y_i \in \X_i$ we therefore have $f_i(\mu_{x^i_{\ell-1}} \circ \cdots \circ \mu_{x^i_1} (y_i)) = \mu_{x_{\ell-1}} \circ \cdots \circ \mu_{x_1}(f_i(y_i))$ and the Claim (4) now follows from our assumption for the directed system $\mu_{x_{\ell-1}} \circ \cdots \circ \mu_{x_1}(F)$, since $(x_\ell)$ is a $\mu_{x^i_{\ell-1}} \circ \cdots \circ \mu_{x^i_1}(\Sigma_i)$-biadmissible sequence of length $1$.

 Consider thus a length one $\Sigma_i$-admissible sequence $\underline{x}_i = (x_i)$ for some $x_i \in \ex_i$. For all $j \geq i$, set $f_{ij}(x_i) = x_j$ and $f_i(x_i) = x$. We have 
\[
	x_j = f_{ij}(x_i) \in f_{ij}(\ex_i) \subseteq \ex_j \cup \bZ.
\]
Since $f_j(x_j) =  f_i(x_i) \notin \bZ$, we must have $x_j \in \ex_j$ as desired. This proves Claim (1). Claim (2) is a direct consequence of Claim (1) combined with \cite[Proposition~3.5]{Gcolimits}. 

It remains to prove Claims (3) and (4). Consider the set of $B$-neighbours $N_B(x)$ of $x$. This set is finite by Lemma \ref{L:ex-locally finite}, and thus gives rise to a finite submatrix $A$ of $B$. By Lemma \ref{L:attain finite submatrix}, the submatrix $A$ is uniformly attained at some $j \geq i$. Hence there exists a $\sigma \in \{1, -1\}$ such that for all $v \in N(x)$ with $b_{uv} \neq 0$ we have
\[
	b_{uv} = \sigma b^j_{u_j v_j}
\]
for a unique $v_j \in \X_j$. Without loss of generality we will assume that $\sigma = 1$, the case $\sigma=-1$ follows symmetrically. For all $x_i \neq y_i \in \X_i$ with $f_{i}(y_i) = y$ we have $x \neq y$, since otherwise $f_{ij}(x_i) = f_{ij}(y_i)$ for some $j \geq i$ contradicting Lemma \ref{L:ws}. Therefore we have
\[
	f_i(\mu_{x_i}(y_i)) = f_i(y_i) = y = \mu_{x}(x).
\]
On the other hand, we have
\begin{eqnarray*}
	f_i(\mu_{x_i}(x_i)) &=& f_j(\mu_{x_j}(x_j)) \\
	& = &  \frac{\prod_{u_j \in \X_j \colon b_{x_ju_j} > 0} f_j(u_j)^{b_{x_ju_j}} + \prod_{v_j \in \X_j \colon b_{x_jv_j} < 0} f_j(v_j)^{-b_{x_jv_j}}}{f_j(x_j)}.
\end{eqnarray*}
Consider the sets $(\X_j)_{\bZ}$ and $(\X_j)_{cl}$, where 
\[
	(\X_j)_{\bZ}= \{v_j \in \X_j \mid f_j(v_j) \in \bZ\} \; \; \text{ and } (\X_j)_{cl}= \{v_j \in \X_j \mid f_j(v_j) \in \X\}.
\]
There exists a $k \geq j$ such that for all $v_j \in (\X_j)_{\bZ}$ we have $f_{jk}(v_j) = f_j(v_j) \in \bZ$. 
By Lemma \ref{L:specialisation of neighbours} we obtain
\begin{eqnarray*}
	f_i(\mu_{x_i}(x_i)) &=& f_j(\mu_{x_j}(x_j)) = f_k \circ f_{jk}(\mu_{x_j}(x_j))\\
	& = &  \frac{\prod_{u_j \in (\X_j)_{cl} \colon {b^j}_{x_ju_j}>0} f_j(u_j)^{b_{x_ju_j}} + \prod_{v_j \in (\X_j)_{cl} \colon b^j_{x_jv_j}<0} f_j(v_j)^{-b_{x_jv_j}}}{f_j(x_j)} \\
	& = &  \frac{\prod_{u \in \X \colon b_{xu}>0} u^{b_{xu}} + \prod_{v \in \X \colon b_{xv}<0} v^{-b_{xv}}}{x} \\
	&=& \mu_{x}(x).
\end{eqnarray*}
%Consider the sets $S_j^+$ and $S_j^-$, where 
%\[
%	S_j^\pm = \{v_j \in \tilde{X}_j \mid \pm b^j_{v_jx^j_\ell} > 0 \text{ and } f_j(v_j) \in \bZ\},
%\]
%and their counterparts $T_j^{\pm} = \{v_j \in \tilde{X}_j \mid \pm b^j_{v_jx^j_\ell} > 0 \text{ and } f_j(v_j) \in \X\}$, so that $N_{B_j}(x_j) = S_j^+ \sqcup T_j^+ \sqcup S_j^- \sqcup T_j^-$.
%Then there exists a $k \geq j$ such that for all $v_j \in S^+_j \cup S^-_j$ we have $f_{jk}(v_j) = f_j(v_j) \in \bZ$. For all $k \geq j$, by Lemma \ref{L:specialisation of neighbours} we have
%\[
%	\prod_{v_j \in S^\pm_j}f_j(v_j)^{b^j_{v_jx_j}} = f_k \circ f_{jk} \left( \prod_{v_j \in S^\pm_j} v_j^{b^j_{v_jx_j}} \right) = 1.
%\]
%Returning to our computation of $f_i(\mu_{x_i}(x_i))$, since $A$ is uniformly attained at $j$, we obtain
%\begin{eqnarray*}
%	f_i(\mu_{x_i}(x_i)) &=& f_j(\mu_{x_j}(x_j))\\
%	& = &  \frac{\prod_{u_j \in S_j^+} f_j(u_j)^{b_{u_jx_j}}\prod_{u_j \in T_j^+} f_j(u_j)^{b_{u_jx_j}} + \prod_{v_j \in S_j^-} f_j(v_j)^{-b_{v_jx_j}}\prod_{v_j \in T_j^-} f_j(v_j)^{-b_{v_jx_j}}}{f_j(x_j)} \\
%	& = &  \frac{\prod_{u_j \in T_j^+} f_j(u_j)^{b_{u_jx_j}} + \prod_{v_j \in T_j^-} f_j(v_j)^{-b_{v_jx_j}}}{f_j(x_j)} \\
%	& = &  \frac{\prod_{u \in \X \colon b_{ux}>0} u^{b_{ux}} + \prod_{v \in \X \colon b_{vx}<0} v^{-b_{vx}}}{x} \\
%	&=& \mu_{x}(x).
%\end{eqnarray*}
This proves Claim (4), and moreover, it proves part of Claim (3), namely that $\mu_x(\X) = \colim \mathbb{X} \circ \mu_{x}(F) \setminus \bZ$, and $\mu_x(\ex) = \colim \mathbb{E} \circ \mu_{x}(F) \setminus \bZ$. Hence, to finish our proof it remains to show that the entries of $\mu_{x}(B) = B' = (b'_{uv})_{u \in \mu_x(\ex),v \in \mu_{x}(\X)}$ and of the exchange matrix $\tilde{B} = (\tilde{b}_{uv})_{u \in \mu_x(\ex),v \in \mu_x(\X)}$ of $\cA(\mu_x(F))$ agree %on rows and columns labelled by exchangeable variables, 
up to a consistent sign on exchangeably connected components. By construction, this is the case on all exchangeably connected components {\em not} containing $x \in \ex$, so it remains to check it on the exchangeably connected component containing $x$. Let $u \in \mu_x(\ex)$ and $v \in \mu_x(\X)$. For all $\ell \geq j$ set $\mu_{x_\ell}(B_\ell) = ({b'}^\ell_{uv})$.  By definition of matrix mutation, we have
\[
	b'_{uv} = 	\begin{cases}
				-b_{xv} & \text{if $u =\mu_{x}(x)$}\\
				-b_{ux} & \text{if $v =\mu_{x}(x)$}\\
				b_{uv} + \frac{1}{2}(|b_{ux}|b_{xv} + b_{ux}|b_{xv}|) & \text{otherwise.}
			\end{cases}
\]
The finite submatrix $A$ of $B$ with rows and columns labelled by $N_B(x) \cup \{x\}$ is uniformly attained at $j > 0$. Therefore, whenever the relevant entries are defined in $B$, for every $\ell \geq j$ there exists a $\sigma_\ell \in \{+1, -1\}$ and a unique $x_\ell$, $u_\ell$ and $v_\ell$ with $b_{uv} = \sigma_\ell b^\ell_{u_\ell v_\ell}$, $b_{ux} = \sigma_\ell b^\ell_{u_\ell x_\ell}$, and $b_{xv} = \sigma_\ell b^\ell_{x_\ell v_\ell}$. 
%Note that, while $x_\ell$ and $u_\ell$ are the unique cluster variables in $\X_\ell$ getting respectively mapped to $x$ and $u$ under $f_\ell$, we might a priori have two distinct cluster variables $v_\ell$ and $v'_\ell$ mapping to $v$, if the variables in question are frozen. However, since $f_\ell(v_\ell) = f_\ell(v'_\ell)$ there exists an $m \geq \ell$ such that for all $k \geq m$ we have $f_{\ell k }(v_\ell) = f_{\ell k} (v'_\ell) = v_k$ and setting $f_{\ell k }(x_\ell) = x_k$ and $f_{\ell k}(u_\ell) = u_k$, by Lemma \ref{L:well-defined-seed} we have $b_{uv} = \sigma_k b^k_{u_k v_k}$, $b_{ux} = \sigma_k b^k_{u_k x_k}$ and $b_{xv} = \sigma_k b_{x_k v_k}$ for some $\sigma_k \in \{+1, -1\}$. In the case where $u = \mu_{x}(x)$, respectively $v = \mu_{x}(x)$, we set $u_\ell = \mu_{x_\ell}(x_\ell)$, respectively $v_\ell = \mu_{x_\ell}(x_\ell)$. For all $k \geq m$ we obtain
We obtain
\begin{eqnarray*}
	\sigma_\ell b'_{uv} &=& 	\begin{cases}
				-b^\ell_{x_\ell v_\ell} & \text{if $u =\mu_{x}(x)$}\\
				-b^\ell_{u_\ell x_\ell} & \text{if $v =\mu_{x}(x)$}\\
				b^\ell_{u_\ell v_\ell} + \frac{1}{2}(|b^\ell_{u_\ell x_\ell}|b^\ell_{x_\ell v_\ell} + b^\ell_{u_\ell x_\ell}|b_{x_\ell v_\ell}|) & \text{otherwise.}
			\end{cases}	\\
			& = & {b'}^\ell_{u_\ell v_\ell}.
\end{eqnarray*}
This is, up to a sign depending on the choice of sign-determining entry for the $\tilde{B}$-exchangeably connected component containing $x$, equal to the entry $\tilde{b}_{uv}$ of the exchange matrix $\tilde{B}$ of $\Sigma(\mu_x(F))$, as desired.
\end{proof}

An immediate consequence is that every cluster variable in $\cA(\Sigma(F))$ comes from a cluster variable in the filtered system.

\begin{corollary}\label{C:big-admissible}
	For every $\Sigma(F)$-admissible sequence $\underline{x} = (x_1, \ldots, x_\ell)$ there exists an $i \in \cF$ such that for all $j \geq i$ there exists a unique $\Sigma_j$-admissible sequence $\underline{x}_j = (x_1^j, \ldots, x_\ell^j)$ with $f_j(\underline{x}_j) = \underline{x}$. For $k \geq j \geq i$ we have $f_{jk}(\underline{x}_j) = \underline{x}_k$. In particular, for every cluster variable $y \in \cA(\Sigma(F))$ there exists an $i \in \cF$ such that for all $k \geq j \geq i$ we have $y = f_j(y_j)$ for some cluster variable $y_j \in \cA(\Sigma_j)$ and $f_{jk}(y_j) = y_k$.
\end{corollary}

\begin{proof}
	The first part of the claim is trivially true if $\ell = 0$. Assume thus it holds for sequences of lengths at most $\ell -1$. By Proposition \ref{P:cone} we have a strong similarity $\Sigma(\mu_{x_{\ell-1}} \circ \cdots \circ \mu_{x_1}(F)) \cong \mu_{x_{\ell-1}} \circ \cdots \circ \mu_{x_1}(\Sigma(F))$. The cluster variable $x_\ell$ thus is exchangeable in the seed $\Sigma(\mu_{x_{\ell-1}} \circ \cdots \circ \mu_{x_1}(F))$. By Lemma \ref{L:well-defined-seed} there exists an $i \in \cF$ such that for all $j \geq i$ there exists a unique $x_\ell^j \in \X_j$ with $f_j(x_\ell^j) = x_\ell$. This proves the first part of the claim. The second part follows immediately.
\end{proof}

\begin{remark}\label{R:uniform attainment preserved under mutation}\rm
	Let $B$ be the exchange matrix of $\Sigma(F)$, and let $x \in \ex$. Consider the finite submatrix $\tilde{B}$ of $B$ with rows and columns labelled by $N_B(x) \cup \{x\}$. If $\tilde{B}$ is attained at $\ell$, then so is the finite submatrix of $\mu_{\underline{x}}(B)$ with rows and columns labelled by $N_B(x) \cup \{\mu_{x}(x)\}$. This is an immediate consequence of the proof of Proposition \ref{P:cone}(3).
\end{remark}

We can now prove one of our main results.

\begin{theorem}\label{T:directed_closed}
 The category $\mCl$ of inducible melting cluster morphisms is closed under directed colimits. More precisely, given a directed system $F \colon \cF \to \mCl$, we have $\colim F \cong \cA(\Sigma(F))$.
\end{theorem}

\begin{proof}

By Proposition \ref{P:cone} it is enough to show that $\cA(\Sigma(F))$ with the co-cone maps $f_i \colon \cA(\Sigma_i) \to \cA(\Sigma(F))$ is universal. Set $\Sigma(F) = (\X,\ex,B)$ and assume that $\cA(\Sigma')$ with $\Sigma' = (\X',\ex',B')$ is also a co-cone via inducible melting roooted cluster morphisms $g_k \colon \cA(\Sigma_k) \to \cA(\Sigma')$, such that $g_j \circ f_{ij} = g_i$ for all $i \leq j$. 

Consider the restriction of maps $g_i \colon \X_i \cup \bZ \to \X' \cup \bZ$ in the category of sets. Since $\X \cup \bZ$ is defined to be the colimit of $\bX \circ F$, it follows that there is a unique map $\varphi \colon \X \cup \bZ \to \X' \cup \bZ$ such that $\varphi \circ f_i = g_i$ for all $i \in \cF$, which must restrict to the universal map $\varphi \colon \ex \cup \bZ \to \ex' \cup \bZ$. Moreover, for all $x \in \ex$, there exists some $i \in \cF$ and an $x_i \in \ex_i$ with $f_i(x_i) = x$, so we must have $\varphi(x) = g_i(x_i) \in \ex' \cup \bZ \setminus \{0\}$. The map $\varphi$ extends to a unique ring homomorphism $\cA(\Sigma(F)) \to \cA(\Sigma')$ which satisfies CM1 and iMCM.

It remains to check that $\varphi$ satisfies CM2. Let thus $\underline{x}$ be a $\varphi$-biadmissible sequence, and let $y \in \X$. By Corollary \ref{C:big-admissible} there exists an $i \in \cF$ such that for all $j \geq i$ we have a unique $\Sigma_j$-admissible sequence $\underline{x}_j$ with $f_j(\underline{x}_j) = \underline{x}$. In particular, this sequence $\underline{x}_j$ is $g_i$-biadmissible. Moreover, we can pick $i$ big enough such that $y = f_i(y_i)$ for some $y_i \in \X_i$. Using that $f_i$ and $g_i$ satisfy CM2, we obtain
\begin{eqnarray*} 
	\varphi(\mu_{\underline{x}}(y)) &=& \varphi(\mu_{f_i(\underline{x}_i)}(f_i(y_i)))
	= \varphi \circ f_i (\mu_{\underline{x}_i}(y_i)) \\
	&=& g_i(\mu_{\underline{x}_i}(y_i))	
	= \mu_{g_i(\underline{x}_i)}(g_i(y_i))\\
	&=& \mu_{\varphi \circ f_i(\underline{x}_i)}(\varphi \circ f_i(y_i))
	= \mu_{\varphi(\underline{x})}(\varphi(y)).
\end{eqnarray*}
\end{proof}

Every cluster of $\cA(\Sigma(F))$ can be constructed via a colimit of clusters in the directed system.

\begin{remark}\label{R:cluster_colimits} \rm
    Assume $F \colon \cF \to \mCl$ is directed. By Corollary \ref{C:big-admissible}, every cluster of $\cA(\Sigma(F))$ can be viewed as a directed colimit of clusters in $\cA(\Sigma_i)$: If $\X$ is the initial cluster, and $\tilde{\X}$ is any cluster of $\cA(\Sigma(F))$, then there exists a $\Sigma(F)$-admissible sequence $\underline{x}$ such that $\mu_{\underline{x}}(\X) = \tilde{\X}$. By Corollary \ref{C:big-admissible} we have an $i \in\cF$ such that for $j \geq i$ there exists a unique $\Sigma_i$-admissible sequence $\underline{x}_i$ with $f_i(\underline{x}_i) = \underline{x}$. For $j \geq i$ set $\tilde{\X}_j = \mu_{\underline{x}_j}(\X_j)$. We obtain a directed system $\tilde{\bX} \colon \cF \to \mathsf{Set}$ with $\tilde{\bX}(i) = \tilde{\X}_i \cup \bZ$ and $\bX(j \geq i) = f_{ji} \mid_{\tilde{X}_i \cup \bZ}$, of which $\tilde{\X} \sqcup \bZ \subseteq \cA(\Sigma(F))$ is the colimit. That is, we have $\tilde{\X} = \bigcup_{j \geq i} \tilde{\X}_i / \sim$, where $x_j \sim x_k$ for $x_j \in \X_j$ and $x_k \in \X_k$ if there exists $\ell \geq j,k$ such that $f_{k\ell}(x_k) = f_{j\ell}(x_j)$.
\end{remark}

It is useful to observe that a filtered colimit in $\mCl$ is also a filtered colimit of rings. In particular, we will make use of this when discussing the coordinate ring of the Sato-Segal-Wilson Grassmannian as an example of an ind-cluster algebra.

\begin{proposition}\label{P:forgetful_commutes}
    The forgetful functor $G \colon \mCl \to \mathsf{Ring}$ commutes with filtered colimits.
\end{proposition}

\begin{proof}
    By \cite{AndrekaNemeti} it is enough to show it commutes with directed colimits. Consider a directed system $F \colon \cF \to \mCl$ with $F(i) = \cA(\Sigma_i)$ for all $i \in \cF$. Then the ring $G(\colim F) = \cA(\Sigma(F))$ is a co-cone of $G \circ F$. Let $R$ be a co-cone of $G \circ F$  with maps $g_i \colon \cA(\Sigma_i) \to R$. The ind-cluster algebra $\cA(\Sigma(F))$ is generated, as a ring, by its cluster variables, and by Corollary \ref{C:big-admissible} for every cluster variable $y \in \cA(\Sigma(F))$ there exists an $i \in \cF$ such that for all $k \geq j \geq i$ we have $y = f_j(y_j)$ for some cluster variable $y _j$ with $f_{jk}(y_j) = y_k$. Setting $\varphi(y) = g_i(y_i)$ yields the desired universal map $\varphi \colon \cA(\Sigma(F)) \to R$ in $\mathsf{Ring}$.
\end{proof}

%\begin{corollary}\label{C:creates}
%	The initial cluster functor $\bX \colon \mCl \to \mathrm{Set}$ commutes with filtered colimits.
%\end{corollary}
%
%\begin{proof}
%	By \textcolor{red}{cite AN} it is enough to verify this for directed colimits and this is true by construction.
%\end{proof}

\subsection{Compact objects and ind-objects}

We denote by $\mCl^{\mathrm{f}}$ the full subcategory of $\mCl$ whose objects are rooted cluster algebras of finite rank. Analogously to the way finite sets control the category of all sets (finite and infinite) via filtered colimits, finite rank cluster algebras control $\mCl$. In this instance, one talks about {\em compact objects}.

\begin{theorem}\label{T:compact}
	The compact objects in $\mCl$ are precisely the rooted cluster algebras of finite rank.
\end{theorem}

\begin{proof}
	Let $\cA(\Sigma)$ be a rooted cluster algebra of finite rank with $\bX(\cA(\Sigma)) = \X \sqcup \bZ$. Let $F \colon \cF \to \mCl$ be a filtered system with $F(i) = \cA(\Sigma_i = (\X_i,\ex_i,B_i))$. We must show that the natural map
	\[
		\colim	\Hom(\cA(\Sigma),\cA(\Sigma_i)) \to \Hom(\cA(\Sigma), \colim F)
	\]
is an isomorphism. Injectivity follows from the fact that %the initial cluster functor $\bX$ commutes with filtered colimits, that 
every morphism in $\mCl$ is determined by its value on the initial cluster, and that finite sets are compact in the category of sets.

It remains to show surjectivity. Set $\colim F = \cA(\tilde{\Sigma} = (\tilde{\X}, \tilde{\ex}, \tilde{B}))$ with colimit maps $f_i \colon \cA(\Sigma_i) \to \cA(\tilde{\Sigma})$ and consider an inducible melting cluster morphism $f \colon \cA(\Sigma) \to \cA(\tilde{\Sigma})$. This yields a map $\bX(f) \colon \bX(\cA(\Sigma)) \to \bX(\cA(\tilde{\Sigma})) = \tilde{\X} \sqcup \bZ$, which restricts to a map $\bX(f) \mid_{\X} \colon \X \to \bX(\tilde{\Sigma})$. We have $\bX(\tilde{\Sigma}) \cong \colim \bX \circ F$, and since $\X$ is finite and thus compact in the category of sets, there exists an $\ell \in \cF$ such that for all $i \geq \ell$ the map $\bX(f) \mid_{\X}$ factors through $\X_i \cup \bZ$, say via $g_i \colon \X \to \X_i \cup \bZ$, so that $\bX(f)\mid_{\X} = f_i \circ g_i \mid_{\X}$. Now there is some choice for this map $g_i$, and we want to execute this carefully to make sure it is compatible with our cluster structures. Consider the finite set 
\[
	S_f = \left(f(\X) \cap \tilde{\X}\right) \cup \left(N_{\tilde{B}}(f(\X) \cap \tilde{\ex})\right).
\] 
This is the set of cluster variables in $f(\X)$ together with all of the neighbours of exchangeable variables in $f(\X)$. Consider the finite submatrix $\tilde{B}_f$ of $\tilde{B}$ with rows and columns labelled by $S_f \cap \tilde{\ex}$ and $S_f$ respectively.
Let $\ell \in \cF$ be such that $\tilde{B}_f$ is uniformly attained at $\ell$. %, and for all $\tilde{x} \in S_f$ let $\tilde{x}$ be $\tilde{B}$-attained by $x_\ell \in \X_\ell$.
% Let $\ell \in \cF$ be such that $\tilde{B}_f$ is uniformly attained at $\ell$ and such that for all $\tilde{x} \in S_f$ we find a cluster variable $x_\ell \in \X_\ell$ with $f(x_\ell) = \tilde{x}$, such that if $\tilde{x} \in \tilde{\ex}$ then also $x_\ell \in \ex_i$. Note the latter is automatically true for all $\tilde{x} \in S_f \cap \ex$ which have a $\tilde{B}_f$-neighbour, by the stipulation that $\tilde{B}_f$ is uniformly attained at $\ell$, but might need to be adjusted by a sufficiently larger choice of $\ell$ if $\tilde{x}$ is isolated in $\tilde{B}_f$. 
Let from now on $i \geq \ell$ and for $x \in \X$ with $f(x) \in \tilde{\X}$, let $f(x)$ be $\tilde{B}_f$-attained by $x_i \in \X_i$.  We define
\[
	g_i \colon \X \to \X_i \cup \bZ
\]
for all $x \in \X$ by 
\[
	g_i(x) = \begin{cases}
						f(x) & \text{ if $f(x) \in \bZ$} \\
						x_i & \text{ if $f(x) \in \tilde{\X}$}.
				\end{cases}
\]
Then we have $f_i \circ g_i \mid_{\X} = f \mid_{\X}$. For all $x \in \ex$ we have $f(x) \neq 0$ and thus also $g_i(x) \neq 0$. Therefore, by the Laurent phenomenon, the map $g_i$ extends to a unique ring homomorphism
\[
	g_i \colon \cA(\Sigma) \to \cA(\Sigma_i).
\]
It satisfies CM1 and iMCM. It remains to show that $g_i$ satisfies CM2. By induction over $\ell$, we show that $g_i$ commutes with mutation along biadmissible sequences of length $\ell$. This is trivially true for sequences of length $0$. Now assume it holds for sequences of length $\ell-1$. Let $\underline{x} = (x_1, \ldots, x_\ell)$ be $g_i$-biadmissible and consider its truncation $\underline{x}' = (x_1, \ldots, x_{\ell-1})$ of length $\ell-1$. By Proposition \ref{P:cone} we get a directed system $\mu_{f(\underline{x}')}(F)$ with colimit $\cA(\mu_{f(\underline{x}')}(\Sigma(F)))$. By Remark \ref{R:uniform attainment preserved under mutation} the finite submatrix $\mu_{f(\underline{x}')}(\tilde{B}_f)$ of the exchange matrix of $\mu_{f(\underline{x}')}(\Sigma)$ is uniformly attained at $\ell$. To show that the claim holds for $\underline{x}$ it is thus enough to show that it holds for sequences of length $1$, by passing from $\cA(\Sigma)$ to $\cA(\mu_{\underline{x}'}(\Sigma))$ and from $F$ to $\mu_{f(\underline{x}')}(F)$.%By \cite[Proposition~3.5]{Gcolimits} and Remark \ref{R:uniform attainment preserved under mutation} it suffices to check that $g_i$ commutes with mutation along $g_i$-biadmissible sequences of length $1$.

Let thus $(x)$ be $g_i$-biadmissible, that is $x \in \ex$ and $g_i(x) = x_i \in \ex_i$. Then also $f(x) \in \tilde{\ex}$ and $f(x)$ is $\tilde{B}_f$-attained by $x_i$. Let next $u_i = g_i(u) \in \X_i$ for some $u \in \X$. Then $f(u) = \tilde{u} \in S_f$ and $\tilde{u}$ is attained by $u_i$. By (\ref{ws2}) of Lemma \ref{L:ws} applied to the melting cluster morphism $f$ for all $u_i \in g(\X) \cap \X_i$ we obtain
\begin{eqnarray}\label{E:compact}
	%b^i_{x_i u_i} = \tilde{b}_{f(x)\tilde{u}} = \sigma \sum_{f(u') = \tilde{u}} b_{xu'} \; \; \text{ and } 
b^i_{x_i u_i } = \tilde{b}_{f(x)\tilde{u}} = \sigma \sum_{f(u') = \tilde{u}} b_{xu'} = \sigma \sum_{g_i(u') = u_i} b_{xu'},
\end{eqnarray}
and moreover, in the sum either all summands are positive or all summands are negative. Without loss of generality assume $\sigma = 1$.
We partition $\X$ as $\X = \X_{cl} \sqcup \X_{\bZ}$, where $\X_{cl} = \{x \in \X \mid f(x) \in \tilde{\X}\}$ and $\X_{\bZ} = \{x \in \X \mid f(x) \in \bZ\}$. We obtain
\begin{eqnarray*}
	g_i(\mu_{x}(x)) &=& \frac{\prod_{u \in \X \colon b_{xu} > 0}g_i(u)^{b_{xu}} + \prod_{v \in \X \colon b_{xv} < 0}g_i(v)^{-b_{xv}}}{g_i(x)} \\
	 &=& \frac{\prod_{u \in \X_{cl} \colon b_{xu} > 0}g_i(u)^{b_{xu}} + \prod_{v \in \X_{cl} \colon b_{xv} < 0}g_i(v)^{-b_{xv}}}{x_i} \\
	%&=& \frac{\prod_{u_i = g_i(u) \colon b_{xu} > 0}u_i^{\sum_{g_i(u') = u_i} b_{xu'}} + \prod_{v_i = g_i(v) \colon b_{xv} < 0}g_i(v)^{-\sum_{g_i(v') = v_i} b_{xv'}}}{x_i} \\
	&=& \frac{\prod_{u_i = g_i(u) \colon b_{xu} > 0}u_i^{\sum_{g_i(u') = u_i} b_{xu'}} + \prod_{v_i = g_i(v) \colon b_{xv} < 0}v_i^{-\sum_{g_i(v') = v_i} b_{xv'}}}{x_i}\\
	&=& \frac{\prod_{u_i \in \X_i \colon b^i_{x_iu_i } > 0}u_i^{b^i_{x_iu_i }} + \prod_{v_i \in \X_i \colon b^i_{x_iv_i } < 0}v_i^{- b^i_{x_iv_i }}}{x_i}\\
	&=& \mu_{x_i}(x_i) = \mu_{g_i(x)}(g_i(x)).
\end{eqnarray*}
where the first and penultimate equalities are by definition of mutation, the second is by Lemma \ref{L:specialisation of neighbours}, the third is just a reformulation using that, if $g_i(u) = g_i(u')$ then $b_{xu}$ and $b_{xu'}$ have the same sign, and the fourth is by (\ref{E:compact}). Therefore $\cA(\Sigma)$ is compact.

To show that a cluster algebra of infinite rank is not compact, consider a seed $\Sigma = (\X, \ex,B)$ with $|\X| = \infty$, and the identity map $\mathrm{id}_{\cA(\Sigma)} \colon \cA(\Sigma) \to \cA(\Sigma)$. By Theorem \ref{T:colimit} we can write $\cA(\Sigma)$ as a filtered colimit of finite rank cluster algebras. %, say $\cA(\Sigma) \cong \colim F$ for some $F \colon J \to \mCl^{\mathrm{f}}$, with $F(i) = \cA(\Sigma_i = (\X_i,\ex_i,B_i))$ for every $i \in J$. 
If $\mathrm{id}_{\cA(\Sigma)}$ would factor through one of these finite rank cluster algebras $\cA(\Sigma_i = (\X_i,\ex_i,B_i))$, then in particular the identity map restricted to the initial clusters, $\mathrm{id}_{\X} \colon \X \to \X$ would factor through the finite set $\X_i$. This is absurd.
\end{proof}

\begin{corollary}\label{C:ind-objects}
	The ind-objects of $\mCl^{\mathrm{f}}$ are precisely the cluster algebras of infinite rank.
\end{corollary}

\begin{proof}
	We have a fully faithful embedding $\mCl^{\mathrm{f}} \to \mCl$, and $\mCl$ has filtered colimits by Theorem \ref{T:filteredclosed}. By Theorem \ref{T:compact} the finite rank cluster algebras are the compact objects in $\mCl$. By Theorem \ref{T:colimit} all objects in $\mCl$ are filtered colimits of compact objects. The claim follows by \cite[Corollary~6.3.5]{KashiwaraSchapira}.
\end{proof}

Motivated by Corollary \ref{C:ind-objects}, we will refer to cluster algebras of infinite rank as {\em ind-cluster algebras}. This fits seamlessly into the classical definition of a cluster algebra, which asserts that the initial cluster of a `true' cluster algebra is to be finite.

%%%%%%%%%%%%%%%%%%%%%%%%%%%%%%%%%%%%%%%%%%%%
\section{The coordinate ring $\iring$ is an ind-cluster algebra}

%In Sato's seminal papers \cite{sato1981soliton} and \cite{sato1983soliton} the solutions of the KP-hierarchy are connected to infinite-dimensional Grassmann manifolds, now often simply referred to as `the Sato Grasmannian', whose points are described in terms of infinite-dimensional matrices. Closely related infinite Grassmannians were studied by Segal and Wilson \cite{segal1985loop} and later by Pressley and Segal \cite{pressley1985loop}. 

Recall the definition of $\Gr$ as stated in Def \ref{def:Gr} in the introduction which follows the conventions in \cite{segal1985loop} and \cite{pressley1985loop}. Our main focus though shall be on its respective coordinate ring as first introduced in Sato's seminal papers \cite{sato1981soliton} and \cite{sato1983soliton} in connection with the solutions of the KP-hierarchy.

Fioresi and Hacon investigate in \cite{FH-Sato} a ring $\iring$, which they argue can be viewed as the `coordinate ring' for the `Sato Grassmannian', in the sense that there is a bijection between closed points of $\Proj(\iring)$ and points of $\Gr$ whose coordinates satisfy the union of all Pl\"ucker relations of the finite Grassmannians; see  \cite[Proposition~2.10]{FH-Sato}. This is the same ring as discussed in \cite{sato1983soliton} using the Pl\"ucker embedding; see also the discussion of the Pl\"ucker embedding of $\Gr$ in \cite[\S8,\S10]{segal1985loop} and \cite[\S7.5]{pressley1985loop} where additional convergence conditions on Pl\"ucker coordinates are imposed. 

We show in Section \ref{S:Sato-ind-cluster} that the ring $\iring$ has a natural structure of an ind-cluster algebra, induced by the classical finite rank Grassmannian cluster algebra structures introduced by Scott \cite{Scott}. One of the applications of our definition and description of the ind-cluster algebra structure of this coordinate ring is that it identifies among the infinite set of equations making up the KP-hierarchy 
natural sets of relations for the $\tau$-function, the solution of the KP-hierarchy, which generate all the remaining ones. The mutations of its clusters provide a natural combinatorial framework to encode the algebraic dependence among the infinite relations of the KP-hierarchy. In fact, the KP-equation itself can be seen as a mutation in the coordinate ring $\iring$.

\subsection{The coordinate ring of the Sato-Segal-Wilson Grassmannian}\label{S:Sato-ind-cluster}

\begin{figure}\label{fig:BFcorr}
\centering
\includegraphics[width=1.\textwidth]{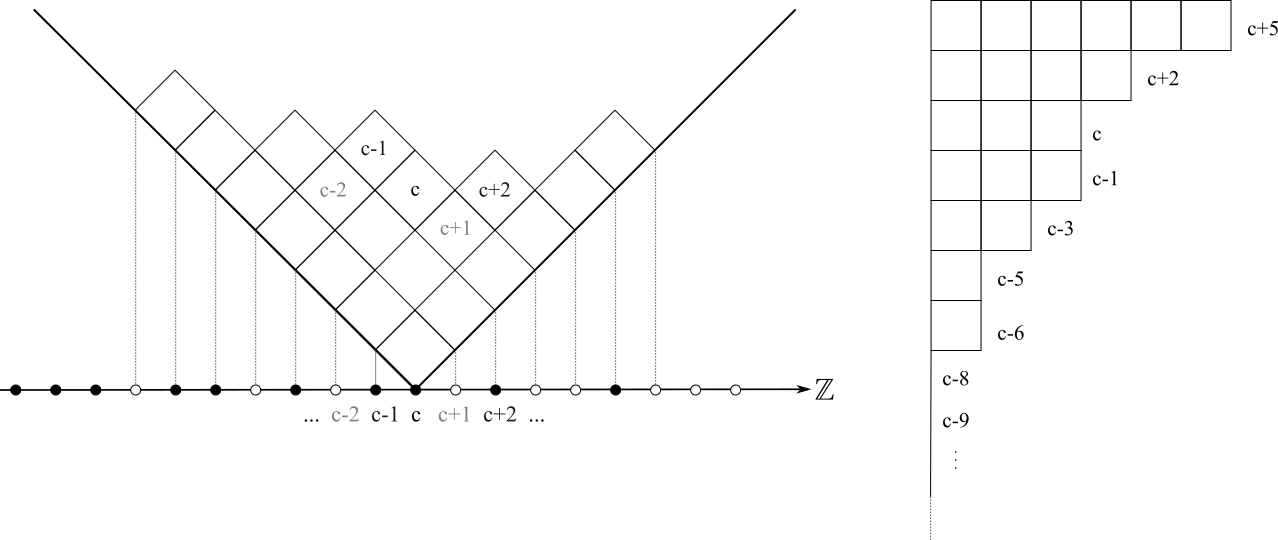} 
\caption{A graphical depiction of the bijection between Maya sequences and Young diagrams: each black go-stone in the left half of the Figure corresponds to a $45^\circ$ downward step and a white go-stone to a $45^\circ$ degree upward step when drawing the outline of a Young diagram. In the example shown the partition is $\lambda=(6,4,3,3,2,1)$ and the corresponding Maya sequence of charge $c$ is $a_\bullet=(c+5,c+2,c,c-1,c-3,c-5,c-6,c-8,c-9,\ldots)$, which are just the positions of the black go-stones on the left.}
\end{figure}

%%%%%%%%%%%%%%%%%%%%%%%%%%%%%%%%%%%%%%%%%%%%%%%%%%%%%%%%%%%

We recall the construction of $\iring$ from \cite{FH-Sato}. Denote by $\Gr_{m,n}$ the Grassmannian of $m$-dimensional subspaces in $\bC^{m+n}$, and by $\ring$ the homogeneous coordinate ring of $\Gr_{m,n}$ viewed as a projective variety under the Pl\"ucker embedding. The $\bC$-algebra $\ring$ is generated by the Pl\"ucker variables 
\[
	\{d_{j_1, \ldots, j_m} \mid -m \leq j_1 < \ldots < j_m \leq n-1\}\;.
\]
The latter are the minors obtained by taking the rows $-m,\ldots,-1$ and columns $j_1,\ldots,j_m$ in a $m\times (m+n)$ matrix which corresponds to a point in $\Gr_{m,n}$. It is convenient to extend the definition to any $m$-tuple $I$ from $[-m,n-1]=\{-m, \ldots, n-1\}$ by setting for any permutation $\sigma$ in the symmetric group $\mathfrak{S}_m$
\[
    d_{i_{\sigma(1)}, \ldots, i_{\sigma(m)}} = \begin{cases}
        \mathrm{sign}(\sigma) d_{i_1, \ldots, i_k} & \text{if } i_1 < \ldots < i_m\\
        0 & \text{ else}
        \end{cases}\;.
\]
In what follows we will sometimes use the interval notation $[a,b]=\{a,a+1,\ldots,b\}$ for subsets of consecutive integers to ease the notation. With these conventions in place we can formulate the Pl\"ucker relations in the usual manner: For any $(m-1)$-tuple $I = (i_1, \ldots, i_{m-1})$ with $-m\le i_1 < \ldots < i_{m-1}\le n-1$ and any $(m+1)$-tuple $J = (j_1, \ldots, j_{m+1})$ with $-m\le j_1 < \ldots < j_{m+1}\le n-1$ we have the quadratic relation
\begin{equation}\label{Pluecker0}
    \sum_{\ell = 1}^{m+1} (-1)^{\ell} d_{i_1, \ldots, i_{m-1},j_\ell}d_{j_1, \ldots, j_{\ell-1}, j_{\ell+1}, \ldots, j_{m+1}}=0\;.
\end{equation}
Consider the following directed system of $\bC$-algebras: for each $m' \geq m$ and $n' \geq n$ let  $r_{m,n,m',n'} \colon \ring  \to  \bC[\Gr_{m',n'}]$ be defined as the algebraic extension of
\begin{eqnarray}\label{eqn:r-maps}
	r_{m,n,m',n'} \colon d_{j_1, \ldots, j_m} \mapsto d_{-m', \ldots, -m-1, j_1, \ldots, j_m}.
\end{eqnarray}
In words, we simply identify the Pl\"ucker coordinates $d_{j_1, \ldots, j_m}$ of the `smaller Grassmannian' $\Gr_{m,n}$ with Pl\"ucker coordinates $d_{-m', \ldots, -m-1, j_1,\ldots j_m}$ of the larger one by pre-pending the integers in the interval $[-m',-m-1]$. Alternatively, we will label below the Pl\"ucker coordinates using Young diagrams which fit inside a rectangle of height $m$ and width $n$ and the above map simply embeds the same Young diagram into a larger $m'\times n'$ rectangle.
\begin{definition}[\cite{FH-Sato}] The coordinate ring $\iring$ of the Sato-Segal-Wilson Grassmannian $\Gr$ is defined as the colimit of the above system $(\ring,r_{m,n,m',n'})$ in the category of rings, with the natural inclusions denoted by $r_{m,n} \colon \ring \to \iring$.
\end{definition}
Following \cite{FH-Sato}, the $\bC$-algebra $\iring$ is generated by the set
	\[
		\{d_{a_\bullet} \mid a_\bullet \text{ is a Maya sequence of virtual cardinality $0$}\},
	\]
the elements of which are defined as follows: a {\em Maya sequence} is a strictly decreasing integer sequence $a_\bullet = (a_i)_{i \geq 1}$ for which there exists a $j \geq 1$ and $c\in\bZ$ such that $a_{k} = c-k$ for all $k \geq j$. The smallest such $j$ is the {\em rank $||a_\bullet||$ of $a_\bullet$}.  The integer $c$ is called the {\em virtual cardinality} in the mathematics literature \cite{segal1985loop} or the {\em charge} of $a_\bullet$ in the physics literature; see e.g. \cite{miwa2000solitons}. In what follows we always choose $c=0$ unless stated explicitly otherwise. 

A Maya sequence $a_\bullet = (a_i)_{i \geq 1}$ with $||a_\bullet || \leq m+1$ and $a_1 \leq  n-1$ corresponds to a Pl\"ucker coordinate for the finite Grassmannian $\Gr_{m,n}$
\begin{equation}\label{Maya2Pluecker}
	d_{a_{\leq m}} = d_{\ell_1, \ldots, \ell_m} \in \ring
\end{equation}
where we set $\ell_i = a_{m-i+1}$. If $a_\bullet$ is such a Maya sequence with $||a_\bullet || \leq m+1$ and $a_1 \leq  n-1$ we set
\[
	d_{a_\bullet} = r_{m,n}(d_{a_{\leq m}}).
\]

\begin{remark}\label{rmk:Maya2Young}\rm
(i) We recall that Maya sequences $a_\bullet$ of charge $c$ are in bijection with Young diagrams of partitions: consider the trivial sequence $a_i=c-i$ then we map the latter to the empty partition $\varnothing$. Given a partition $\lambda=(\lambda_1\ge\lambda_2\ge\ldots\lambda_\ell\ge 0)$ we can interpret the latter as an infinite sequence by appending infinitely many zeros at the end. Define the bijection by mapping $\lambda$ to a Maya sequence of charge $c$ via 
\begin{equation}\label{Young2Maya}
\lambda\mapsto a_\bullet(\lambda)=(\lambda_i-i+c)_{i\ge 1}\;.
\end{equation}
Under this bijection, the Maya sequences of charge $c$ satisfying $||a_\bullet || \leq m+1-c$ and $a_1 \leq  c+n-1$ are then precisely those partitions whose Young diagram fits into a bounding box of height $m$ and width $n$. In terms of the Young diagram of $\lambda$ the rank $||a_\bullet(\lambda)||$ is related to the height $\ell(\lambda)$ of the diagram via $\ell(\lambda)+1-c$, while $a_1$ is the given in terms of length of first row $\lambda_1$ via $a_1=c+\lambda_1-1$; see Figure \ref{fig:BFcorr} for an example of the bijection. In what follows we interchangeably use as labels Maya diagrams or partitions for the elements in $\iring$. We shall denote by $\Pi$ the set of all partitions $\lambda$.\medskip

\noindent (ii) In what follows it will also be convenient to describe the bijection between Maya sequences of charge 0 and partitions using the Frobenius notation $\lambda=(\alpha_1,\ldots,\alpha_\ell|\beta_1,\ldots,\beta_\ell)$ of partitions, where $\alpha_i=\lambda_i-i$ is the arm length (number of boxes to the right) and $\beta_i=\lambda'_i-i$ is the leg length (number of boxes below) of the $i$th box in the diagonal of the Young diagram of $\lambda$. The connection with the Maya diagram viewed as a subset $a_\bullet(\lambda)\subset\bZ$ is then given via the following set complements:
\begin{equation}\label{Frob2Maya}
\{\alpha_1,\ldots,\alpha_\ell\}=a_\bullet(\lambda)\setminus \bZ_{<0}\qquad
\text{and}\qquad
\{-\beta_1-1,\ldots,-\beta_\ell-1\}=\bZ_{<0}\setminus a_\bullet(\lambda)\;.
\end{equation}
Considering a general Maya sequence $a_\bullet$ (with $c$ not necessarily zero) the difference of cardinalities of the two finite sets $a_\bullet\setminus \bZ_{<0}$ and $\bZ_{<0}\setminus a_\bullet$ is called  the {\em virtual cardinality} $c$ of $a_\bullet$ in \cite{segal1985loop} because the associated subspace $H_{a_\bullet}\in\Gr(H)$ spanned by the sequence $\{z^{a_i+1}\}_{i\ge 1}$ has virtual dimension $c$.
\end{remark}

\subsection{The ind-cluster structure}\label{S:Cluster structure}

\begin{figure}\label{fig:finite_quiver}
\centering
\includegraphics[width=.9\textwidth]{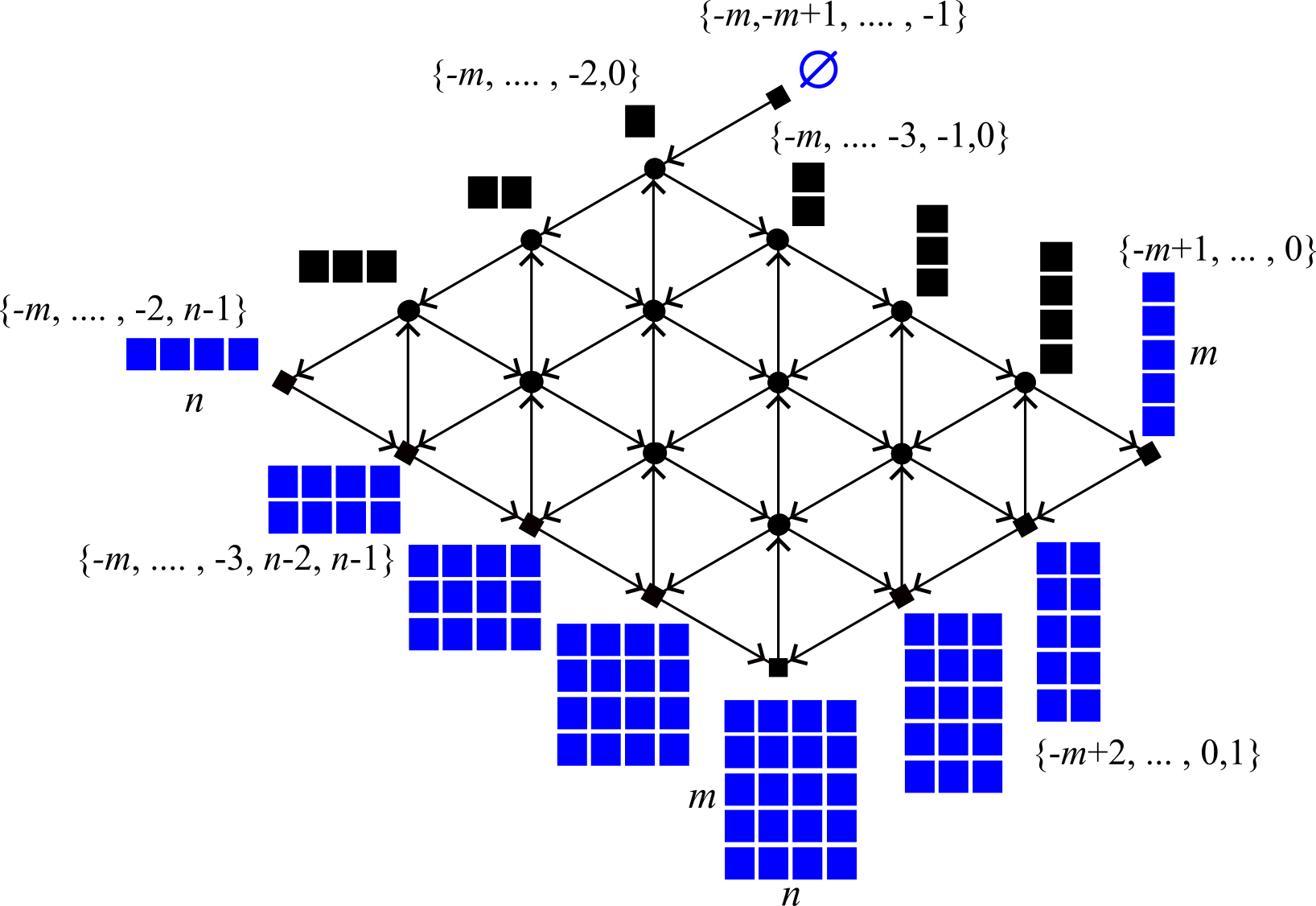} 
\caption{The quiver $Q_{m,n}$ for the coordinate ring $\ring$ of the (finite) Grassmannian $\Gr_{m,n}$. Here $n=4$ and $m=5$. The displayed sets of integers are the corresponding labels for the Pl\"ucker coordinates and the blue Young diagrams correspond to the frozen vertices.}
\end{figure}

The coordinate ring of the Sato Grassmannian has a natural cluster algebra structure induced by the classical cluster algebra structure in the finite-rank components $\ring$ introduced by Scott \cite{Scott}. More precisely, we have $\ring \cong \bZ[\Gr_{m,n}] \otimes_\bZ \bC$, where $\bZ[\Gr_{m,n}]$ is a cluster algebra with integer coefficients in line with Definition \ref{D:cluster algebra} which is generated, as a ring, by the Pl\"ucker coordinates $d_{j_1, \ldots, j_m}$, where $-m \leq j_1 < \ldots < j_m \leq n-1$, subject to the Pl\"ucker relations \eqref{Pluecker0}. 

For our construction the choice of the initial seed in the coordinate ring $\ring$, interpreted as rooted cluster algebra in $\mCl$, is crucial. We describe our choice of initial seed %$\Sigma_{m,n}$ for $\bZ[\Gr_{m,n}]$ 
in terms of its ice quiver, cf.\ Remark \ref{R:quiver}. 

The {\em rectangle seed} $\Sigma_{m,n}$ of $\ring$ is obtained from the following ice quiver $Q_{m,n}$: Its vertices are labelled by rectangular Young diagrams $i\times j$ of height $0<i \leq m$ and width $0<j \leq n$, together with an additional vertex labelled by the empty Young diagram, denoted by the empty set $\varnothing$. The frozen vertices of $Q_{m,n}$ are given by the the following rectangular Young diagrams:
\begin{itemize}
    \item{The empty Young diagram $\varnothing$;}
    \item{the Young diagrams of maximal width, i.e. $i=m$ and $1 \leq j \leq n$;}
    \item{the Young diagrams of maximal height, i.e. $1 \leq i \leq m$ and $j=n$.}
\end{itemize}
We have the following exhaustive list of arrows in $Q_{m,n}$, where $1 \leq i \leq m$ and $1 \leq j \leq n$, where an arrow is defined if and only if both its source and its target is a well-defined vertex in $Q_{m,n}$:
\begin{itemize}
    \item{There is an arrow $\varnothing \to 1 \times 1$;}
    \item{There is an arrow $i \times j \to i \times (j+1)$;}
    \item{There is an arrow $i \times j \to (i+1) \times j$;}
     \item{There is an arrow $i \times j \to (i-1) \times (j-1)$.}
\end{itemize}
Figure \ref{fig:finite_quiver} illustrates the case $m = 5$ and $n=4$.
We use the natural bijection between partitions and Maya sequences from Remark \ref{rmk:Maya2Young} to identify vertices in $Q_{m,n}$ with elements of $\ring$: 
Each rectangluar Young diagram of height $\le m$ and width $\le n$ gives rise to a Maya sequence $a_\bullet$ (of charge $0$) with rank $||a_\bullet||\le m+1$ and $a_1\le n-1$. Thus, we can identify that a vertex with the Pl\"ucker coordinate $d_{a\le m}$ defined in \eqref{Maya2Pluecker}. Explicitly, the vertex $\varnothing$ corresponds to the Pl\"ucker coordinate $d_\varnothing=d_{-m, \ldots, -1}$, and the vertex $i \times j$ corresponds to the Pl\"ucker coordinate 
 \begin{equation}\label{dixj}
 	d_{i \times j} %= d_{[-m,-(i+1)][j-i,j-1]} 
  = d_{-m, \ldots, -(i+1), j-i, \ldots, j-1}.
 \end{equation}

We set $\bZ[\Gr]$ to be the ring generated by the set
\[
    \{d_\lambda \mid \lambda \in \Pi\}.
\]
Thus we have $\iring \cong \bZ[\Gr] \otimes_\bZ \bC$, analogous to our notation for the finite rank components, where we have $\ring \cong \bZ[\Gr_{m,n}] \otimes_\bZ \bC$. Analogously to $Q_{m,n}$ we define the labelled ice quiver $Q_{\infty}$: Its non-frozen vertices are labelled by $d_\lambda$, where $\lambda$ ranges over all rectangular Young diagrams of height $i$ and width $j$, where $i,j \in \bZ_{>0}$, and it has a single frozen vertex labelled $d_\varnothing$, and arrows are defined as the union of the arrows in $Q_{m,n}$ for all $m,n>0$. We denote the seed associated to the ice quiver $Q_\infty$ by $\Sigma_{Q_\infty} = (\X_{Q_\infty}, \ex_{Q_\infty}, B_{Q_\infty})$.
\begin{theorem}\label{T:coordinate ring is colimit}
	The coordinate ring of the Sato-Segal-Wilson Grassmannian is a cluster algebra of infinite rank. More precisely, we have $\iring \cong \bZ[\Gr] \otimes_\bZ \bC$, where $\bZ[\Gr] \cong \cA(\Sigma_{Q_\infty})$ is a cluster algebra %in the sense of Definition \ref{D:cluster algebra} 
    with initial seed described by the labelled ice quiver $Q_\infty$.  
\end{theorem}

\begin{proof}
By \cite[Lemma~4.5]{Gcolimits} the maps $r_{m,n,m',n'} \colon \bZ[\Gr_{m,n}] \to \bZ[\Gr_{m',n'}]$ induced by the assignment rule (\ref{eqn:r-maps}) are melting cluster morphisms with respect to the rootings $\Sigma_{Q_{m,n}}$ for $\bZ[\Gr_{m,n}]$. They describe a filtered system with colimit $\tilde{\cA}$ in $\mCl$, which by Theorem \ref{T:directed_closed} is the rooted cluster algebra with initial seed induced by the labelled quiver $Q_\infty$. On the other hand, the colimit of the same directed system in the category of rings is given by $\bZ[\Gr]$. The statement now follows from Proposition \ref{P:forgetful_commutes}.
\end{proof}

    The Pl\"ucker relations hold in $\iring$ in the following sense. Let $r_{m,n} \colon \ring \to \iring$ denote the canonical maps. For $d_{j_1, \ldots, j_m} \in \ring$ we have
    \[
        d_{(\ldots,-m-1,j_1,\ldots,j_m)}= r_{m,n}(d_{j_1 \ldots j_m}).
    \]
    Note that the elements of the form $d_{(\ldots,m-1,j_1, \ldots, j_m)}\in\iring$ are labelled by the same partition $\lambda(J)$ as their finite counterparts $d_{j_1,\ldots,j_m}\in\ring$. We call them the {\em Pl\"ucker variables of $\iring$} and, when convenient, simply write $d_\lambda$ where $\lambda$ is any partition. 

Let $m,n>0$ be arbitrary but fixed. Suppose $(\lambda^{(1)},\ldots,\lambda^{(\ell)})$ is a tuple of partitions $\lambda^{(i)}$ whose Young diagrams fit inside the rectangle of height $m$ and width $n$ and, in addition, satisfy in $\ring$ the equation  $R(d_{\lambda^{(1)}},\ldots,d_{\lambda^{(\ell)}})=0$ for some polynomial $R\in\bC[d_\lambda~|~\lambda\subset (n^m)]$ where $(n^m)$ denotes the rectangular partition of height $m$ and width $n$. Then the same identity must also hold in $\iring$, because 
\[
r_{m,n}(R(d_{\lambda^{(1)}},\ldots,d_{\lambda^{(\ell)}}))=R(r_{m,n}(d_{\lambda^{(1)}}),\ldots,r_{m,n}(d_{\lambda^{(\ell)}}))=0\;.
\]
In particular, for any $-m \leq i_1 < \ldots < i_{m-1}$ and $-m \leq j_1 < \ldots < j_{m+1}$ the {\em Pl\"ucker relations}
\begin{equation}\label{Plucker}
        \sum_{\ell = 1}^{m+1} (-1)^\ell d_{(\ldots,-m-1,i_1,\ldots,i_{m-1},j_\ell)}d_{(\ldots,-m-1,j_1,\ldots,j_{\ell-1},j_{\ell+1},\ldots,j_{m+1})}=0
    \end{equation}
hold in $\iring$. Moreover, the Pl\"ucker relations determine {\em all} the relations that the Pl\"ucker variables of $\bC[\Gr]$ satisfy (\cite[Theorem~2.8]{FH-Sato}). All Pl\"ucker relations on $\bC[\Gr]$ are obtained via finite mutations on the vertices of $Q_{\infty}$. We will provide concrete examples below.

Another immediate consequence of having identified $\bZ[\Gr]$ as an ind-cluster algebra is an application of the so-called `Laurent phenomenon' and positivity of cluster algebras:
\begin{corollary}\label{cor:Laurent}
    (i) Every Pl\"ucker variable $d_\lambda$ is a Laurent polynomial $F_\lambda$ in the variables $d_{i\times j}$ (i.e. the ones labelled by rectangular Young diagrams) of the labelled ice quiver $Q_\infty$ with non-negative integer coefficients.
\medskip

\noindent (ii) Let $m\times n$ be the minimal bounding box containing the Young diagram of $\lambda$. Then the Laurent polynomial $F_\lambda$ only contains Pl\"ucker variables $d_{i\times j}$ with $1\le i\le m$ and $1\le j\le n$.
\end{corollary}

Recall that for each fixed $m,n$ and $-m\le i_k,j_k\le n-1$ the corresponding subset of equations \eqref{Plucker} describe the image of $\Gr_{m,n}$ in $\Proj(\bigwedge^m\bC^{m+n})$ under the Pl\"ucker embedding. Only 
\[\dim\Proj(\bigwedge\nolimits^m\bC^{m+n}) - \dim \Gr_{m,n}=\binom{m+n}{m}-mn-1\]
of these Pl\"ucker relations \eqref{Plucker} are algebraically independent. Having identified $\iring$ as an ind-cluster algebra allows us to naturally identify maximal sets of algebraically independent Pl\"ucker variables from which one can generate via mutations all other Pl\"ucker variables and all Pl\"ucker relations. 
\begin{corollary}\label{cor:indep}
    The Pl\"ucker variables $d_\lambda$ contained in any cluster which is obtained from $Q_\infty$ via finite mutations are algebraically independent.
\end{corollary}
From the case of $\ring$ for finite Grassmannians it is known \cite{oh2015weak} that every maximal set of {\em weakly separated} Pl\"ucker variables is a maximal set of algebraically independent Pl\"ucker variables. We extend this definition from finite to infinite sets to accommodate the labelling for $\iring$. 
\begin{definition}
        Let $I = (-\infty, -M) \cup \{\ell_1, \ldots, \ell_M\}$ and $J = (-\infty, -M') \cup \{\ell'_1, \ldots,\ell'_{M'}\}$, such that $M \leq M'$. We say that $I$ and $J$ are {\em weakly separated} if $\{\ell'_1, \ldots, \ell'_{M'}\}$ and $(-M', \ldots, -(M+1)) \cup \{\ell_1, \ldots, \ell_M\}$ are weakly separated in $\{-M', \ldots, \max\{\ell_M, \ell'_{M'}\}\}$. 
\end{definition}
For the convenience of the reader we reformulate the last criterion in terms of partitions. Given two partitions $\lambda,\mu$ denote by $A_\lambda,A_\mu$ the sets of their arm lengths (positions of the black go-stones $>-1$) and by $B_{\lambda},B_\mu$ the sets of their negative leg lengths minus one (the positions of the white go-stones $\le-1$) as defined previously. Then we have that
        \[
        I_\lambda\setminus I_\mu=(A_\lambda\setminus A_\mu)\;\bigsqcup\; (B_\mu\setminus B_\lambda)
        \]
        where $A_\lambda\setminus A_\mu\subset\bZ_{>0}$ and $B_\mu\setminus B_\lambda\subset \bZ_{\le 0}$.
\begin{figure}
\centering
\includegraphics[width=.9\textwidth]{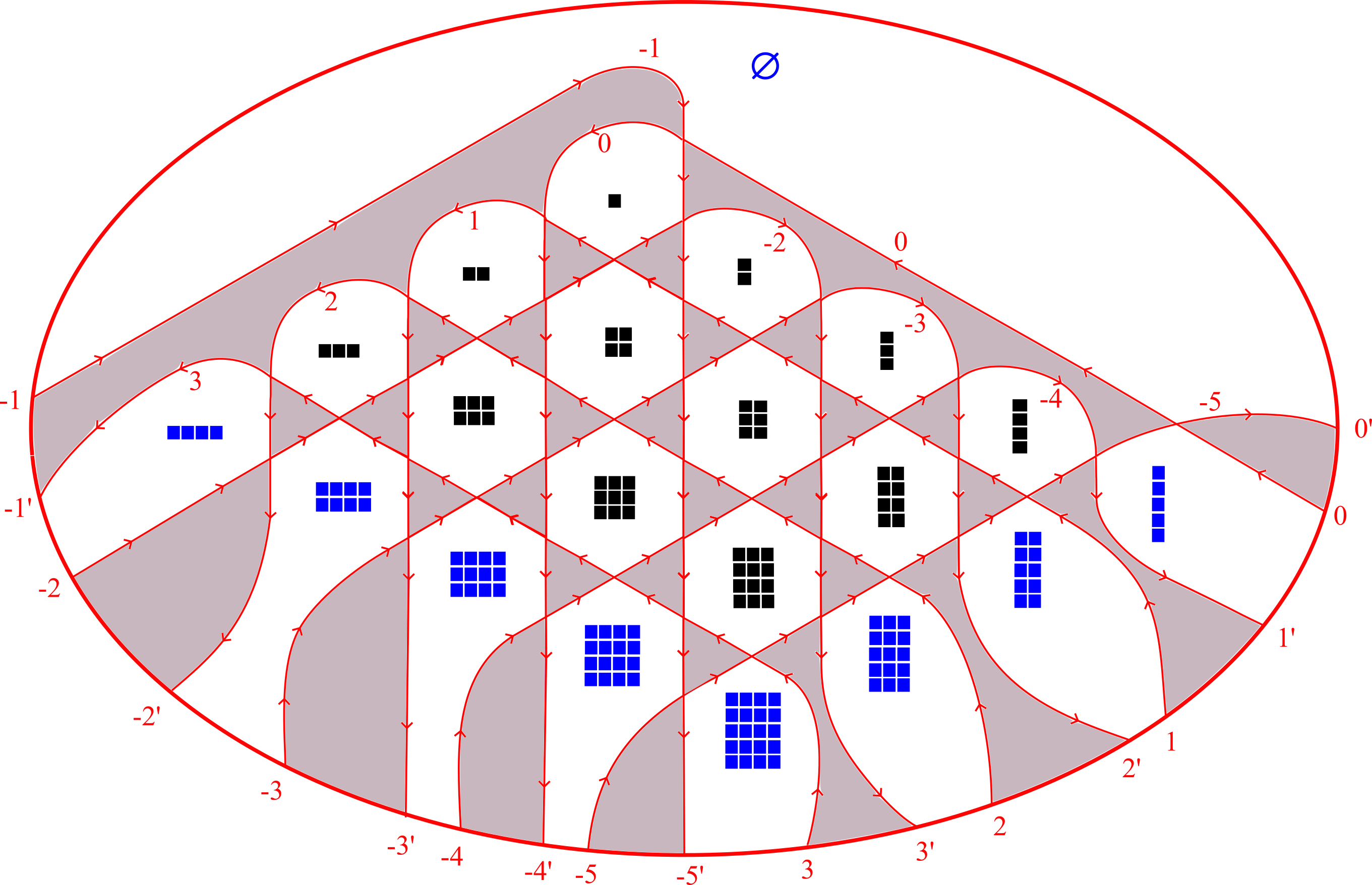} 
\caption{The rectangle-quiver $Q_{m,n}$ for the finite Grassmannian from Figure \ref{fig:finite_quiver} corresponds to the Postnikov diagram shown here.}
\label{fig:postnikovfinite}
\end{figure}

\begin{figure}
\centering
\includegraphics[width=.8\textwidth]{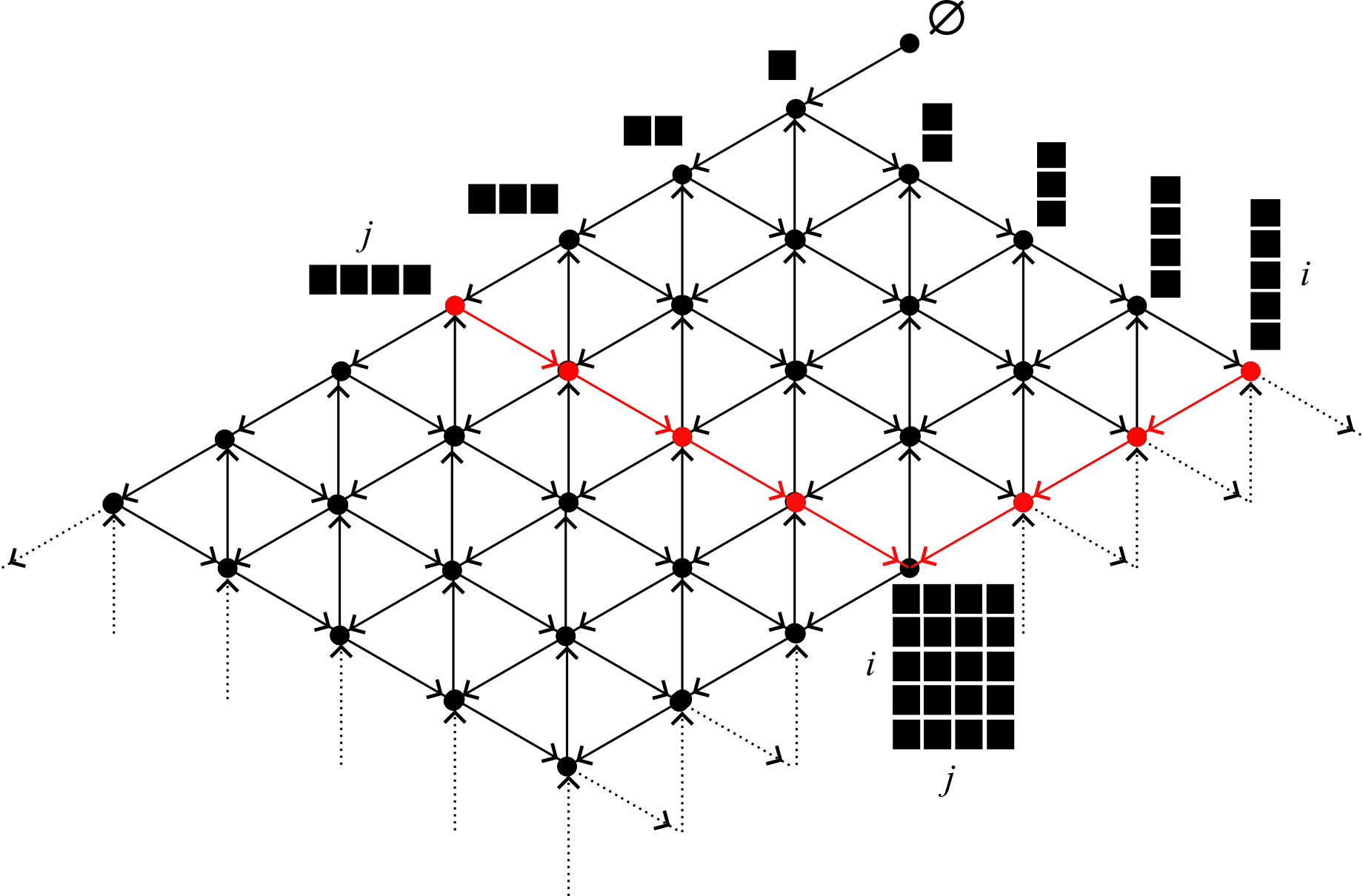} 
\caption{The (infinite) quiver $Q_\infty$ for the ind-cluster algebra $\iring$. The vertex in the $i$th row and $j$th column is labelled by the rectangular partition of width $i$ and height $j$.}
\label{fig:quiver}
\end{figure}

While the initial seed quiver $Q_\infty$ is natural in light of the colimit we have considered, it has the disadvantage that not every mutation on an exchangeable vertex leads to a Pl\"ucker relation of the form \eqref{Plucker} and may result in a cluster variable that is a Laurent polynomial of Pl\"ucker variables instead of a single new Pl\"ucker variable. This prompts the following definition.

\begin{definition}
    A cluster of $\iring \cong \cA(\Sigma_{Q_\infty}) \otimes_\bZ \bC$ whose elements are all Pl\"ucker variables is called a {\em Pl\"ucker cluster}. A Pl\"ucker seed is a seed of $\iring$ whose cluster is a Pl\"ucker cluster.
\end{definition}

The initial cluster of $\iring$ with the canonical rooting $\Sigma_{Q_\infty}$ is a Pl\"ucker cluster and we note that mutation of the Pl\"ucker variables $d_{i\times j}$ in \eqref{dixj} results in a three-term quadratic relation for $i=1$ or $j=1$ (the case of a single row or column) of the type \eqref{Plucker}, but that is not true in general when $i,j>1$ where one obtains equations of higher order.

 In the next section we give explicit examples of how to obtain via finite mutations an infinite sequence of quivers containing large subsets of exchangeable vertices where mutation results in three-term (or short) Pl\"ucker relations of type \eqref{Plucker}.

\begin{remark}\rm
We recall some reported solutions to the Pl\"ucker relations \eqref{Plucker} in terms of determinant formulae. While these have been stated in the context of actual Pl\"ucker coordinates for points on the Sato-Segal-Wilson Grassmannian they should hold true in the coordinate ring $\iring$. Provided that $d_\varnothing$ is invertible, the following Giambelli type formula for Pl\"ucker variables has been reported in the literature (see e.g. \cite[Prop. 2.1 and Cor. 2.1]{harnad2011schur}) 
    \begin{equation}\label{Giambelli}
    d_\lambda=\det(d_{(\alpha_i|\beta_j)})_{1\le i,j\le k}\,,
    \end{equation}
    where $\lambda=(\alpha_1,\ldots,\alpha_k|\beta_1,\ldots,\beta_k)$ in Frobenius notation and the $d_{(\alpha_i|\beta_j)}$ are the Pl\"ucker variables for hook partitions. For a generalised determinant formula when $d_\varnothing=0$, see e.g. \cite[Thm 1.1]{nakayashiki2017expansion}.
    
    Note that the Pl\"ucker variables appearing on the right hand side of \eqref{Giambelli} do in general not 
    belong to the same Pl\"ucker cluster. For instance, take $\lambda=(3,2)$ which in Frobenius notation is $\lambda=(2,0|1,0)$. Then \eqref{Giambelli} involves the Pl\"ucker variables $d_{(2|1)}$ and $d_{(0|0)}$. As explained above the corresponding Pl\"ucker variables $d_\lambda$ and $d_\mu$ can belong to the same cluster if and only if their corresponding sets (Maya diagrams) $a_\bullet(\lambda)$ and $a_\bullet(\mu)$ are weakly separated. We find for $\lambda=(2|1)$ and $\mu=(0|0)$ that $a_\bullet(\lambda)\setminus a_\bullet(\mu)=\{2,-1\}$ and $a_\bullet(\mu)\setminus a_\bullet(\lambda)=\{0,-2\}$ with $-2<-1<0<2$.
\end{remark}

\begin{proposition}
    Every Pl\"ucker cluster of the ind-cluster algebra $\cA(\Sigma_{Q_\infty})$ is a maximal weakly separated collection of Pl\"ucker variables. However, there are maximal weakly separated collections of Pl\"ucker variables which are not clusters of $\cA(\Sigma_{Q_\infty})$.
\end{proposition}

\begin{proof}
    Let $\tilde{\X}$ be a Pl\"ucker cluster of $\cA(\Sigma_{Q_\infty})$, which has initial cluster $\X= \X_{Q_\infty}$. Then there exists a finite $\Sigma_{Q_\infty}$-admissible sequence $\underline{x}$ such that $\mu_{\underline{x}}(\X) = \tilde{\X}$. Since $\cA(\Sigma_{Q_\infty})$ is a directed colimit in $\mCl$ (cf.\ Theorem \ref{T:coordinate ring is colimit}) we can apply Corollary \ref{C:big-admissible} and see that there exists an $i \in \cF$ and for all $j \geq i$ a unique $\Sigma_i$-admissible sequence $\underline{x}_j$ such that $f_j(\underline{x}_j) = \underline{x}$. We set $\tilde{\X}_j = \{x \in \mu_{\underline{x}_j}(\X_j) \mid f_{jk}(x) \in \X_k \text{ for all $k \geq j$}\} \subseteq \mu_{\underline{x}_j}(\X_j)$ and by Remark \ref{R:cluster_colimits} obtain that $\tilde{\X} = \bigcup_{j \geq i}(\tilde{\X}_j) / \sim$, where $y_j \sim y_k$ for $y_j \in \tilde{\X}_j$ and $y_k \in \tilde{\X}_k$ if and only if there exists an $\ell \geq j,k$ such that $f_{j\ell}(y_j) = f_{k\ell}(y_k)$. If the set $\tilde{\X}$ was not weakly separated, then the set $\tilde{\X}_j$ would not be weakly separated for some big enough $j$; a contradiction to $\tilde{\X}_j \subseteq \mu_{\underline{x}_j}(\X_j)$ \cite{oh2015weak}. Moreover, if $d_\lambda \in \iring$ is a Pl\"ucker variable such that $\tilde{\X} \cup d_\lambda$ is weakly separated, for big enough $j$ we have $d_\lambda = f_j(d_{\lambda_j})$ for a Pl\"ucker coordinate $d_{\lambda_j}$ in $\cA(\Sigma_j)$, and $\tilde{\X}_j \cup d_{\lambda_j}$ is weakly separated. Hence $d_\lambda \in \tilde{\X}_j$ by \cite{oh2015weak} and hence $d_\lambda \in \tilde{\X}$. Therefore, $\tilde{\X}$ is a maximal weakly separated collection of Pl\"ucker variables.

    Conversely, consider the weakly separated set of Pl\"ucker coordinates $\{b_\ell = d_{(-\infty,-3],-1,\ell} \mid \ell \geq 1\}$. This corresponds to a Young diagram which is a row of length $\ell+1$ followed by a row of length $1$. This can be completed to some maximal weakly separated set $W$. This is not a cluster: The infinite set $R = \{r_\ell = d_{(-\infty,-2],\ell-1} \mid \ell \geq 1\}$ is a subset of the initial cluster $\X = \X_{Q_\infty}$ of $\cA(\Sigma_{Q_\infty})$. For every $\ell \geq 1$ the pair $\{r_\ell,b_\ell\}$ is not weakly separated. Hence, the set $W$ cannot be reached from the initial cluster $\X$ by a finite sequence of mutations, as infinitely many elements of $R \subseteq \X$ would need to be replaced to make room for $W$. 
\end{proof}

\begin{remark}\rm
    There might be situations where we would want to consider all maximal weakly separated collections of Pl\"ucker variables as clusters for an appropriate ring associated to an infinite Grassmannian with a cluster-like structure. This will be addressed in a forthcoming paper on pro-cluster algebras \cite{gratz2025}.
\end{remark}

\subsection{Combinatorial description in terms of Postnikov diagrams}\label{S:combinatorial_description}

One of the advantages of describing the coordinate ring $\ring$ of a finite Grassmannian as a cluster algebra is that its algebraic relations can be iteratively generated from the chosen initial seed $\Sigma$ by successive mutations for which combinatorial descriptions are available. For example, Scott employed in \cite{Scott} diagrams introduced by Postnikov in \cite{Postnikov} to describe the Pl\"ucker clusters and the mutation procedure for $\ring$ and it is known that a general Pl\"ucker relation \eqref{Pluecker0} can be obtained via mutation from those which only involve three terms. We briefly recall the main definitions and results for finite Grassmannians before adapting them to the case of $\iring$. 

As before let $m,n>0$ be arbitrary but fixed. Denote by $\pi_{m,n}$ the following Grassmannian permutation of the set $\{-m,-m+1,\ldots,n-2,n-1\}$,
\[
\pi_{m,n}=\left(\begin{array}{ccccccc}
-m &-m+1&\ldots& n-m-1& n-m&\ldots& n-1\\
0'&1'&\ldots&(n-1)'&-m'&\ldots &-1'
\end{array}\right)
\]
and fix a convex $2(m+n)$-gon $P_{m,n}$ with vertices labelled $-m,-m',\ldots,n-1,(n-1)'$ clockwise.
\begin{definition}[Postnikov diagrams \cite{Postnikov}]
A $\pi_{m,n}$-diagram (or Postnikov diagram) is a tuple $(p_1,\ldots,p_{m+n})$ of oriented paths inside $P_{m,n}$, where each $p_i$ connects the source $i$ with its target $\pi_{m,n}(i)'$ subject to the following conditions:
\begin{itemize}
    \item[(P1)] the paths do not self-intersect;
    \item[(P2)] all intersections between different paths are transversal;
    \item[(P3)] running along a path $p_i$ from source to target, there is always an even number of intersections which alternate in their orientation: $p_i$ is intersected first from the left, then from the right, then from the left etc. 
    \item[(P4)] for any pair of mutually distinct paths the configuration shown in Figure \ref{fig:forbidden} is not allowed.
\end{itemize}
\end{definition}

Note that our labelling of the boundary of a Postnikov diagram\footnote{We also point out that Postnikov diagrams have been introduced for arbitrary permutations, not just the Grassmannian permutation $\pi_{m,n}$ we consider, but as it is only the latter which we need for describing the cluster algebra structure of $\ring$ we omit the general case.} is by integers $-m', -m, \ldots, (n-1), (n-1)'$, while \cite{Scott} uses labels $1, 1', \ldots, (m+n), (m+n)'$; see Figure \ref{fig:postnikovfinite}.

We collect some well-known additional notions and facts regarding $\pi_{m,n}$-diagrams from \cite{Postnikov} and \cite{Scott}:
\begin{itemize}
    \item[(P5)] The interior of $P_{m,n}$ is divided into two types of disjoint cells by the paths of a $\pi_{m,n}$-diagram: there are {\em odd cells} which have an oriented boundary and {\em even cells} which do not. No two cells of the same type share a common edge.
    \item[(P6)] There exists a bijection between $\pi_{m,n}$-diagrams $\Gamma$ and the labelled ice quivers $Q$ representing the Pl\"ucker clusters of $\ring$: Each even cell of a fixed diagram $\Gamma$ is mapped to a vertex of $Q$ with the frozen vertices corresponding to the even cells which share an edge with the boundary of $P_{m.n}$. Two vertices $v,v'\in Q$ are connected by an arrow $v\to v'$ if the corresponding even cells $c,c'\in\Gamma$ touch in one point and the direction of the arrow is aligned with the orientation of the intersecting paths in $\Gamma$. For example, in the $\pi_{m,n}$-diagram of Figure \ref{fig:postnikovfinite} we obtain for the even cells containing the empty partition $\varnothing$ and $\lambda=(1)$ the arrow $\varnothing\to\Box$. The labelling of the even cells, and thus the bijection, is fixed as follows: The Maya sequence corresponding to the Young diagram contains the path label $i$ if the path either runs beneath the cell from west to east or to the right of it from south to north. Doing this consistently for the $\pi_{m,n}$-diagram of Figure \ref{fig:postnikovfinite} the reader will obtain the ice quiver $Q_{mn}$ in Figure \ref{fig:finite_quiver} except for some additional arrows connecting the frozen vertices (which do not play a role in the cluster algebra structure as the frozen vertices are exempt from mutation).
    \item[(P7)] Any two $\pi_{m,n}$-diagrams can be obtained from each other through a finite sequence of `local moves' called {\em geometric exchange}. This move, which is only defined for even cells that form a `quadrilateral' (i.e. their boundary contains precisely four intersection points of paths) is depicted in Figure \ref{fig:geometricex}. The bijection between $\pi_{m,n}$-diagrams and quivers from (P2) is compatible with this `mutation' on quadrilateral cells, i.e. if the diagrams $\Gamma$ and $\Gamma'$ are related by a sequence of geometric exchanges at cells $c_1,\ldots,c_\ell$, then the corresponding quivers $Q(\Gamma)$ and $Q(\Gamma')$ are related by quiver mutation on the corresponding vertices $v_1,\ldots,v_\ell$.
\end{itemize}

Using the results for Postnikov diagrams for finite Grassmannians we now extend the discussion to the case of the ind-cluster algebra $\iring$. There is an obvious embedding of the rectangle-quivers $Q_{m,n}$ into the quiver $Q_\infty$ by identifying vertices which are labelled by the same (rectangular) partition, recall that $d_\lambda=r_{m,n}(d_\lambda)$ for any $\lambda\subset (n^m)$. Using the bijection from (P6) one easily verifies that the `infinite Postnikov diagram' $\Gamma_\infty$ shown in Figure \ref{fig:postnikov} -- when following the same rules for labelling even cells and drawing arrows as in the finite rank case -- leads to $Q(\Gamma_{\infty})=Q_\infty$. 

In drawing the infinite Postnikov diagram $\Gamma_\infty$ we have chosen the following conventions: Consider the circle $S^1$ with boundary points given by $\bZ \sqcup \bZ'$, that is, they are consecutively labelled, in a clockwise direction, by $ \ldots, -3, -2, -1, 0, 1, 2, ,3, \ldots$ followed by $\ldots, -3', -2', -1', 0', 1', 2', 3', \ldots$, cf.\ Figure \ref{fig:postnikov}. There are four accumulation points of boundary points, which do not count as boundary points themselves, and which we mark in the figure by $-\infty', \infty', -\infty, \infty$. We denote this ``$\infty$-gon'' by $P_\infty$. The diagram $\Gamma_\infty$ is a doubly-infinite sequence $(\ldots,p_{-1},p_0,p_1,\ldots)$ of oriented paths $p_i$ inside $P_\infty$ that connect the boundary point $i$ with $i'$ and satisfy (P1), (P2) and (P4). Moreover, we have instead of (P3) that each path is intersected in an alternating fashion, swapping between left and right; see Figure \ref{fig:postnikov}. One verifies from the construction that the paths still divide the interior of $P_\infty$ into even and odd cells; see (P7). 

\begin{definition}
We define a {\em $\Gamma_\infty$-Postnikov diagram} to be any diagram obtained from $\Gamma_\infty$ by finite sequences of geometric exchanges on quadrilateral even cells. 
\end{definition}
This definition ensures that any infinite Postnikov diagram $\Gamma$ only differs in a finite region from $\Gamma_\infty$. That is, there always exists an integer $k>1$ such that (1) the paths $p_{-k}$ and $p_{k-1}$ intersect below the Young diagram $\lambda=(k^k)$ and (2) below these paths all vertices are rectangular Young diagrams and the Postnikov diagram $\Gamma$ matches $\Gamma_\infty$ for all paths $p_j$ with $j\ge k-1$ or $j\le -k$, while for $-k<j<k-1$ the paths $p_j\in\Gamma$ follow the same arrangement as in $\Gamma_\infty$ after intersecting $p_{k-1}$ (if $j>0$) or $p_{-k}$ (for $j\le 0$). 

\begin{figure}
\centering
\includegraphics[width=.5\textwidth]{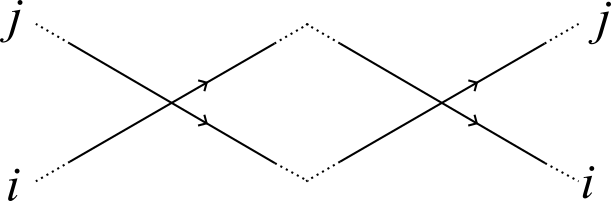} 
\caption{The displayed double crossing of two paths, here labelled $i$ and $j$, is not allowed in a Postnikov diagram.}
\label{fig:forbidden}
\end{figure}

\begin{figure}
\centering
\includegraphics[width=.8\textwidth]{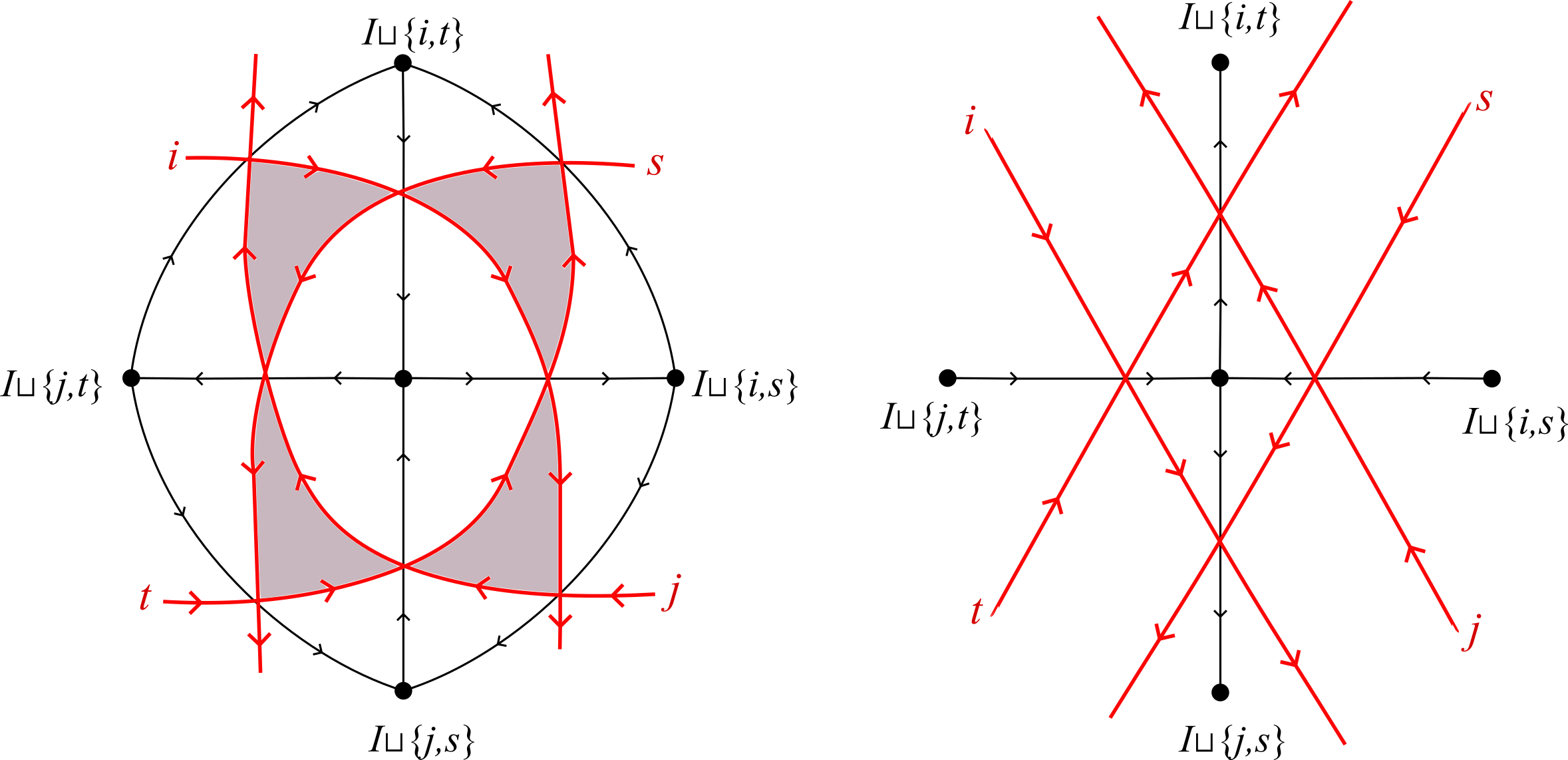} 
\caption{The `geometric exchange relation' for Postnikov diagrams (red) which describes the quiver mutation at the centre vertex (black).}
\label{fig:geometricex}
\end{figure}

\begin{example}\rm
The infinite Postnikov diagram shown in Figure \ref{fig:postnikov2quads4} matches $\Gamma_\infty$ below the paths $p_{4}$ and $p_{-5}$.
\end{example}

This example motivates that a Pl\"ucker cluster of $\iring$ looks `locally' like a Pl\"ucker cluster of some $\ring$ with $m,n$ sufficiently large.  

\begin{proposition}
    The set of $\Gamma_\infty$-Postnikov diagrams is in bijection with the set of Pl\"ucker clusters of $\iring$.
\end{proposition}
   
\begin{proof}
    This is now an immediate consequence of our construction: The Pl\"ucker clusters differ from the initial cluster $\X_{Q_\infty}$ of $\Sigma_{Q_\infty}$ only in finitely many variables. Hence, 
    by \cite[Prop. 6]{Scott}, any Pl\"ucker cluster $\tilde{\X}$ can be reached from $\X_{Q_\infty}$ by a finite mutation sequence, corresponding to a finite sequence of geometric exchanges (see Figure \ref{fig:geometricex}) of $\Gamma_\infty$.
\end{proof}

\begin{figure}
\centering
\includegraphics[width=.9\textwidth]{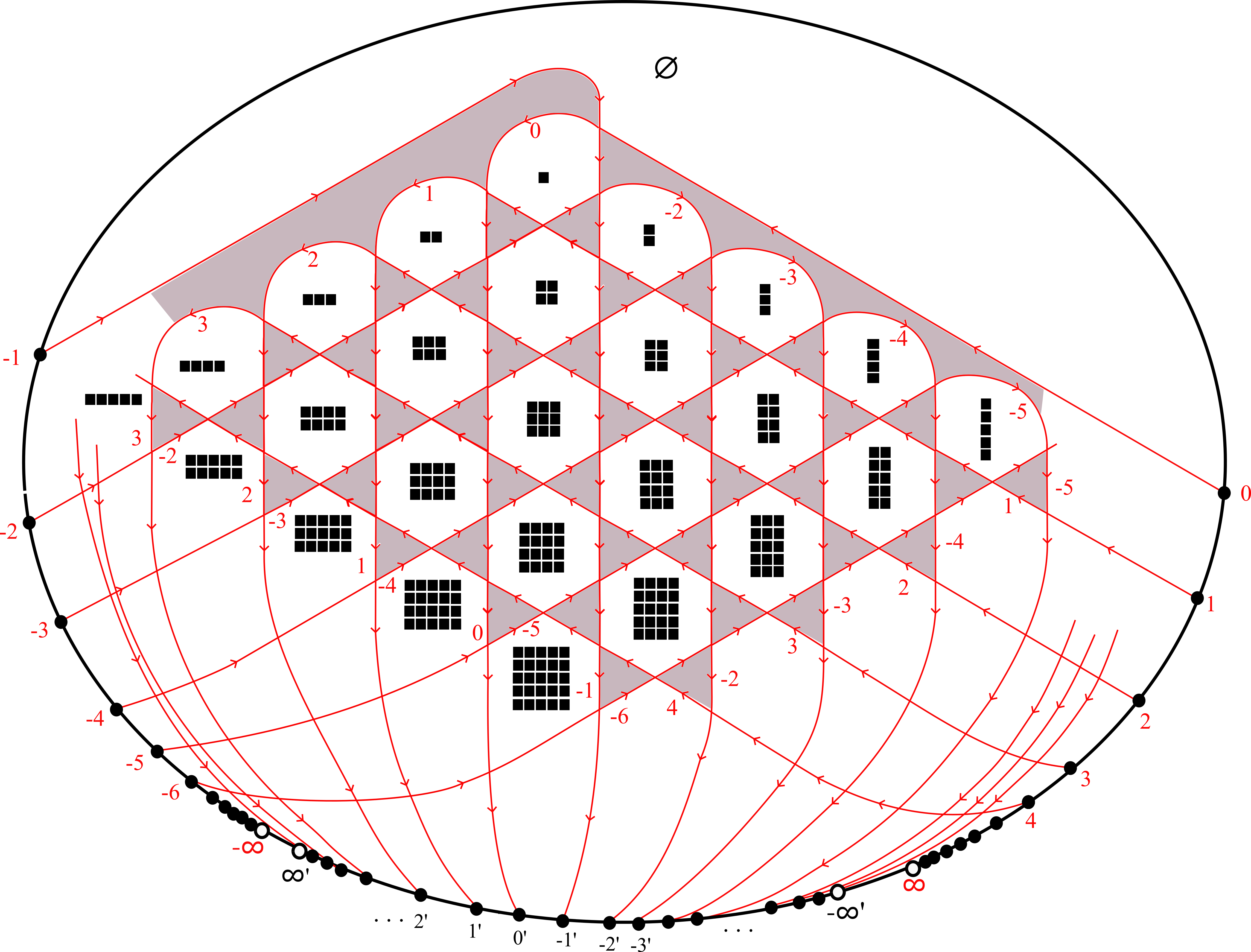} 
\caption{The infinite Postnikov diagram $\Gamma_\infty$ (shown in red) corresponding to the quiver $Q_\infty$ which fixes the initial seed of the ind-cluster algebra $\iring$. The hollow accumulation points at $\pm\infty$ and $\pm\infty'$ are {\em not marked boundary points} of $P_\infty$, i.e. they do not have paths of $\Gamma_\infty$ ending in them.}
\label{fig:postnikov}
\end{figure}

Applying repeatedly the `geometric exchange' relation from Figure \ref{fig:geometricex} to the infinite Postnikov diagram (red) in Figure \ref{fig:postnikov} (or the corresponding quiver mutation) starting with the quadrilateral cells at the boundary in $\Gamma_\infty$ one obtains via finite mutations so-called short (three-term) Pl\"ucker relations.
\begin{example}\rm
 Set $m>2$ in \eqref{Plucker}. Then the resulting Pl\"ucker relation will in general contain $m+1$ terms. For instance, choose $m=3$ then \eqref{Plucker} yields (among others) the relation
    \begin{multline*}
    d_{(\ldots,-4,0,1,2)}d_{(\ldots,-4,-3,-2,-1)}-d_{(\ldots,-4,-1,1,2)}d_{(\ldots,-4,-3,-2,0)}\\
    + d_{(\ldots,-4,-2,1,2)}d_{(\ldots,-4,-3,-1,0)}- d_{(\ldots,-4,-3,1,2)}d_{(\ldots,-4,-2,-1,0)}=0
    \end{multline*}
    or, equivalently, in terms of Young diagrams,
     \[
    d_{(3,3,3)}d_{\varnothing}-d_{(3,3,2)}d_{(1)}+ d_{(3,3,1)}d_{(1,1)}- d_{(3,3)}d_{(1,1,1)}=0\;.
    \]
    In contrast, the geometric exchange relation from Figure \ref{fig:geometricex} always gives three-term relations, even for Pl\"ucker coordinates labelled by partitions $\lambda$ with $\ell(\lambda)>2$. For example, by repeatedly applying the geometric exchange relation on the `diagonal' of the ice quiver $Q_\infty$ (see Figures \ref{fig:quiver} and \ref{fig:postnikov}), whose vertices consist entirely of rectangular Young diagrams, we find for any $k\ge 1$ that
    \begin{equation}
    d_{(k^k)}d_{((k+1)^k,k)}=d_{((k+1)^k)}d_{(k^{k+1})}+d_{(k^{k-1},k-1)}d_{((k+1)^{k+1})},
    \end{equation}
    see Figure \ref{fig:genKP}. The case $k=0$, which corresponds to the KP-equation, is shown in Figure \ref{fig:postnikov2KPsimple}.
\end{example}

\begin{figure}
\centering
\includegraphics[width=.9\textwidth]{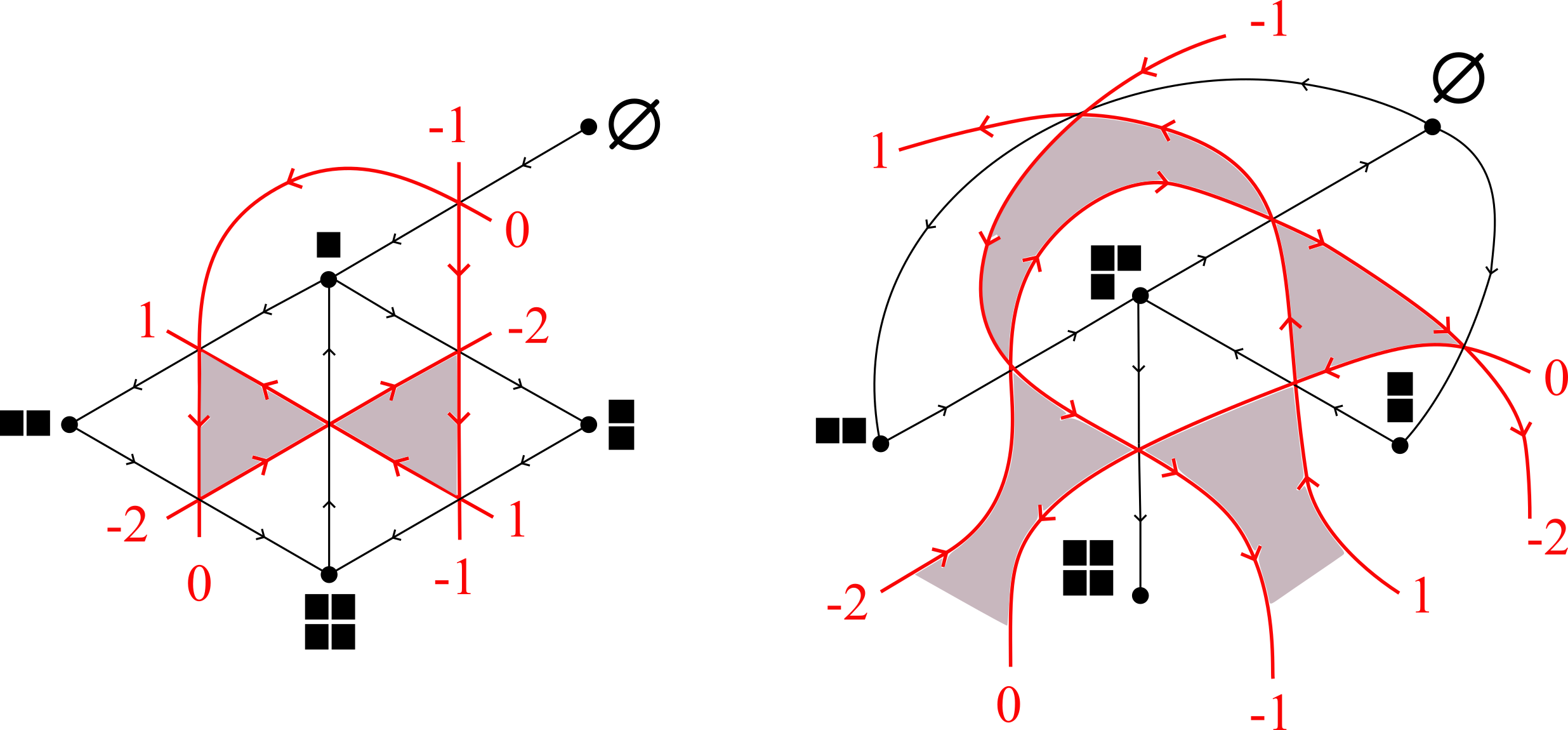} 
\caption{The KP equation as quiver mutation: applying the geometric exchange relation to the first vertex on the diagonal in Figure \ref{fig:postnikov} yields a Pl\"ucker relation which corresponds to the KP equation for the $\tau$-function.}
\label{fig:postnikov2KPsimple}
\end{figure}

\begin{figure}
\centering
\includegraphics[width=.9\textwidth]{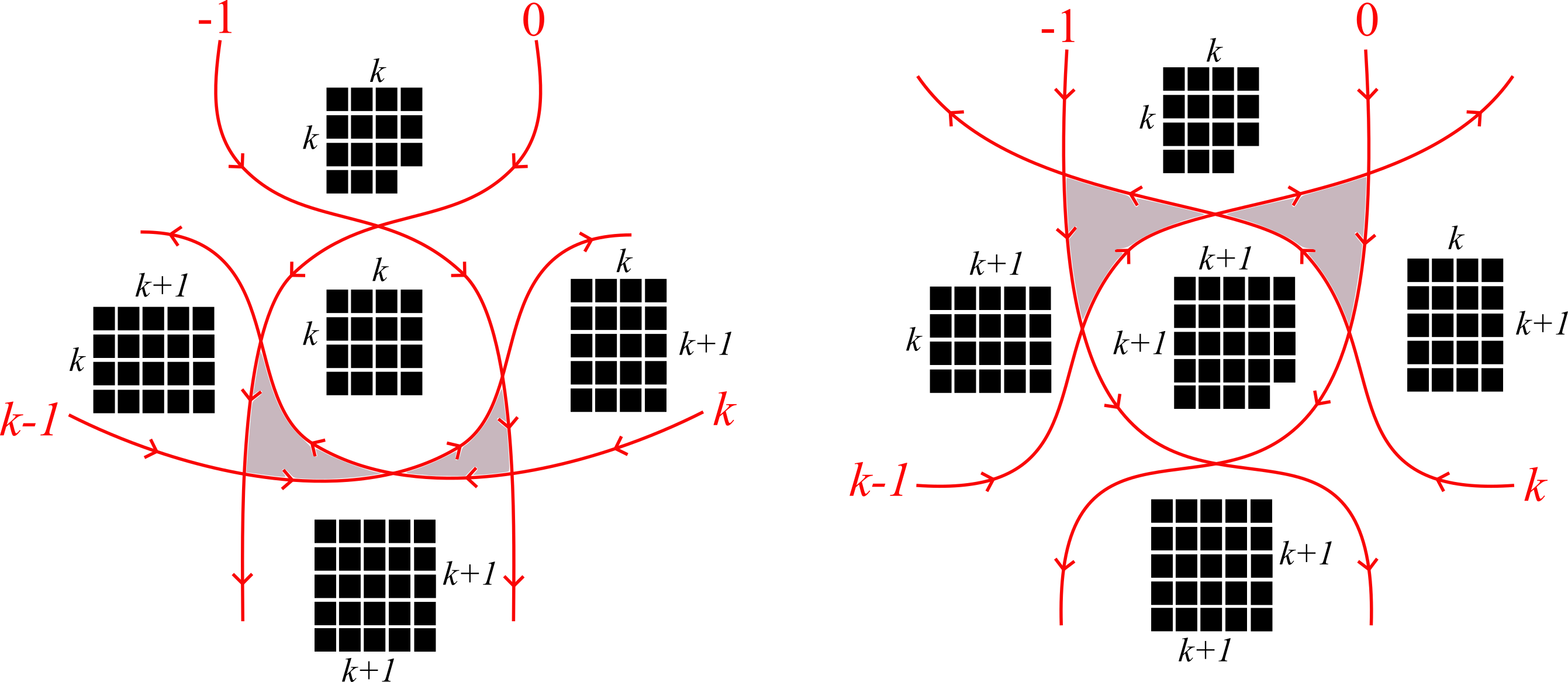} 
\caption{Let $k\ge 1$. Applying repeated the geometric exchange relation on the diagonal of the ice quiver $Q_{\infty}$ (respectively the corresponding infinite Postnikov diagram) one proves the above relation which results in a non-trivial 3-term PDE for the $\tau$-function of the KP-hierarchy.
The case $k=0$ corresponds to the KP-equation shown in Figure \ref{fig:postnikov2KPsimple}.}
\label{fig:genKP}
\end{figure}
    
We now construct an infinite sequence of new diagrams $\Gamma_m$ which contain a large subregion whose size is controlled by $m$ and consists only of quadrilateral cells. 

\begin{example}\label{ex:quad quiver}
\rm
Fix any integer $m>2$. Using repeatedly the geometric exchange relation from Figure \ref{fig:geometricex}, move the $0$-line in the Postnikov diagram from Figure \ref{fig:postnikov} down to the Young diagram consisting of a single column of height $m$ and the $(-1)$-line left to the Young diagram of a single row of width $m$; see Figure \ref{fig:postnikov2quads4} for an example where $m=4$. The resulting quiver $Q_\infty^{(m)}$ contains a finite sub-quiver $Q(m)$ for $\bC[\Gr_{m,m}]$ which only contains vertices of valency four; see Figures \ref{fig:postnikov2quads4} and \ref{fig:quad_quiver}. We now describe $Q(m)\subset Q_\infty^{(m)}$ explicitly: 
\begin{itemize}
\item The vertices of $Q(m)$ are labelled by partitions $\lambda$ which in Frobenius notation are all of the form
\[
\lambda(a,b;k)=(a,a-1\ldots,a-k|b,b-1,\ldots,b-k),\qquad 0\le k\le \min(a,b),\quad 0\le a,b\le m-1\;,
\]
where $a-i$ denotes the arm length (boxes to the right) and $b-i$ the leg length (number of boxes below) for the $i$th square on the diagonal of the Young diagram of $\lambda$. The corresponding Maya diagrams (interpreted as infinite sets) are given by the two relations
\[
a_\bullet\cap\bZ_{\ge 0}=\{a,a-1,\dots,a-k\}\qquad\text{and}\qquad
\bZ_{<0}\setminus a_\bullet=\{-b-1,-b,\ldots,-b-1+k\}\;.
\]
In other words, we place in Figure \ref{fig:BFcorr} black go-stones at positions $\{a,a-1,\dots,a-k\}$ and white go-stones at positions $\{-b-1,-b,\ldots,-b-1+k\}$.

\item If there is a vertex labelled by $\lambda(a,b;k)$ then there is also a vertex labelled by the conjugate partition $\lambda(a,b;k)'=\lambda(b,a;k)$. In other words, the set of vertex labels is invariant under taking the transposes of Young diagrams. 

\item The frozen vertices are labelled by the partitions $\varnothing$, $\lambda(m-1,m-1;m-1)=(m)^m$ as well as 
\[
 \lambda(m-1,m-2;m-2)=(m)^{m-1},\ldots, \lambda(m-1,0;0)=(m)
\]
and, taking conjugates,
\[
\lambda(m-2,m-1;m-2)=(m-1)^{m},\ldots, \lambda(0,m-1;0)=(1^m)\;.
\]
We arrange the frozen vertices such that they form the lower half of a square as shown in Figure \ref{fig:quad_quiver}. We refer to the vertices on the straight line connecting the frozen vertices $\varnothing$ and $\lambda(m-1,m-1;m-1)$ as the `main diagonal' and to the vertices on the lines above as the first, second, etc. diagonal. The vertex labels on the line below the main diagonal are obtained from the ones above via taking the conjugate partition. 

\item The quiver $Q(m)$ is `reflection symmetric' with respect to the main diagonal when sending $\lambda(a,b;k)$ to its conjugate $\lambda(b,a;k)$. In particular, the vertices on the main diagonal are invariant under conjugation, i.e. they are of the form $\lambda(a,a;k)$ with $\lfloor m/2\rfloor\le a\le m-1$ and $0\le k\le m-1$. The vertices in the $r$th diagonal are of the form $\lambda(a,a-r;k)$ with $\lfloor (m+r)/2\rfloor\le a\le m-1$ and $0\le  k\le m-1-r$. Place the vertices labelled by $\lambda(a,b;k)$ in the $k$th row of the square if $a> b$ and in the $k$th column if $a<b$. 

\item There are two types of vertex-arrow configurations which are shown in Figure \ref{fig:quad_vertices}. Which type occurs depends on the diagonal $r$ and the integer $k$ in $\lambda(a,a-r;k)$: if $r+k$ is odd then the configuration on the left in Figure \ref{fig:quad_vertices} occurs, if $r+k$ is even then the one on the right. Because the quiver is reflection symmetric on the diagonal this fixes $Q(m)$ uniquely.

\end{itemize}

Using the sub-quiver $Q(m)$ from Figure \ref{fig:quad_quiver} for any $m\ge 2$ we can obtain an infinite set of three-term or `short' Pl\"ucker relations for hook partitions using the geometric exchange relation: for any $m>2$ and $1\le a,b\le m-1$ we have that (using Frobenius notation for the partitions involved)
    \begin{equation}\label{hookPlucker}
    d_{(a|b)}d_{(a-1|b-1)}=d_\varnothing d_{(a,a-1|b,b-1)}+d_{(a-1|b)}d_{(a|b-1)}\;.
    \end{equation}
If we assume that \eqref{Giambelli} holds then this set of relations is a direct consequence of the Desnanot-Jacobi identity for the determinant $d_{(a,a-1|b,b-1)}$. Furthermore, when setting $a=b=1$ the relation reduces to the one from Figure \ref{fig:postnikov2KPsimple}. As we will discuss below these relations result in partial differential equations for the $\tau$-functions of the KP-hierarchy.
\end{example}

\begin{figure}
\centering
\includegraphics[width=.9\textwidth]{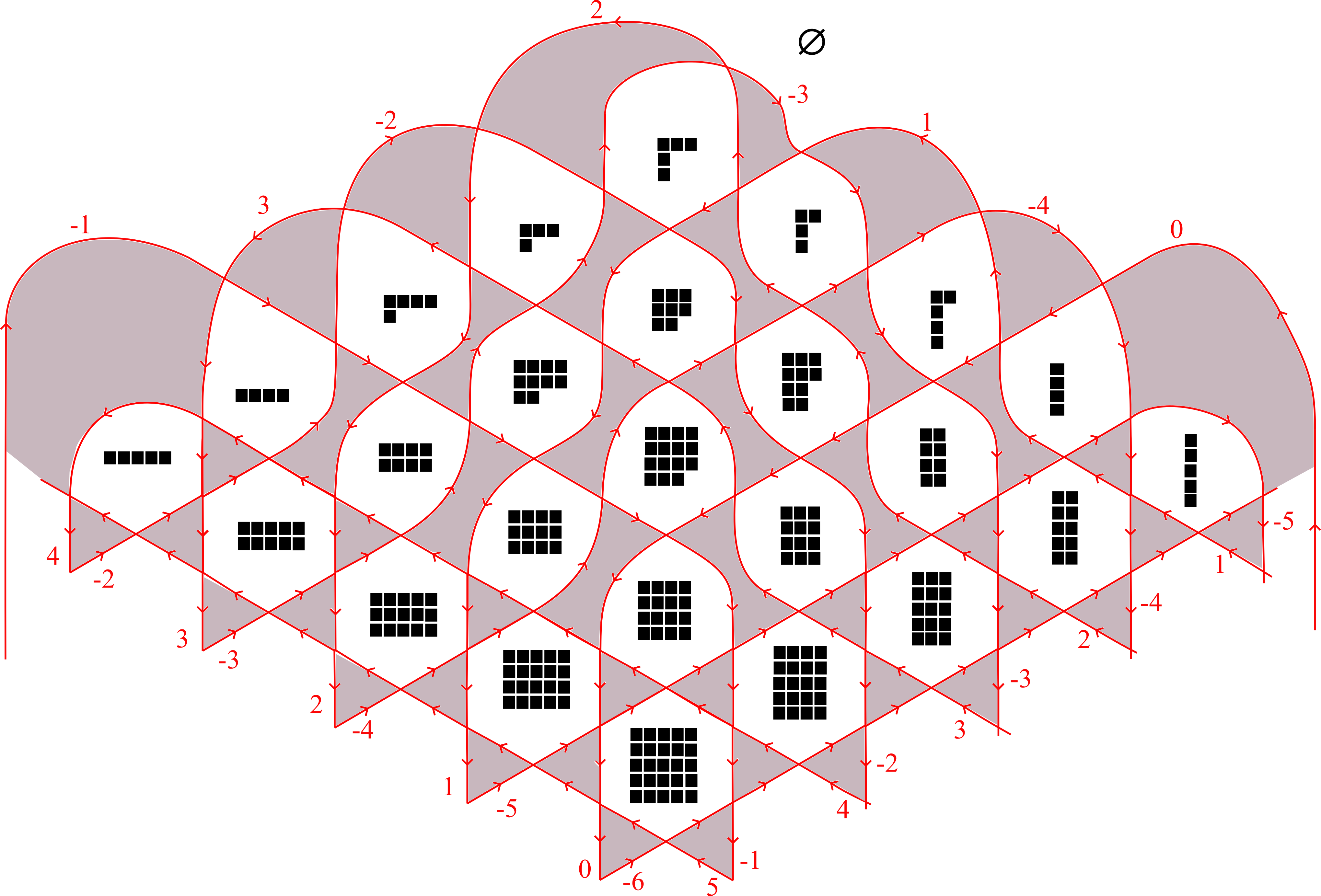} 
\caption{The (infinite) Postnikov diagram \ref{fig:postnikov} after the mutation sequence 
\[
\left((1),(1,1),(2),(2,2),(1,1,1),(3),(2,2,2),(3,3),(3,3,3),(2,1)
%\ytableausetup
%{boxsize=0.5em}
%\ydiagram{1},\ydiagram{1,1},\ydiagram{2},\ydiagram{2,2},\ydiagram{1,1,1},\ydiagram{3},\ydiagram{2,2,2},\ydiagram{3,3},\ydiagram{3,3,3},\ydiagram{2,1}
\right)\;.
\]
It contains a subdiagram for the Grassmannian $\Gr_{4,4}$ that only consists of quadrilateral cells.
}
\label{fig:postnikov2quads4}
\end{figure}

\begin{figure}
\centering
\includegraphics[width=.7\textwidth]{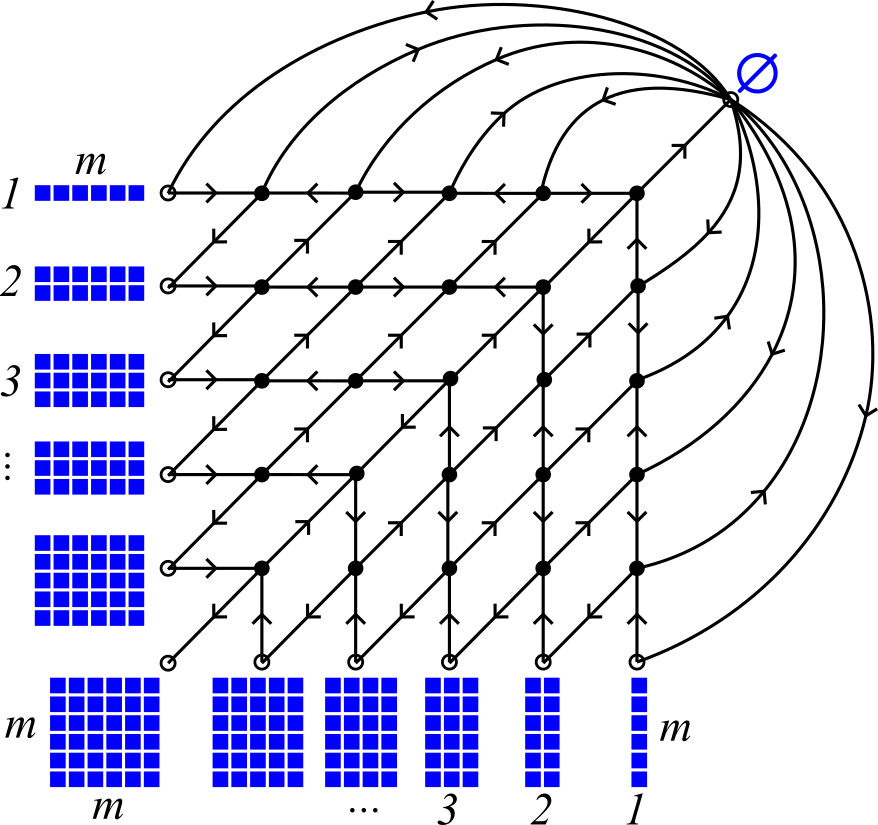}
\caption{A quiver which only consists of four-valent vertices. Setting $m=4$ one obtains via the bijection described under (P6) the Postnikov diagram from Figure \ref{fig:postnikov2quads4}.}
\label{fig:quad_quiver}
\end{figure}

\begin{figure}\label{fig:quad_vertices}
\centering
\includegraphics[width=1\textwidth]{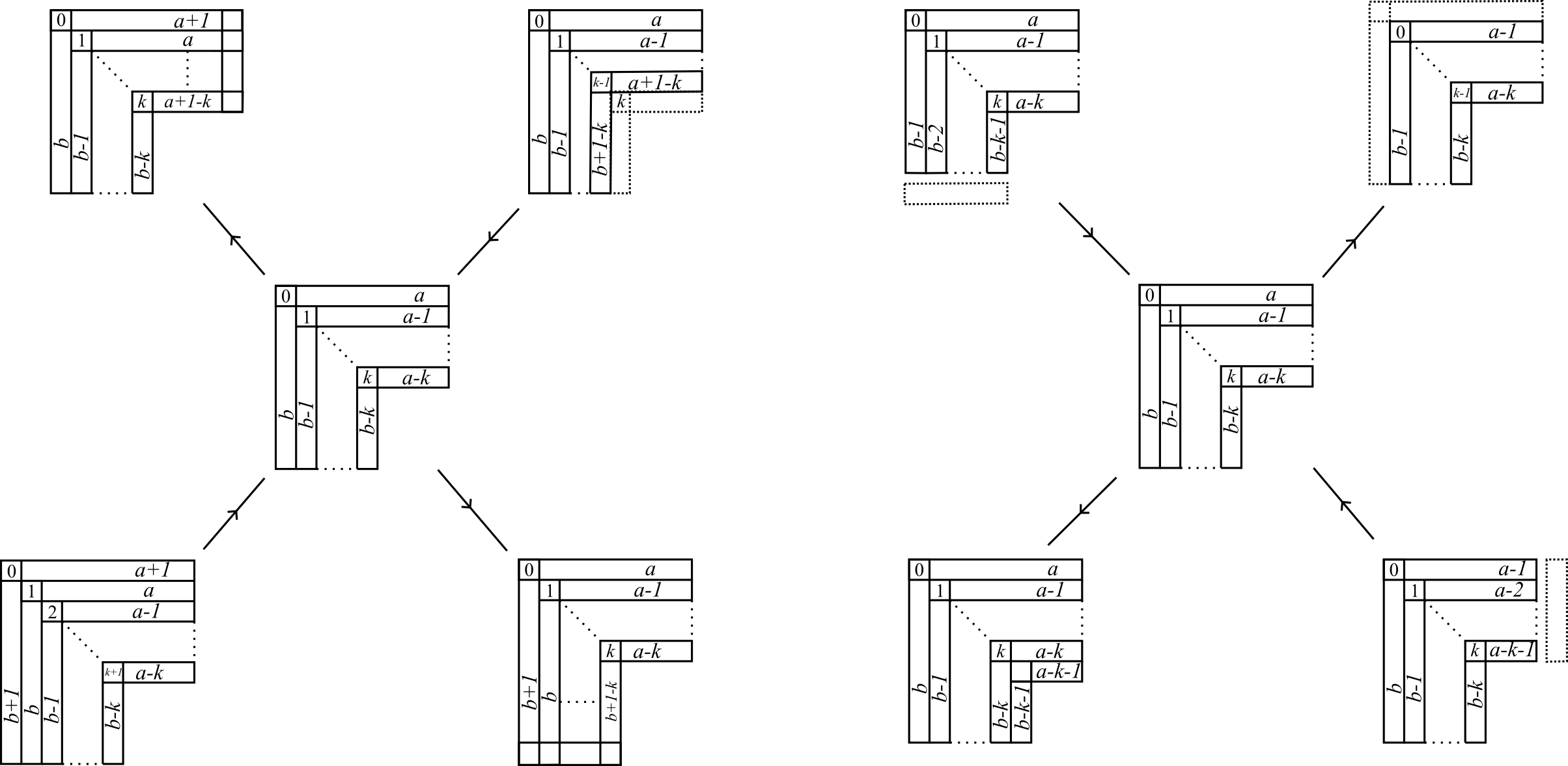}
\caption{The two types of vertex-arrow configurations which appear in the quiver from Figure \ref{fig:quad_quiver}. Vertices in the $r$th diagonal have the configuration on the left if $r+k$ is odd and the one on the right if $r+k$ is even. On the left the neighbouring vertices are obtained by  adding a maximal hook or removing an inner hook, $\lambda(a+1,b+1;k+1)\to\lambda(a,b;k)\leftarrow\lambda(a,b,;k-1)$, or increasing all the arm lengths or all the leg lengths, $\lambda(a+1,b;k)\leftarrow\lambda(a,b;k)\to\lambda(a,b+1;k)$. On the right the neighbouring vertices differ by adding an inner hook or removing an outer hook, $\lambda(a,b;k+1)\leftarrow\lambda(a,b;k)\rightarrow\lambda(a-1,b-1;k-1)$ or decreasing all leg lengths or arm lengths by one, $\lambda(a,b-1;k)\to\lambda(a,b;k)\leftarrow\lambda(a-1,b;k)$.}
\end{figure}

%%%%%%%%%%%%%%%%%%%%%%%%%%%%%%%%%%%%%%%%%%%%
\section{The connection with the KP-hierarchy}
\label{S:KP hierarchy}

In this section, we briefly recall some known facts about the connection of the coordinate ring $\iring$ with $\tau$-functions, the solutions of the KP hierarchy, in order to make the connection between the latter and the ind-cluster algebra $\iring$ explicit. 

\subsection{Schur functions and $\tau$-functions of the KP-hierarchy}
We start by recalling the definition of Schur functions which form an important $\bZ$-basis in the ring of symmetric functions $\Lambda$; see \cite[Chap. I]{Macdonald:symmetric}. The latter is the projective limit of the rings of symmetric polynomials $\Lambda_n=\bZ[x_1,x_2,\ldots,x_n]^{S_n}$ in $n$ variables in the category of graded rings. That is, setting $\Lambda_n=\bigoplus_{k\ge 0}\Lambda^k_n$, where $\Lambda^k_n$ is the degree $k$ subspace, we consider the projective system $\pi_{mn}:f(x_1,\ldots,x_m)\mapsto f(x_1,\ldots,x_n,0,\ldots,0)$ for $m>n$. Then we have that 
\[
\Lambda=\bigoplus_{k\ge 0}\Lambda^k,\qquad\Lambda^k=\lim\limits_{\longleftarrow}\Lambda_n^k
\]
where $\Lambda$ is a free $\bZ$-module with basis $\{m_\lambda~:~\lambda\in\Pi\}$, the monomial symmetric functions, and where $\Pi$ is the set of all partitions of integers. Let $\Bbbk$ be any unital associative and commutative ring, then we shall adopt the notation $\Lambda_\Bbbk=\Lambda\otimes_\bZ\Bbbk$.

In particular, if we set $\Bbbk=\bQ$ and let $p_r=m_{(r)}$, the $p_r=\sum_{i\ge 1}x_i^r$ are called the Newton power sums, then $\Lambda_{\bQ}\cong\bQ[p_1,p_2,\ldots]$ and the $p_\lambda=p_{\lambda_1}p_{\lambda_2}\cdots$ with $\lambda\in\Pi$ form a $\bQ$-basis of $\Lambda_\bQ$, see e.g. \cite[Chap.I]{Macdonald:symmetric}. Using the Frobenius formula for the irreducible characters $\chi_\lambda$ of the symmetric group $S_n$ with $|\lambda|=n$ one has
\begin{equation}\label{Schur}
    s_{\lambda}(p)=\sum_{\mu\vdash|\lambda|}\frac{\chi^{\lambda}(\mu)}{z_\mu}\;p_{\mu_1}p_{\mu_2}\ldots p_{\mu_\ell},\qquad
    z_\mu=\prod_{i>0}i^{m_i}m_i(\mu)!\,,
\end{equation}
    where $\chi^{\lambda}(\mu)$ is the value of the irreducible character on permutations of cycle type $\mu$ and $z_\mu$ is the cardinality of the centraliser of a permutation of cycle type $\mu$ in $S_n$. 
    \begin{remark}\rm
    We alert the reader to the fact that we have broken standard conventions with large parts of the integrable systems literature: commonly the `flow parameters' $t_r=p_r/r$ with $r\ge 1$ are used instead in the physics literature with $t_1=x$, $t_2=y$ being identified as the spatial coordinates and $t_3=t$ as the time parameter. The remaining flow parameters belong to higher integrals of motion; see e.g. \cite{miwa2000solitons} for details.
    \end{remark}
    
    There is a natural non-degenerate bilinear form $\Lambda\otimes\Lambda\to\bZ$, called the Hall inner product, with respect to which the Schur functions form an orthonormal basis, $\langle s_\lambda,s_\mu\rangle=\delta_{\lambda,\mu}$; see e.g \cite[Chapter I]{Macdonald:symmetric}. Denote by $s^\ast_\lambda$ the adjoint of the multiplication operator $s_\lambda$ with respect to this Hall inner product. Explicitly, we have the following realisation of $s^*_\lambda$ as a differential operator in $\Lambda_\bQ$,
    \[
    s^\ast_\lambda=s_\lambda(\partial),\qquad \partial=\left(\frac{\partial}{\partial p_1},2\frac{\partial}{\partial p_2},3\frac{\partial}{\partial p_3},\ldots\right)
    =\left(\frac{\partial}{\partial t_1},\frac{\partial}{\partial t_2},\frac{\partial}{\partial t_3},\ldots\right)
    \;.    
    \]
We are now ready to state the definition of a `formal solution' of the KP-hierarchy. To this end we return to considering the extension of the ring of symmetric functions by an arbitrary unital associative and commutative ring $\Bbbk$, %
$
\Lambda_{\Bbbk}=\Lambda\otimes_\bZ\Bbbk
$ %
and denote by $\Lambda_{\Bbbk}^*$ its (algebraic) dual space. Using the obvious extension of the Hall inner product $\Lambda_\Bbbk\otimes_{\Bbbk}\Lambda_{\Bbbk}\to\Bbbk$ we naturally identify Schur functions with elements in the dual via $s_\lambda\mapsto\langle s_\lambda,\bullet\rangle$. N.B. under this identification the dual $\Lambda_{\Bbbk}^*$ allows for `infinite' linear combinations in Schur functions, while $\Lambda_{\Bbbk}$ only contains finite linear combinations.
\begin{definition}
    By a {\em $\Bbbk$-valued $\tau$-function} we mean an element $\langle \tau, \bullet \rangle$ in $\Lambda_{\Bbbk}^*$ such that the $\langle\tau,s_\lambda\rangle\in\Bbbk$ satisfy the Pl\"ucker relations \eqref{Plucker} for all $m,n\in\bN$ and $\lambda\in\Pi$, i.e. the map $d_\lambda\mapsto\langle\tau,s_\lambda\rangle$ defines a ring homomorphism $\bZ[\Gr]\to\Bbbk$.
\end{definition}
We note that in this formal setting we impose no convergence condition on the elements in the set $\{\langle\tau,s_\lambda\rangle~:~\lambda\in\Pi\}\subset\Bbbk$. Moreover, each ring homomorphism $f:\bZ[\Gr]\to\Bbbk$ defines an associated formal $\tau$-function by setting $\langle\tau^f,s_\lambda\rangle=f(d_\lambda)$ for all $\lambda\in\Pi$ and as the Schur functions $s_\lambda$ form a basis of $\Lambda_{\Bbbk}$ this fixes $\tau^f\in\Lambda^*_{\Bbbk}$ uniquely.

\begin{example}\rm
   The simplest examples of $\tau$-functions are the Schur functions: namely set $\tau=s_\lambda$ then one has $\langle\tau,s_\mu\rangle=\delta_{\lambda\mu}$, which trivially solves the Pl\"ucker relations \eqref{Plucker} as the latter are only containing quadratic terms $d_\alpha d_\beta$ with $\alpha\neq\beta$.  N.B. this solution of the Pl\"ucker relations corresponds to the ideal in $\bZ[\Gr]$ generated by $d_\lambda-1$ and $d_\mu$ for $\mu\neq\lambda$.
\end{example}

Employing the Hall inner product the Pl\"ucker relations \eqref{Plucker} can be written over the base $\Bbbk=\bQ$ as an infinite set of nonlinear partial differential equations -- which make up the KP-hierarchy -- using that,
\begin{equation}\label{s2d}
\langle\tau,s_\lambda\rangle=\langle s_\lambda^*\tau,1\rangle=s^*_\lambda(\tau)(0)=
s_{\lambda}(\partial)\tau|_{p_1=p_2=\cdots=0}\;.
\end{equation}

We now demonstrate on the simplest non-trivial example how to recast the Pl\"ucker relations \eqref{Plucker} as differential equations for the $\tau$-function and how these relations arise from the mutation at a vertex in the quiver $Q_\infty$ of the ind-cluster algebra $\bZ[\Gr]$.
\begin{example}\rm
    Fix $m=2$. Then \eqref{Plucker} reduces to the following three term relation,
    \[
    -d_{(\ldots,-3,0,1)}d_{(\ldots,-3,-2,-1)}+d_{(\ldots,-3,-1,1)}d_{(\ldots,-3,-2,0)}-d_{(\ldots,-3,-2,1)}d_{(\ldots,-3,-1,0)}=0\;.
    \]
    This is the only non-trivial Pl\"ucker relation coming from $\bZ[\Gr_{2,2}]$. Recall that each of the Maya sequences (of charge 0) appearing in the above equation encodes the outline of a Young diagram; see Remark \ref{rmk:Maya2Young}. Labelling the coefficients in terms of partitions instead we obtain 
    \begin{equation}\label{KPd}
        -d_{(2,2)}d_{\varnothing}+d_{(2,1)}d_{(1)}-d_{(1,1)}d_{(2)}=0\;.
    \end{equation}
    We recognise in this labelling that this relation is obtained by mutating at the vertex labelled by a single box in Figure \ref{fig:postnikov}. The resulting Postnikov diagram and quiver is shown in Figure \ref{fig:postnikov2KPsimple}. In terms of Schur functions this Pl\"ucker relation translates into 
\begin{equation}\label{KPs}
\langle\tau,1\rangle\langle\tau, s_{(2,2)}\rangle-\langle\tau,s_{(2,1)}\rangle\langle\tau, s_{(1)}\rangle+\langle\tau, s_{(1,1)}\rangle\langle\tau, s_{(2)}\rangle=0,
\end{equation}
which is the KP-equation. The latter can be expressed in its Hirota bilinear form by expanding each Schur function into power sums using \eqref{s2d}; see e.g. \cite[Example 10.1]{miwa2000solitons}. More generally, using the relation \eqref{hookPlucker} from Example \ref{ex:quad quiver}, we obtain (using Frobenius notation for the partitions) the following set of 3-term relations for solutions $\tau$ of the KP-hierarchy,
    \begin{equation}
    \langle\tau,1\rangle\langle\tau, s_{(a,a-1|b,b-1)}\rangle-\langle\tau,s_{(a|b)}\rangle\langle\tau, s_{(a-1|b-1)}\rangle+\langle\tau, s_{(a-1|b)}\rangle\langle\tau, s_{(a|b-1)}\rangle=0,
    \end{equation}
    which for $a=b=1$ reduces to the KP-equation above.
\end{example}

\subsection{The Sato-Segal-Wilson Grassmannian and $\tau$-functions}\label{S:tau functions}

We shall now specialise to the case $\Bbbk=\bC$ and consider the ring $\iring$. Recall Definition \ref{def:Gr} of the Grassmannian $\Gr(H)$ and of its connected component $\Gr$ of virtual dimension zero. For technical details we refer the reader to \cite[Section 8]{segal1985loop} and \cite[Chapter 7]{pressley1985loop}.

We recall that any subspace $W\in\Gr$ contains a sequence $\{w_i\}_{i\ge 1}$ called an {\em admissible basis} in \cite[\S3]{segal1985loop} and \cite[Def. 7.5.1]{pressley1985loop} such that 
\begin{itemize}
    \item the associated map $w:H_\varnothing\to H$ which maps $z^{-i+1}\mapsto w_i$ for all $i\ge 1$ has image $W$ and is continuous as well as injective;
    \item the matrix transforming $\{\mathrm{pr}(w_i)\}$ into $\{z^{i+1}\}$ differs from the identity by an operator who has a determinant (i.e. it differs from the identity by an operator of trace class).
\end{itemize}
In particular, one can choose for each $W\in\Gr$ a particular admissible basis making use of the following known stratification of $\Gr(H)$ in terms of Maya diagrams \cite[\S2,\S3]{segal1985loop} and \cite[\S7.3]{pressley1985loop}: for each $W\in\Gr$ there exists a unique (minimal) Maya diagram $a_\bullet$ of virtual cardinality zero (or equivalently a partition $\lambda$) such that the orthogonal projection $W\to H_{a_\bullet}$ is an isomorphism. Here $H_{a_\bullet}$ is the subspace spanned by $\{z^{a_i+1}\}_{i\ge 1}$. That is, $W$ contains elements $w_i$ of the form
\[
w_i=z^{a_i+1}+\sum_{n>i}w_{n,i}z^{n+1},\qquad w_{n,i}\in\bC,
\]
which span the dense subspace $W^{\rm fin}\subset W$ of finite order elements in $W$. For a given fixed Maya diagram $a_\bullet$ (or, equivalently, a partition $\lambda$) the set of all subspaces $W\in\Gr$ obeying the latter property form a submanifold, the stratum $\Gr_\lambda$. In particular, $W=H_\varnothing\in\Gr_\varnothing$ which explains our notation.
\begin{example}\rm
The last example generalises to all partitions $\lambda$: let $a_\bullet=a_\bullet(\lambda)$ be the unique Maya sequence of virtual cardinality zero corresponding to $\lambda$; c.f. Remark \ref{rmk:Maya2Young}. Then the subspace $H_\lambda=H_{a_\bullet}$ with admissible basis $\{z^{a_i+1}\}_{i\ge 1}$ lies in $\Gr_\lambda$. The subspaces $H_\lambda$ correspond to the simplest solutions of the KP-hierarchy $\tau=s_\lambda$ discussed above.
\end{example}

Using the concept of admissible bases, one can assign to each $W\in\Gr$ and partition $\lambda$ (or equivalently Maya diagram of charge 0) the following `infinite determinant' called `the Pl\"ucker coordinates of $W$' in \cite[\S8]{segal1985loop},
\begin{equation}\label{W2d}
\Delta_\lambda(W):=\det(w_{\lambda_i-i,j})_{i,j\ge 1}\in\bC\;.
\end{equation}
In fact, using the admissible basis above for $W\in\Gr_\lambda$, one can show that $\Delta_\lambda(W)$ reduces always to a finite determinant and is therefore well-defined; see \cite[\S8]{segal1985loop}. Moreover, analogous to the finite dimensional case the Pl\"ucker coordinates \eqref{W2d} define a projective embedding; c.f. \cite[\S10]{segal1985loop} and \cite[\S7.5]{pressley1985loop}.

The following summarises the key result of the connection between the infinite dimensional Grassmannian $\Gr$ and the ring $\iring$; we refer the reader to \cite{sato1983soliton} and \cite{segal1985loop} for details.

\begin{theorem}[Sato, Segal-Wilson, Pressley-Segal] For each $W\in\Gr$ the collection $\{\Delta_\lambda(W)\}_{\lambda\in\Pi}\subset\bC$ satisfies the Pl\"ucker relations \eqref{Plucker} and is square-summable, i.e. $\sum_{\lambda\in\Pi}|\Delta_\lambda(W)|^2<\infty$.
\end{theorem} 
The last theorem defines a holomorphic embedding $\Gr\to\Proj(\ell^2(\Pi))$, where $\ell^2(\Pi)$ denotes the Hilbert space of square summable sequences labelled by partitions $\lambda\in\Pi$; see \cite[Chapter 7, Prop. 7.5.2]{segal1985loop}. This map has been called the Pl\"ucker embedding in {\em loc. cit.} and should be thought of as a generalisation of the Pl\"ucker embedding from the finite to the infinite case.

\begin{remark}\rm
Recall the presentations from \cite[Prop.2.8 (ii) and Prop 2.10 (ii)]{FH-Sato} of the coordinate ring $\iring$. In {\em loc. cit.} the maximal ideals of $\iring$ are identified as points on Sato's `universal Grassmann manifold' $\widetilde{UGM}$; see \cite[Remark 2.11]{FH-Sato}. In this case no convergence conditions are imposed on the Pl\"ucker coordinates.   
\end{remark}

\begin{example}\rm
   The simplest subspaces $H_\lambda\in\Gr$ which contain the admissible basis $\{z^{\lambda_i-i}\}_{i\ge 1}$ have Pl\"ucker coordinates $\Delta_\mu(H_\lambda)=\delta_{\lambda\mu}$. They trivially solve the Pl\"ucker relations \eqref{Plucker}. More generally, we have for $W\in\Gr_\mu$ that $\Delta_\mu(W)\neq 0$ and $\Delta_\lambda(W)=0$ when $\lambda\subset\mu$; see e.g. \cite[Prop.7.5.4]{pressley1985loop}. Thus, in general there will be infinitely many nonzero Pl\"ucker coordinates for a generic point $W\in\Gr$. 
\end{example}

Instead of working with the (algebraic) dual space $\Lambda^*_\bC$, as in the previous section, we now consider the Hilbert space completion $\overline{\Lambda}_\bC$ of $\Lambda_\bC=\Lambda\otimes_\bZ\bC$ with respect to the Hall inner product $\Lambda_\bC\otimes_\bC\Lambda_\bC\to\bC$ which we assume to be anti-linear in the first factor. The Hilbert space completion allows for infinite linear combinations of Schur functions but it also imposes a convergence condition on the Pl\"ucker coordinates $\langle\tau,s_\lambda\rangle$ for any $\tau\in\overline{\Lambda}_\bC$: namely, they must be square summable and this condition then matches with the convergence condition imposed on the Pl\"ucker coordinates $\Delta_\lambda(W)$ of a point $W$ on the Sato-Segal-Wilson Grassmannian. The following definition of the Sato-Segal-Wilson $\tau$-function, a solution of the KP-hierarchy, is Proposition 8.3 in \cite{segal1985loop} (see also \cite{miwa2000solitons} and references therein)\footnote{We note that in \cite[Section 3]{segal1985loop} the $\tau$-function is defined in terms of sections of the dual determinant bundle instead and using explicit determinant formulae. Here we have used \cite[Prop. 8.3]{segal1985loop} for the definition instead, the expansion of a $\tau$-function into Schur functions, in order to make the connection with the Pl\"ucker coordinates explicit.}.
\begin{definition}
    Let $W\in\Gr$. Then the associated $\tau$-function $\tau^W\in\overline{\Lambda}_\bC$ is the series
    \begin{equation}\label{tau}
      \tau^W(p)=\sum_{\lambda\in\Pi}\Delta_\lambda(W)s_\lambda(p),
    \end{equation}
    where the Pl\"ucker coordinates $\Delta_\lambda(W)\in\bC$ of $W$ are the ones defined in \eqref{W2d} and the $s_\lambda(p)$ are the Schur functions \eqref{Schur}.
\end{definition}
This definition allows us to recall Sato's original observation \cite{sato1981soliton,sato1983soliton} of the relation between solutions of the KP hierarchy and infinite dimensional Grassmannians in the setting of the Sato-Segal-Wilson Grassmannian.
\begin{theorem}[Sato \cite{sato1981soliton,sato1983soliton}]
    Any element $\tau\in\overline{\Lambda}_\bC$ which satisfies the KP-hierarchy, i.e. the coefficients $\langle \tau, s_\lambda\rangle$ satisfy the Pl\"ucker relations \eqref{Plucker}, corresponds to a point of the Sato-Segal-Wilson Grassmannian $\Gr$. That is, there exists $W\in\Gr$ such that $\langle s^*_\lambda(\tau),1\rangle=\Delta_\lambda(W)$.
\end{theorem}

As another application of the ind-cluster algebra structure we are now considering positivity questions for the Sato-Segal-Wilson Grassmannian. Some of our motivation comes from combinatorics as well as the close connection between $\Gr$ and the representation theory of loop groups; see \cite{pressley1985loop}. 

The following closely resembles the definition of the {\em totally positive Grassmannian} in the finite case; see e.g. \cite[Def. 1.2.1]{fomin2024introduction} .

\begin{definition}
    Define the {\em totally positive Sato-Segal-Wilson Grassmannian} $\Gr^+$ to be the set of all points $W\in\Gr$ such that $\Delta_\lambda(W)>0$ for all $\lambda\in\Pi$.
\end{definition}

Recall that in algebraic combinatorics an important concept is that of Schur positivity: an element $f\in\Lambda$ in the ring of symmetric functions is called {\em Schur positive} if it has an expansion into Schur functions where each non-zero coefficient is positive, i.e. $\langle f,s_\lambda\rangle\ge 0$. We extend this notion to the Hilbert space completion $\overline{\Lambda}_{\bC}$ in the obvious manner.

\begin{corollary}
Let $\tau=\tau^W$ be a solution of the KP-hierarchy with $W\in\Gr_\varnothing$. Then $\tau^W$ is {\em totally Schur positive}, i.e. the Pl\"ucker coordinates $\Delta_\lambda(W)=\langle s_\lambda^\ast(\tau^W),1\rangle$ in the expansion \eqref{tau} are all positive, if and only if  $\Delta_\mu(W)>0$ for any rectangular partition $\mu$. In particular, in this case we have that $W\in\Gr^+$.
\end{corollary}

\begin{remark}\rm
    The previously known explicit descriptions of the Pl\"ucker coordinates $\Delta_\lambda(W)$ in the literature are in determinant form such as \eqref{Giambelli}; see e.g. \cite[Cor.2.1]{harnad2011schur} and also \cite[Thm 1.1]{nakayashiki2017expansion} for Pl\"ucker coordinates belonging to points $W\in\Gr_\lambda$ in strata with $\lambda\neq\varnothing$. These determinant formulae are derived using the Pl\"ucker embedding and Wick's Theorem for free fermions and are therefore not manifestly positive. They also do not allow one to address the question of algebraic independence of relations; c.f. Cor \ref{cor:indep}.
\end{remark}

The notion of positivity for $\tau$-functions considered here is different from the one considered in \cite{kodam2011kp,kodama2014kp} for (a certain class of) multi-soliton solutions of the KP-equation \eqref{KPs}, which corresponds to the simplest Pl\"ucker relation \eqref{KPd}. In the latter work solutions describing multi-solitons are described in terms of monomials of exponential functions in the flow parameters $t_1=p_1,t_2=p_2/2,t_3=p_3/3$ with positive coefficients (that come from the Pl\"ucker coordinates of finite Grassmannians). There also exist more general multi-soliton solutions which satisfy the entire KP-hierarchy (see e.g. \cite{miwa2000solitons} and \cite{kac2013bombay}), i.e. all of the Pl\"ucker relations \eqref{Plucker} and correspond to the $\tau$-functions considered here, but in general it is not expected that the expansions of these more general multi-soliton solutions (or the special class considered in \cite{kodamakodama2011kp}) into Schur functions gives {\em positive} or {\em non-negative} Pl\"ucker coordinates $\Delta_\lambda(W)$ for the corresponding point $W$ on the Sato-Segal-Wilson Grassmannian. To clarify this point, we briefly recall the multi-soliton solutions discussed in \cite{kodama2011kp,kodama2014kp}.

\subsection{Multi-Soliton Solutions and Finite Grassmannians}\label{S:multi-soliton}
There is a class of special multi-soliton solutions of the KP-equation, which are obtained from points on {\em finite} Grassmannians $\Gr_{m,n}$; see e.g. \cite{kodama_book}. This particular subset of solutions and their connection with the cluster algebra structure of the (finite) coordinate rings $\bC[\Gr_{m,n}]$ \cite{Scott} has been the subject of study in the works by Kodama and Williams \cite{kodama2011kp,kodama2014kp}. In order to avoid confusion and highlight the difference between the finite Pl\"ucker coordinates entering these special multi-soliton solutions as combinatorial data and their Pl\"ucker coordinates when viewed as points on the Sato-Segal-Wilson Grassmannian we briefly comment on these solutions; for details we refer to \cite{kodama2011kp,kodama2014kp} and \cite{kodama_book} as well as references therein. 

Fix some integer $N\ge 2$, some free parameters $\zeta=(\zeta_1,\ldots,\zeta_N)$ and a subspace $V\in\Gr_{m,n}$ with $m+n=N$. Then the special $N$-soliton solution associated to these data in \cite{kodama2011kp} is given by 
\begin{equation}\label{kodama_solitons}
\tau^{V,\zeta}(x,y,t)=\sum_{I\subset[N]}c_I(\zeta)\prod_{i\in I}e^{\zeta_ix+\zeta_i^2y+\zeta_i^3t},\quad
c_I(\zeta)=\Delta^{(m,n)}_I(V)\prod_{i,j\in I}(\zeta_i-\zeta_j)\,,
\end{equation}
where the sum runs over all $m$-subsets $I$ of $[N]=\{1,\ldots,N\}$ and $\Delta^{(m,n)}_I(V)$ are the Pl\"ucker coordinates of $V\in\Gr_{m,n}$. For generic choice of parameters it is not clear that these are solutions of the higher Pl\"ucker relations \eqref{Plucker} other than \eqref{KPd} and, thus, in general will not correspond to points on the Sato-Segal-Wilson Grassmannian (despite this we have labelled these solutions with the letter $\tau$ although they fall in general outside of our definition of a `$\tau$-function'). However, general multi-soliton solutions which satisfy all of the relations \eqref{Plucker} are known, see e.g. \cite{miwa2000solitons} and 
also \cite[Thm 2]{nakayashiki2018degeneration} for an extension of the soliton solutions from \cite{kodama2011kp} to the Sato Grassmannian as well as references therein. In order to obtain the Pl\"ucker coordinates of these multi-soliton $\tau$-functions when viewed as points $W\in\Gr$  on the Sato-Segal-Wilson Grassmannian, they need first to be rewritten in terms of Schur functions according to \eqref{tau}. This rewriting can be achieved by using the Cauchy identity (see e.g. \cite[Sec. 2.2]{kodama2021space} for such a discussion in the case of certain soliton solutions). 

From the expansion \eqref{tau} one sees that in general, i.e. for generic points $W\in\Gr$, one has infinitely many non-zero P\"ucker coordinates $\Delta_\lambda(W)$. The Pl\"ucker relations of the latter are described by the ind-cluster algebra $\iring$, while the finite number of relations satisfied by $\Delta^{(m,n)}_I(V)$ in the special solutions \eqref{kodama_solitons} of the KP-equation are described by the finite rank cluster algebra $\bC[\Gr_{m,n}]$. 

%%%%%%END%%%%%%%%%%%%%%%%%%%%%%%%%%%%%%%%%%%%
\bibliographystyle{alpha}

\bibliography{ind_cluster_bib}

\end{document}